\documentclass[ijoc]{informs4post}
\usepackage[table]{xcolor}
\definecolor{mygray}{rgb}{0.816, 0.808, 0.808}
\usepackage{eqndefns-left}
\RequirePackage{tgtermes}
\RequirePackage{newtxtext}
\RequirePackage{newtxmath}
\RequirePackage{bm}
\RequirePackage{endnotes}
\RequirePackage{enumitem}
\RequirePackage{subcaption}
\RequirePackage{booktabs}
\RequirePackage{threeparttable}
\RequirePackage{threeparttablex}
\RequirePackage{cases}
\RequirePackage{pifont}
\RequirePackage{arydshln}
\RequirePackage{xcolor}
\RequirePackage{colortbl}
\RequirePackage{array}
\RequirePackage{longtable}
\RequirePackage{multicol}
\RequirePackage{tabularx}
\RequirePackage{tabularray}
\RequirePackage{caption}
\RequirePackage{float}
\RequirePackage{wrapfig}
\RequirePackage{array}
\RequirePackage{longtable}
\RequirePackage{multirow}
\RequirePackage{tabularx}
\RequirePackage{tabularray}
\RequirePackage{booktabs}
\RequirePackage{threeparttable}
\RequirePackage{threeparttablex}
\RequirePackage{multirow}
\RequirePackage{multicol}
\RequirePackage{tabularx}
\RequirePackage{tabularray}
\RequirePackage{empheq}

\makeatletter
\renewcommand\AAnormalsizeXII{%
   \@setfontsize\normalsizeXI{12}{20}%
   \Dskips{1.0}{0.3}{0.5}{0.6}%
   \let\@listi\@listI
   \baselineskip20pt plus 1pt minus 1pt
\@bls=20pt\@em=12pt
\relax}
\makeatother
\OneAndAHalfSpacedXII
\usepackage{amsmath,amssymb,amsfonts}
\usepackage{algorithm}
\usepackage{algpseudocode}
\usepackage{tikz}
\usepackage{tikz}
\usepackage{mathtools}
\usepackage{multirow}
\usepackage{babel}
\usepackage{verbatim}
\usepackage{booktabs}       
\usepackage[colorlinks=blue,allcolors=blue]{hyperref}
\usepackage{ifthen}
\newboolean{IJOC}
\setboolean{IJOC}{true} 

\usepackage{natbib}
 \bibpunct[, ]{(}{)}{,}{a}{}{,}%
 \def\bibfont{\small}%
 \def\bibsep{\smallskipamount}%
\global\long\def\inner#1#2{\left\langle#1,#2\right\rangle}%

\global\long\def\norm#1{\left\|#1\right\|}%

\global\long\def\brbra#1{\left(#1\right)}%

\global\long\def\Brbra#1{\left(#1\right)}%

\global\long\def\abs#1{\left|#1\right|}%

\global\long\def\bcbra#1{\left\{#1\right\}}%

\global\long\def\vertiii#1{\left\vert \kern-0.25ex  \left\vert \kern-0.25ex  \left\vert #1\right\vert \kern-0.25ex  \right\vert \kern-0.25ex  \right\vert }%

\global\long\def\mcal#1{\mathcal{#1}}%

\global\long\def\mbb#1{\mathbb{#1}}%

\global\long\def\zerobf{\mathbf{0}}%

\global\long\def\proj{\operatorname{Proj}}%

\global\long\def\argmin{\operatorname{argmin}}%

\global\long\def\argmax{\operatornamewithlimits{argmax}}%

\global\long\def\st{\operatornamewithlimits{s.t.}}%

\global\long\def\diag{\mathrm{diag}}%

\global\long\def\and{\mathrm{and}}%

\global\long\def\vep{\varepsilon}%

\global\long\def\eps{\varepsilon}%

\global\long\def\err{\text{err}}%

\global\long\def\soc{\operatorname{soc}}%

\global\long\def\rsoc{\operatorname{rsoc}}%

\global\long\def\tsum{{\textstyle {\sum}}}%

\global\long\def\Kcal{\mathcal{K}}%

\global\long\def\abf{\mathbf{a}}%

\global\long\def\Nbf{\mathbf{N}}%

\global\long\def\bbf{\mathbf{b}}%

\global\long\def\cbf{\mathbf{c}}%

\global\long\def\sbf{\mathbf{s}}%

\global\long\def\gbf{\mathbf{g}}%

\global\long\def\lbf{\mathbf{l}}%

\global\long\def\ubf{\mathbf{u}}%

\global\long\def\vbf{\mathbf{v}}%

\global\long\def\wbf{\mathbf{w}}%

\global\long\def\xbf{\mathbf{x}}%

\global\long\def\ybf{\mathbf{y}}%

\global\long\def\zbf{\mathbf{z}}%

\global\long\def\dbf{\mathbf{d}}%

\global\long\def\hbf{\mathbf{h}}%

\global\long\def\Abf{\mathbf{A}}%

\global\long\def\Gbf{\mathbf{G}}%

\global\long\def\Ibf{\mathbf{I}}%

\global\long\def\Mbf{\mathbf{M}}%

\global\long\def\Dbf{\mathbf{D}}%

\global\long\def\vbfp{\vbf_{\text{p}}}%

\global\long\def\vbfd{\vbf_{\text{d}}}%

\global\long\def\dr{d_{\textrm{r}}}%
\global\long\def\ds{d_{\textrm{s}}}%
\global\long\def\dt{d_{\textrm{t}}}%

\newcommand{\qedtext}{{\unskip\nobreak\hfil\penalty50\hskip1em\hbox{}\nobreak\hfil$\square$\parfillskip=0pt\finalhyphendemerits=0\par}}

\floatstyle{ruled}
\newfloat{algorithm}{tbp}{loa}
\providecommand{\algorithmname}{Algorithm}
\floatname{algorithm}{\protect\algorithmname}

\EquationsNumberedThrough

\TheoremsNumberedThrough
\ECRepeatTheorems  %

\MANUSCRIPTNO{}

\makeatletter
\def\proof{\@ifnextchar[{\@proof}{\@proof[Proof]}}
\def\@proof[#1]{\trivlist
  \item[\hskip \labelsep {\itshape #1.}] \ignorespaces}
\makeatother

\let\oldsum\sum
\renewcommand{\sum}{\mathop{\textstyle\oldsum}\limits}

\usepackage{color-edits}
\addauthor{zx}{orange}

\begin{document}

\RUNAUTHOR{Lin et al.}

\RUNTITLE{A Practical GPU-Enhanced Matrix-Free Primal-Dual Method for Conic Programs}

\TITLE{A Practical GPU-Enhanced Matrix-Free Primal-Dual Method for Large-Scale Conic Programs}

\ARTICLEAUTHORS{%
\AUTHOR{Zhenwei Lin}
\AFF{School of Industrial Engineering, Purdue University, West Lafayette, IN 47907, USA. \EMAIL{lin2193@purdue.edu}}
\AUTHOR{Zikai Xiong}
\AFF{H. Milton Stewart School of Industrial and Systems Engineering, Georgia Institute of Technology, Atlanta, GA 30332, USA. \EMAIL{zxiong84@gatech.edu}}
\AUTHOR{Dongdong Ge}
\AFF{Antai School of Economics and Management, Shanghai Jiao Tong University, Shanghai, China. \EMAIL{ddge@sjtu.edu.cn}}
\AUTHOR{Yinyu Ye}
\AFF{Department of Management Science and Engineering and Institute of Computational \& Mathematical Engineering, Stanford University, Stanford, CA 94305, USA. \EMAIL{yyye@stanford.edu}}

}

\ABSTRACT{%
In this paper, we introduce a practical GPU-enhanced matrix-free first-order method for solving large-scale conic programming problems, which we refer to as PDCS, standing for the \textbf{P}rimal-\textbf{D}ual \textbf{C}onic Programming \textbf{S}olver. Problems that it solves include linear programs, second-order cone programs, convex quadratic programs, and exponential cone programs. The method avoids matrix factorizations and uses sparse matrix-vector multiplication as its core operation, which is memory-efficient and well suited for GPU acceleration. {It builds on a variant of the primal-dual hybrid gradient method (PDHG) called reflected restarted Halpern PDHG with averaging, which we prove converges on general conic programs, and further incorporates several practical enhancements.}
  Its GPU implementation, cuPDCS, uses customized computational schemes that assign cones of different types and sizes to different levels of the GPU architecture.  Numerical experiments demonstrate that cuPDCS is generally more efficient than state-of-the-art commercial solvers and other first-order methods on large-scale conic program applications, including Fisher market equilibrium problems, Lasso regression, and multi-period portfolio optimization. cuPDCS also exhibits better scalability, efficiency, and robustness compared to other first-order methods on the conic program benchmark dataset CBLIB. These advantages are more pronounced in large-scale settings.
}%


\maketitle

\section{Introduction}

Conic programming seeks to find an optimal solution that minimizes (or maximizes) a linear objective function while remaining within the feasible region defined by the intersection of an affine subspace and a convex cone. Many important real-world decision-making models can be formulated as conic programs. Typical examples are linear programs (LPs), second-order cone programs (SOCPs), semidefinite programs (SDPs), and exponential cone programs. Conic programs have extensive applications across numerous fields, including economics (see, e.g., \citet{greene2003econometric,shmyrev2009algorithm}), transportation (see, e.g., \citet{charnes1954stepping}), energy (see, e.g., \citet{chen2024exponential}), healthcare (see, e.g., \citet{mak2015appointment,bandi2019robust}), finance (see, e.g., \citet{rujeerapaiboon2016robust}), manufacturing (see, e.g., \citet{bowman1956production,hanssmann1960linear}), computer science (see, e.g., \citet{cormen2022introduction}), and medicine (see, e.g., \citet{wagner2004large}), among many others. 

Since the mid-20th century, the development of efficient methods for conic programs has been a central topic in the optimization community, with significant efforts on improving computational speed and scalability. Nearly all existing general-purpose solvers for conic programs are based on either the simplex method (mainly for LPs) or the barrier method (also known as the interior-point method, IPM). These methods are implemented in most commercial optimization solvers.

Despite their effectiveness for solving moderate-sized instances, both the simplex and barrier methods become impractical for large-scale conic programs. The primary bottleneck is their reliance on matrix factorizations, such as LU factorization for the simplex method and Cholesky factorization for the IPM, which are required to repeatedly solve linear systems at each iteration. The computational cost of matrix factorizations grows superlinearly with problem size, measured in terms of the number of decision variables, constraints, or nonzero entries in the problem data. 
This limitation arises from two main challenges. First, matrix factorizations are highly memory-intensive, and even sparse matrices can produce dense factors, making storage and computation prohibitive. Second, these factorization-based methods are inherently sequential, limiting their suitability for modern parallel and distributed computing architectures, such as graphics processing units (GPUs). Due to these challenges, early efforts to use GPUs in commercial solvers have largely been unsuccessful \citep{gurobi_parallel_distributed_optimization}. 

Meanwhile, outside of conic solver development, GPU-based parallel computing has emerged as a powerful tool for scaling up modern computational applications, in particular in the training of large-scale deep learning models. This contrast calls for alternative conic program solvers that better align with modern high-performance computing architectures.

To exploit GPUs and solve large-scale conic programs more efficiently, a matrix-free first-order method (FOM) is a promising choice.
A matrix-free FOM addresses the two challenges mentioned above by eliminating the need for matrix factorizations and instead relying on more computationally efficient operations such as matrix-vector products. These operations are well-suited for GPU acceleration. 

In the context of these developments, we develop a practical matrix-free first-order method for large-scale conic programs.
We call the method PDCS, standing for ``Primal-Dual Conic Programming Solver,'' because it is based on the primal-dual hybrid gradient method (PDHG). PDCS is designed to solve conic programs whose cone is a Cartesian product of zero cones, nonnegative cones, second-order cones, and exponential cones. 
We also develop a GPU-enhanced implementation of PDCS that is called cuPDCS. 
As one might expect, for small-scale problems, the IPMs in commercial solvers are often better than first-order methods (including PDCS) in efficiency and robustness. 
However, on the tested instances cuPDCS achieves the following advantages: 
\begin{enumerate}[noitemsep] 
    \item cuPDCS is generally more efficient than state-of-the-art commercial solvers and other existing first-order methods on large-scale conic program applications.
    \item cuPDCS exhibits better scalability, efficiency, and robustness compared to other first-order methods on the small-scale conic program benchmark dataset CBLIB.  
\end{enumerate} 
Furthermore, the above two advantages are more pronounced in large-scale settings.
The PDHG algorithm and the associated practical enhancements used in PDCS do not require any matrix factorizations. As a result, the computational bottleneck of PDCS is performing (sparse) matrix-vector products when computing gradients. {Because recent GPU hardware and software have made sparse matrix-vector multiplication very efficient, cuPDCS is up to an order of magnitude faster than its CPU implementation.}

Recently, PDHG has shown promising progress in solving large-scale linear programs and exploiting GPU acceleration. PDHG can be applied to solve the saddle-point formulation of LP \citep{applegate2023faster}. Several large-scale LP solvers have been developed based on PDHG in combination with effective heuristics. Notable implementations include PDLP for CPUs \citep{applegate2021practical} and cuPDLP \citep{lu2023cupdlp} along with its C-language version cuPDLP-C \citep{lu2023cupdlp-c} for GPUs. 
\citet{lu2023cupdlp,lu2023cupdlp-c} have shown that GPU-based implementations of PDHG already outperform classical methods implemented in commercial solvers on a significant number of problem instances. As a result, PDHG has been integrated as a new algorithm for LP in COPT \citep{ge2022cardinal}, Xpress \citep{xpressnwes}, and Gurobi \citep{gurobinews}. Additionally, it has been incorporated into Google OR-Tools \citep{applegate2021practical}, HiGHS \citep{huangfu2018parallelizing}, and NVIDIA cuOpt \citep{nvdianews}. 

For general conic programs, \citet{xiong2024role} have established that PDHG (with restarts) can also solve conic programs beyond LPs. {However, practical PDHG-based solvers had been developed mainly for LP and QP, and it was unknown whether GPUs could provide similar computational advantages for general-purpose conic programs.} 

{PDCS (and cuPDCS) combines techniques from PDHG-based LP solvers with new components for general conic programs and for GPUs. The adaptive step sizes, adaptive restarts, primal weight updates and diagonal rescaling of PDCS follow PDLP \citep{applegate2021practical}, and its reflected Halpern step follows r$^2$HPDHG \citep{lu2024restarted}, with the Halpern iteration also used in HPR-LP \citep{chen2024hpr} and cuPDLPx \citep{lu2025cupdlpx}. PDCS extends these techniques from LP to general conic programs, and the contributions of this paper are as follows. (a) \textbf{Convergence theory for the base algorithm.} The base algorithm of PDCS, reflected restarted Halpern PDHG with averaging, combines the Halpern restarts of \citet{lu2024restarted} with the average restarts of PDLP \citep{applegate2023faster}. We prove that it converges on general conic programs with an $O(\tfrac{1}{N})$ rate after $N$ iterations, and linearly under a local error bound without knowledge of the error-bound constant. PDCS adds adaptive parameter choices and diagonal rescaling to this algorithm.} (b) \textbf{Specialized projection methods.} PDHG requires frequent Euclidean projections onto the underlying cones. While efficient projection methods exist for many common cones, such as nonnegative cones, second-order cones, and exponential cones, projections onto rescaled cones (after applying diagonal rescaling) can become nontrivial. PDCS {reduces each of these projections to finding the root of a scalar function, which it computes by Newton's method with warm starts, safeguarded by bisection. Each evaluation of this function costs $O(n)$ for an $n$-dimensional second-order cone and $O(1)$ for an exponential cone, so the rescaled projections cost at most a logarithmic factor more than the projections onto the unrescaled cones.} (c) \textbf{Customized parallelism strategies.} Unlike LPs, which involve only nonnegative cones, the underlying cones in general conic programs are often more complex, consisting of Cartesian products of multiple cones of different types and sizes. {To exploit GPU parallelism both across and within these cones, cuPDCS assigns each cone to a thread, a warp, a thread block, or a grid according to its type and dimension (Section~\ref{sec:cupdcs}). The open-source solver cuPDCS can be called from JuMP and CVXPY.}

{The current implementation of PDCS supports nonnegative, second-order, and exponential cones. Extending PDCS to semidefinite programs requires an eigendecomposition in each projection and a specialized GPU implementation, which we leave for future work. For large-scale SDPs, \citet{han2024low} propose LoRADS, which avoids projections onto the semidefinite cone by combining ADMM with the Burer--Monteiro method, and \citet{han2024accelerating} develop its GPU implementation cuLoRADS.}

We conduct computational experiments to evaluate the efficiency, stability, and scalability of cuPDCS. On a benchmark dataset (CBLIB) comprising over $2{,}000$ small-scale conic optimization instances, our method solves the vast majority of problems, demonstrating better numerical stability compared with other first-order methods. To assess scalability, we conduct experiments on three families of large-scale conic program problems: Fisher market equilibrium, Lasso regression, and multi-period portfolio optimization. The experiments demonstrate that cuPDCS scales more effectively to large-scale instances. Across instances with different numbers of cones, cuPDCS produces high-quality solutions within reasonable computational time.

\subsection{Related literature}

\textbf{The primal-dual hybrid gradient method (PDHG).}
The primal-dual hybrid gradient (PDHG) method was introduced by \citet{esser2010general,pock2009algorithm} for solving general convex-concave saddle-point problems, of which the saddle-point formulation of conic programs is a special case. PDHG (with restarts) exhibits linear convergence on LPs, and its performance, including linear convergence rates and fast local convergence, has been extensively studied \citep{applegate2023faster,hinder2024worst,xiong2023computational,xiong2023relation,lu2024geometry,xiong2024accessible}.  \citet{lu2024restarted} introduce a Halpern restart scheme that achieves a complexity improvement of a constant order over the standard average-iteration restart scheme used in PDLP \citep{applegate2021practical}. {Very recently, \citet{ding2026local} prove local linear convergence of restarted Halpern PDHG on conic programs when the restart frequency is set from an error-bound constant.} \citet{xiong2024role} establish the convergence of restarted PDHG for general conic programs, showing that its performance is closely linked to the local geometry of sublevel sets near optimal solutions. \citet{xiong2025high} uses probabilistic analysis to establish a high-probability polynomial-time complexity result for the PDHG used on LPs. \citet{lu2023practical,huang2024restarted} use PDHG to solve large-scale convex quadratic programs, which can also be reformulated as conic programs.

\textbf{Other first-order methods for conic programs.}
Several other first-order methods have been developed for general conic programs. 
ABIP \citep{lin2021admm,deng2024enhanced} solves conic programs using an ADMM-based IPM applied to the homogeneous self-dual embedding of the conic program. 
SCS \citep{o2016conic,o2021operator} employs a similar ADMM-based approach to solve the homogeneous self-dual embedding. These ADMM-based methods require solving a linear equation of similar form in each iteration, which is typically handled either through matrix factorization or the conjugate gradient method.
However, matrix factorizations suffer from scalability limitations and are not well-suited for parallel computing on GPUs. The conjugate gradient method, while avoiding explicit factorizations, requires multiple matrix-vector products per iteration, making the choice of tolerance for solving these linear systems a critical heuristic. We compare PDCS with SCS in Section~\ref{sec:numerical}. 
Beyond general-purpose conic programs, ADMM has also been used to solve semidefinite programs and quadratic programs (which can be equivalently formulated as SOCPs); see \citet{kang2025local} and \citet{stellato2020osqp}.

\textbf{GPU acceleration for conic programs.}  
The simplex method and IPMs generally do not benefit much from GPU due to their reliance on solving linear systems \citep{swirydowicz2022linear,gurobi_parallel_distributed_optimization}.
Recently, NVIDIA released cuDSS, a GPU-accelerated direct sparse solver \citep{NVIDIA_cuDSS}. Based on cuDSS, \citet{CuClarabel} introduced CuClarabel, a GPU-accelerated IPM. The ADMM-based solver SCS \citep{o2016conic,o2021operator} also has a GPU version that uses the conjugate gradient method for solving linear systems. However, these algorithms still face scalability issues for large-scale problems, as demonstrated in our comparison in Section \ref{sec:numerical}.
When we were preparing the first draft of this manuscript, we became aware (by private communication) of a concurrent project by Haihao Lu, Zedong Peng, and Jinwen Yang that develops a GPU-implemented PDHG-based solver for large-scale SOCPs.

\subsection{Outline}

The paper is organized as follows. Section~\ref{sec:notations} introduces the notation used throughout the paper. {Section~\ref{sec:preliminaries} provides the standard-form conic program and the PDHG iteration. Section~\ref{sec:theoretical_foundation} presents reflected restarted Halpern PDHG with averaging and shows its convergence rate.}
{Section~\ref{sec:Practical-enhancement-technique} presents PDCS, which adds practical enhancements to this algorithm.
Section~\ref{sec:cupdcs} describes how cuPDCS maps the projections onto cones of different types and sizes to the GPU.} Section~\ref{sec:numerical} presents numerical experiments comparing PDCS against other first-order methods and commercial solvers, and Section~\ref{sec:conclusion} concludes the paper. The appendices are provided in the Online Appendix.

\subsection{\label{sec:notations}Notation}
Throughout the paper, let $[n]:=\{1,\ldots,n\}$ for an integer $n$. We use bold lowercase letters such as $\vbf$ for vectors and bold capital letters such as $\Abf$ for matrices, and subscripts for the coordinates of a vector unless stated otherwise. We write $\inner{\ubf}{\vbf}$ for the Euclidean inner product, $\|\cdot\|$ (also written $\|\cdot\|_2$) for the Euclidean norm of a vector and the spectral norm of a matrix, and $\|\cdot\|_1$ and $\|\cdot\|_\infty$ for the $\ell_1$- and $\ell_\infty$-norms of a vector. Denote the identity matrix as $\Ibf$ and the diagonal matrix with diagonal entries $x_1,\ldots,x_n$ as $\diag(x_1,\ldots,x_n)$. For a vector $\vbf$, $\vbf^+$ and $\vbf^-$ denote its coordinate-wise positive and negative parts, $\vbf_i^+=\max\{\vbf_i,0\}$ and $\vbf_i^-=-\min\{\vbf_i,0\}$. We write $\proj_{\mcal S}(\xbf)$ for the Euclidean projection of $\xbf$ onto a closed convex set $\mcal S$. For a cone $\mcal K$ and a positive diagonal matrix $\Dbf$, the rescaled cone is $\Dbf\mcal K:=\{\vbf:\Dbf^{-1}\vbf\in\mcal K\}$.

For the cones used in this paper, \(\zerobf^d\) denotes the zero cone, \(\mathbb{R}_+^d\) the nonnegative orthant, \({\mcal K}_{\text{soc}}^{d+1} := \left\{(t, \xbf) \mid \xbf \in \mathbb{R}^{d}, t \in \mathbb{R}, \|\xbf\| \leq t \right\}\) the second-order cone, ${\mcal K}_{\rsoc}^{d+2}:=\{(x,y,\zbf)\mid x,y\in\mbb R_{+},\zbf\in\mbb R^{d},\norm{\zbf}^{2}\leq2xy\}$ the rotated second-order cone, and $\mcal K_{\exp} :=\bcbra{(r,s,t)\in\mbb R^{3}\mid s>0,t\geq s\cdot\exp\brbra{\frac{r}{s}}}\cup\{(r,s,t)\in\mbb R^{3}\mid s=0,t\geq0,r\leq0\}$ the exponential cone. The dual cone of a cone $\mcal K\subseteq\mbb R^d$ is $\mcal K^*:=\{\ybf\in \mbb R^d:\inner{\ybf}{\xbf}\geq 0\text{ for all }\xbf\in \mcal K\}$.

\section{\label{sec:preliminaries}Preliminaries of Conic Optimization and PDHG} 

In this section, we present conic optimization problems (conic programs) in standard form, along with their dual problems and equivalent saddle-point formulations, and recall the iteration of the basic PDHG method for these problems.

We consider the following conic program:
\begin{equation}\label{eq:prob_sf}
    \min_{\xbf\in\mbb R^{n}}\ \inner{\cbf}{\xbf}\ \ \st\ \Abf\xbf=\bbf\ ,\ \xbf\in\mcal K\ ,
\end{equation}
where $\Abf\in \mbb R^{m\times n}$ is the constraint matrix and $\mcal K=\mcal K_1 \times \mcal K_2\times\ldots\times \mcal K_{l}$ is a Cartesian product of $l$ nonempty closed convex cones. These cones include the zero cone $\zerobf^d$, the nonnegative orthant $\mbb R_+^d$, the second-order cone $\mcal K_{\soc}^{d+1}$, the exponential cone $\mcal K_{\exp}$, the rotated second-order cone $\mcal K_{\rsoc}^{d+2}$ and the dual exponential cone $\mcal K_{\exp}^*$, and we call them \textit{disciplined cones}.

The dual problem of~\eqref{eq:prob_sf} is
\begin{equation}\label{eq:dual_sf}
    \max_{\ybf\in\mbb R^{m}}\ \inner{\bbf}{\ybf}\ \ \st\ \cbf-\Abf^{\top}\ybf\in\mcal K^{*}\ ,
\end{equation}
where the dual cone $\mcal K^*$ is correspondingly given by $\mcal K_1^* \times \mcal K_2^*\times\ldots\times \mcal K_{l}^*$.

The problem~\eqref{eq:prob_sf} has an equivalent primal-dual formulation, expressed as the following saddle-point problem of the Lagrangian $\mcal L(\xbf,\ybf)$:
\begin{equation}\label{eq:pd_sf}
    \min_{\xbf \in \mcal K}  \ \max_{\ybf \in \mbb R^m}\ \ \mcal L(\xbf,\ybf):=\inner{\cbf}{\xbf} -\inner{\ybf}{\Abf\xbf-\bbf}\ .
\end{equation}
A saddle point $(\xbf,\ybf)$ of~\eqref{eq:pd_sf} corresponds to a solution satisfying the Karush--Kuhn--Tucker conditions, which in turn corresponds to an optimal primal solution $\xbf$ for \eqref{eq:prob_sf} and an optimal dual solution $\ybf$ for \eqref{eq:dual_sf}. {We denote the set of saddle points by $\mcal Z^*$ and assume that it is nonempty. For $\zbf=(\xbf,\ybf)$ with $\xbf\in\mcal K$, the KKT mapping of~\eqref{eq:pd_sf} is
\[
\mcal F(\xbf,\ybf):=\begin{pmatrix}\cbf-\Abf^{\top}\ybf+N_{\mcal K}(\xbf)\\ \Abf\xbf-\bbf\end{pmatrix},
\]
where $N_{\mcal K}(\xbf):=\{-\sbf:\sbf\in\mcal K^*,\ \inner{\sbf}{\xbf}=0\}$ is the normal cone of $\mcal K$ at $\xbf$. A point $\zbf=(\xbf,\ybf)$ with $\xbf\in\mcal K$ is a saddle point if and only if $\zerobf\in\mcal F(\zbf)$, so we measure how far $\zbf$ is from optimality by the Euclidean distance $\operatorname{dist}(\zerobf,\mcal F(\zbf))$ from $\zerobf$ to the set $\mcal F(\zbf)$. This distance bounds the primal infeasibility $\|\Abf\xbf-\bbf\|$ and the dual infeasibility $\operatorname{dist}(\cbf-\Abf^{\top}\ybf,\mcal K^*)$. Multiplied by $\|\zbf\|$, it also bounds the duality gap $|\inner{\cbf}{\xbf}-\inner{\bbf}{\ybf}|$, since $\inner{\cbf}{\xbf}-\inner{\bbf}{\ybf}=\inner{\cbf-\Abf^{\top}\ybf-\sbf}{\xbf}+\inner{\ybf}{\Abf\xbf-\bbf}$ for every $\sbf\in\mcal K^*$ with $\inner{\sbf}{\xbf}=0$.}

\subsection{\label{subsec:PDHG_CP}PDHG for conic programs}
PDCS builds on PDHG. Let $\tau$ and $\sigma$ be the primal and dual step sizes, respectively. One iteration of PDHG for solving~\eqref{eq:pd_sf} at iterate $\zbf = (\xbf,\ybf)$ (denoted by $\texttt{OnePDHG}(\zbf)$) is given by:
\begin{equation}
    \hat{\zbf}=\texttt{OnePDHG}(\zbf):=\begin{cases}
    \ensuremath{\hat{\xbf}=\proj_{\mcal K}\bcbra{\xbf-\tau(\cbf-\Abf^{\top}\ybf)}}\\
    \ensuremath{\hat{\ybf}=\ybf+\sigma(\bbf-\Abf(2\hat{\xbf}-\xbf))}
    \end{cases}.\label{eq:PDHG_sf}
\end{equation}
{\citet{chambolle2011first} show that the iterates $\zbf^{k+1}=\texttt{OnePDHG}(\zbf^k)$ of PDHG converge to a saddle point of~\eqref{eq:pd_sf} provided that $\tau\sigma\|\Abf\|^2<1$, and \citet{lu2022infimal} show that after $N$ iterations the last iterate satisfies $\operatorname{dist}(\zerobf,\mcal F(\zbf^N))=O(\tfrac{1}{\sqrt{N}})$.}

The single PDHG iteration \eqref{eq:PDHG_sf} does not require any matrix factorization, and its core operations are sparse matrix-vector multiplications and projections onto $\mcal K$.
Matrix-vector multiplications are typically cheaper than matrix factorizations in both memory usage and computational complexity. 
Moreover, projections onto the Cartesian product of disciplined cones can be parallelized. Specifically, each component of the vector (corresponding to different disciplined cones) can be projected independently, i.e., 
\begin{equation}\label{eq:multi_cone_proj}
    \proj_{\mcal K_1\times\ldots \times \mcal K_l}(\xbf)=[\proj_{\mcal K_1}(\xbf_{[1]})^{\top},\ldots, \proj_{\mcal K_l}(\xbf_{[l]})^{\top}]^\top\ ,
\end{equation}
where $\xbf_{[i]}$ denotes the component of $\xbf$ corresponding to the  $i$-th disciplined cone $\mcal K_i$. 
Detailed discussions on projections onto disciplined cones can be found in~\citet{parikh2014proximal,friberg2023projection}. In particular, projections onto $\zerobf^d$, $\mbb R_{+}^{d}$, $\mcal K_{\soc}^{d+1}$ and $\mcal K_{\exp}$ are computationally straightforward. 
{Projection onto the semidefinite cone requires an eigendecomposition, and \citet{kang2025factorization} propose a matrix-free method for this projection that is much faster on GPUs.} 
{The rotated second-order cone $\mcal K_{\rsoc}^{d+2}$ is the image of a second-order cone under an invertible linear map that is not diagonal, so PDCS applies diagonal rescaling (Section~\ref{sec:Practical-enhancement-technique}) in the coordinates of this second-order cone.}
The projection onto $\mcal K_{\exp}^*$
closely resembles that of the exponential cone.

\begingroup
\section{Reflected Restarted Halpern PDHG with Averaging}\label{sec:theoretical_foundation}

In this section, we present the base algorithm of PDCS, reflected restarted Halpern PDHG with averaging. It combines restarting from the Halpern iterates \citep{lu2024restarted} and restarting from the average of the iterates \citep{applegate2023faster}, two schemes that have been studied for accelerating PDHG on linear programs. The algorithm runs the Halpern iteration and, at each restart, chooses between the current iterate and the average. We show that this combination converges on general conic programs, at a sublinear rate in general and at a linear rate under a local error bound.

The iterates of Algorithm~\ref{alg:theory} are denoted by $\zbf^{t,k}$, where $t$ counts the outer iterations and $k$ the iterations within the current outer iteration. In outer iteration $t$, the algorithm applies the Halpern iteration \citep{halpern1967fixed,lieder2021convergence} with anchor $\zbf^{t,0}$ to the reflected map $(1+\alpha)\texttt{OnePDHG}-\alpha\Ibf$ with $\alpha\in[0,1]$, which has the same fixed points as \texttt{OnePDHG} (Line~\ref{alg_theory:halpern}), and maintains the average of the iterates (Line~\ref{alg_theory:average}). It restarts as soon as one of the two candidates has reduced the fixed-point residual $\|\zbf-\texttt{OnePDHG}(\zbf)\|_{\Mbf}$ of the anchor by the factor $\beta$, where $\|\zbf\|_{\Mbf}:=\sqrt{\inner{\zbf}{\Mbf\zbf}}$ and $\Mbf:=\big(\begin{smallmatrix}\tau^{-1}\Ibf & \Abf^{\top}\\ \Abf & \sigma^{-1}\Ibf\end{smallmatrix}\big)$ is positive definite when $\tau\sigma\|\Abf\|^2<1$.

\begin{algorithm}[H]
\caption{{Reflected restarted Halpern PDHG with averaging}\label{alg:theory}}
\small\SingleSpacedXI
    \begin{algorithmic}[1]
        \Require{Initial iterate $\zbf^{0,0}$, step sizes $\tau,\sigma>0$ with $\tau\sigma\|\Abf\|^2<1$, reflection parameter $\alpha\in[0,1]$, restart parameter $\beta\in(0,1)$, $t\leftarrow 0$}
    \Repeat
    \State{$k\leftarrow 0$, $\zbf_{c}^{t,0}\leftarrow\zbf^{t,0}$}
    \While{$\|\zbf_{c}^{t,k}-\texttt{OnePDHG}(\zbf_{c}^{t,k})\|_{\Mbf}>\beta\,\|\zbf^{t,0}-\texttt{OnePDHG}(\zbf^{t,0})\|_{\Mbf}$}
        \State{$\hat{\zbf}^{t,k+1}\leftarrow\texttt{OnePDHG}(\zbf^{t,k})$\label{alg_theory:pdhg}}
        \State{$\zbf^{t,k+1}\leftarrow\tfrac{k+1}{k+2}\big((1+\alpha)\hat{\zbf}^{t,k+1}-\alpha\zbf^{t,k}\big)+\tfrac{1}{k+2}\zbf^{t,0}$\label{alg_theory:halpern}}
        \State{$\bar{\zbf}^{t,k+1}\leftarrow\tfrac{1}{k+1}\sum\nolimits_{i=1}^{k+1}\zbf^{t,i}$\label{alg_theory:average}}
        \State{$\zbf_{c}^{t,k+1}\leftarrow$ whichever of $\zbf^{t,k+1}$ and $\bar{\zbf}^{t,k+1}$ has the smaller fixed-point residual\label{alg_theory:candidate}}
        \State{$k\leftarrow k+1$}
    \EndWhile
    \State{$\zbf^{t+1,0}\leftarrow\zbf_{c}^{t,k}$, $t\leftarrow t+1$}
    \Until{termination criteria hold}
    \State{\textbf{Return} $\texttt{OnePDHG}(\zbf^{t,0})$.}
    \end{algorithmic}
\end{algorithm}

\begin{theorem}\label{thm:theory}
Let $\{\zbf^{t,k}\}$ be generated by Algorithm~\ref{alg:theory}, and let $N$ be the total number of iterations performed up to $\zbf^{t,k}$, counted over all outer iterations.
\begin{itemize}
\item[(a)] The iterates $\zbf^{t,k}$ converge to a saddle point $\zbf^*\in\mcal Z^*$, and
\begin{equation}\label{eq:theory_sublinear}
\operatorname{dist}\big(\zerobf,\mcal F(\textup{\texttt{OnePDHG}}(\zbf^{t,k}))\big)\le\frac{2(2-\beta)}{(1+\alpha)(1-\beta)}\cdot\frac{\|\Mbf\|\operatorname{dist}(\zbf^{0,0},\mcal Z^*)}{N+1}.
\end{equation}
\item[(b)] Suppose further that there exist a neighborhood $\mcal U$ of $\zbf^*$ in $\mcal K\times\mbb R^m$ and a constant $\gamma>0$ such that the local error bound
\begin{equation}\label{eq:theory_error_bound}
\gamma\operatorname{dist}(\zbf,\mcal Z^*)\le\operatorname{dist}(\zerobf,\mcal F(\zbf))
\end{equation}
holds for all $\zbf\in\mcal U$. By part (a), there is $t_0$ such that $\textup{\texttt{OnePDHG}}(\zbf^{t,k})\in\mcal U$ for all $t\ge t_0$ and all $k$. Let $N_{t_0}$ be the number of iterations performed before outer iteration $t_0$. Then, for every $\epsilon>0$, every iterate $\zbf^{t,k}$ with
\begin{equation}\label{eq:theory_linear}
N\ge N_{t_0}+c\,\Big(1+\frac{\|\Mbf\|}{\gamma}\Big)\log\Big(\frac{\|\Mbf\|\operatorname{dist}(\zbf^{0,0},\mcal Z^*)}{\beta\,\epsilon}\Big),\qquad\text{where }c:=\frac{2}{(1+\alpha)\beta\log(1/\beta)},
\end{equation}
satisfies $\operatorname{dist}(\zerobf,\mcal F(\textup{\texttt{OnePDHG}}(\zbf^{t,k})))\le\epsilon$.
\end{itemize}
\end{theorem}
The proof of Theorem~\ref{thm:theory} is given in Appendix~\ref{app:theory_proof}.

Algorithm~\ref{alg:theory} thus reduces $\operatorname{dist}(\zerobf,\mcal F(\texttt{OnePDHG}(\cdot)))$ at the rate $O(\tfrac{1}{N})$, whereas the last iterate of PDHG attains the rate $O(\tfrac{1}{\sqrt{N}})$ \citep{lu2022infimal}. Since each iteration of Algorithm~\ref{alg:theory} evaluates \texttt{OnePDHG} at most twice (the evaluation at $\zbf^{t,k+1}$ in Line~\ref{alg_theory:candidate} is reused in Line~\ref{alg_theory:pdhg} of the next iteration), the same rate holds in terms of the number of evaluations of \texttt{OnePDHG}. Theorem~\ref{thm:theory} holds for every fixed $\alpha\in[0,1]$. Including the averaged candidate does not weaken these guarantees, because Line~\ref{alg_theory:candidate} restarts from whichever candidate has the smaller residual, so the residual at each restart is at most that of the Halpern iterate. The ablation study in Section~\ref{sec:cblib} shows that the averaged candidate also helps in practice.

Furthermore, under the local error bound~\eqref{eq:theory_error_bound}, the restarts make the convergence linear. By~\eqref{eq:theory_linear}, after the first $N_{t_0}$ iterations, the number of iterations required to reach accuracy $\epsilon$ grows linearly in $\tfrac{1}{\gamma}$ and in $\log(1/\epsilon)$, and Algorithm~\ref{alg:theory} does not need to know $\gamma$. For instance, $c<3$ when $\alpha=1$ and $\beta=0.4$, the value of $\beta_{\mathrm{sufficient}}$ used by PDCS (Appendix~\ref{sec:adaptive_restart}). Condition~\eqref{eq:theory_error_bound} is a property of the problem instance alone and does not depend on the algorithm. For LP, Hoffman's bound \citep{hoffman1952approximate} implies that~\eqref{eq:theory_error_bound} holds on every bounded set, and since the points $\texttt{OnePDHG}(\zbf^{t,k})$ remain in a bounded set,~\eqref{eq:theory_linear} holds with $t_0=0$.
For general conic programs, \citet{ding2026local} show that~\eqref{eq:theory_error_bound}, which they call metric subregularity, holds near saddle points that satisfy strict complementarity and a quadratic facial-violation condition, and they verify this condition for the nonnegative orthant, second-order cones, semidefinite cones and the relevant faces of exponential cones.

\endgroup

\section{A Practical Matrix-Free Primal-Dual Method for Conic Programs} \label{sec:Practical-enhancement-technique}
{This section presents PDCS, which adds practical enhancements to Algorithm~\ref{alg:theory}. One of them is to work directly with the following general form:}
\begin{equation}\label{eq:prob_cp}
\small
    \min_{\substack{\xbf=(\xbf_{1},\xbf_{2}):  \xbf_{1}\in\mbb R^{n_{1}},\xbf_{2}\in\mbb R^{n_{2}}}}\inner{\cbf}{\xbf}\ \ \st\ \mathbf{G}\xbf-\mathbf{h}\in\mcal K_{\textrm{d}}^{*}\ ,\ \mathbf{l}\leq\xbf_{1}\leq\mathbf{u}\ ,\ \xbf_{2}\in\mcal K_{\textrm{p}}\ ,
\end{equation}
where $\mathbf{G}\in \mbb R^{m\times n}$ is the constraint matrix, and $\lbf$ and $\ubf$ are elementwise lower and upper bounds on $\xbf_1$. The set $\mcal K_{\textrm{p}}$ denotes the primal cone, which is a Cartesian product of disciplined cones $\mcal K_1 \times \mcal K_2\times\ldots\times \mcal K_{l_{\textrm{p}}}$. Here, $\mcal K_{\mathrm d}^*$ denotes the dual cone of $\mcal K_{\mathrm d}$. We use $\mcal K_{\mathrm d}^*$ in the primal so that the dual feasibility condition can be written using $\mcal K_{\mathrm d}$ as follows:
\begin{equation}\label{eq:dual_conic}
    \begin{aligned}
\max_{\substack{\ybf,\  \bm{\lambda} = (\bm{\lambda}_{1},\bm{\lambda}_{2}):
\\
\ybf\in \mbb R^m, \ \bm{\lambda}_{1}\in\mathbb{R}^{n_1}, \ \bm{\lambda}_{2}\in\mathbb{R}^{n_2}}}  \ \langle \ybf, \mathbf{h}\rangle + \lbf^{\top}\bm{\lambda}_{1}^{+}-\ubf^{\top}\bm{\lambda}_{1}^{-}\ ,\
 \quad \st  \cbf-\Gbf^{\top}\ybf=\bm{\lambda}\ ,\ \ybf\in\mcal K_{\textrm{d}}, \ \bm{\lambda}_{1}\in\Lambda\ ,\ \bm{\lambda}_{2}\in\mcal K_{\textrm{p}}^{*}\ ,
    \end{aligned}
\end{equation}
 \begin{equation}\label{eq:defn_Lambda1}
 \text{where}\,\,
    \Lambda := \Lambda_1\times \Lambda_2 \times \cdots \times \Lambda_{n_1}
    \ \ \text{ and } \ \
    \Lambda_i :=\left\{\begin{array}{ll}
        \{0\}, & \text{if }\lbf_{i}=-\infty,\ubf_{i}=+\infty \\
        \mbb R_{-}, & \text{if }\lbf_{i}=-\infty,\ubf_{i}\in\mbb R \\
        \mbb R_{+}, & \text{if }\lbf_{i}\in\mbb R,\ubf_{i}=+\infty \\
        \mbb R, & \text{otherwise}
        \end{array}\right. \text{ for }i=1,2,\dots,n_1 \ .
\end{equation}
Similarly, $\mcal{K}_\textrm{d}=\mcal K_1 \times \mcal K_2\times\ldots\times \mcal K_{l_{\textrm{d}}}$ is a Cartesian product of disciplined cones, and its dual cone is $\mcal{K}_{\textrm{d}}^*=\mcal K_1^* \times \mcal K_2^*\times\ldots\times \mcal K_{l_{\textrm{d}}}^*$.

PDCS applies PDHG to the saddle-point formulation of~\eqref{eq:prob_cp}, $\min_{\xbf\in\mcal X}\max_{\ybf\in\mcal Y}\mcal L(\xbf,\ybf)$, where $\mcal L(\xbf,\ybf):=\inner{\cbf}{\xbf}-\inner{\ybf}{\Gbf\xbf-\hbf}$, $\mcal X:=[\lbf,\ubf]\times\mcal K_{\textrm{p}}$ and $\mcal Y:=\mcal K_{\textrm{d}}$. One iteration of PDHG at iterate $\zbf=(\xbf,\ybf)$ is
\begin{equation}
    \hat{\zbf}=\texttt{OnePDHG}(\zbf):=\begin{cases}
    \ensuremath{\hat{\xbf}=\proj_{[\lbf,\ubf]\times\mcal K_{\textrm{p}}}\bcbra{\xbf-\tau(\cbf-\Gbf^{\top}\ybf)}}\\
    \ensuremath{\hat{\ybf}=\proj_{\mcal K_{\textrm{d}}}\bcbra{\ybf+\sigma(\hbf-\Gbf(2\hat{\xbf}-\xbf))}}
    \end{cases}.\label{eq:PDHG}
\end{equation}
{The standard form~\eqref{eq:prob_sf} is a special case of~\eqref{eq:prob_cp}, for which~\eqref{eq:PDHG} reduces to~\eqref{eq:PDHG_sf}.} Like~\eqref{eq:PDHG_sf}, the iteration~\eqref{eq:PDHG} is matrix-free. Besides matrix-vector multiplications, it requires only projections onto the box $[\lbf,\ubf]$ and onto the disciplined cones in $\mcal K_{\textrm p}$ and $\mcal K_{\textrm d}$, which are computed as described in Section~\ref{subsec:PDHG_CP}.

{Built upon Algorithm~\ref{alg:theory}, PDCS adds adaptive step sizes, adaptive reflection parameters, adaptive restarts and primal weight updates, as summarized in Algorithm~\ref{alg:PDCS}.}
\begin{algorithm}[htbp]
\caption{PDCS: Primal-Dual Conic Programming Solver\label{alg:PDCS}}
\small\SingleSpacedXI
    \begin{algorithmic}[1]
        \Require{Initial iterate $\bar \zbf^{0,0} = \zbf^{0,0}=(\xbf^{0,0},\ybf^{0,0})$, initial step size $\hat{\eta}^{0,0}$, initial primal weight $\omega^0$, $t\leftarrow 0$} 
    \Repeat
    \State{$k\leftarrow 0$}
    \While{the restart condition is not triggered}
        \State{{$\hat{\zbf}^{t,k+1}, \eta^{t,k+1}, \hat{\eta}^{t,k+1}\leftarrow$ \texttt{AdaptiveStepPDHG}($\zbf^{t,k}, \omega^t, \hat{\eta}^{t,k}$)}\label{alg_pdcs:line4} }
        \State{$\alpha^{t,k}\leftarrow$ \texttt{AdaptiveReflectionParameter}($\hat{\zbf}^{t,k+1}$)\label{alg_pdcs:line5}} 
        \State{${\zbf}^{t,k+1}\leftarrow 
        \texttt{ReflectedHalpern}(\hat{\zbf}^{t,k+1},\zbf^{t,k},{\zbf}^{t,0},\alpha^{t,k})${\small$:=\tfrac{k+1}{k+2}((1+\alpha^{t,k})\hat{\zbf}^{t,k+1}-\alpha^{t,k}\zbf^{t,k}) + \frac{1}{k+2}\zbf^{t,0}$}\label{alg:line:ref_Halpern}}
        \State{$\bar{\zbf}^{t,k+1}\leftarrow \tsum_{i=1}^{k+1}\eta^{t,i}\zbf^{t,i}/{\tsum_{i=1}^{k+1}\eta^{t,i}}$\label{alg:line:average}}
        \State{$\zbf_{c}^{t,k+1}\leftarrow\texttt{GetRestartCandidate}(\zbf^{t,k+1},\bar{\zbf}^{t,k+1})$\label{alg:line:GetRestartCandidate}}
        \State{$k\leftarrow k+1$}
    \EndWhile
    \State{$\zbf^{t+1,0}\leftarrow \zbf_{c}^{t,k}$, $\omega^{t+1}\leftarrow \texttt{PrimalWeightUpdate}(\zbf^{t+1,0}, \zbf^{t,0},\omega^{t})$, {$\hat{\eta}^{t+1,0}\leftarrow\hat{\eta}^{t,k}$}, $t\leftarrow t+1$}
    \Until{termination criteria hold}
    \State{\textbf{Return} $\zbf^{t,0}$.}
    \end{algorithmic}
\end{algorithm}  

{Algorithm~\ref{alg:PDCS} has the same structure as Algorithm~\ref{alg:theory} and differs from it in the following components, whose details are given in Appendix~\ref{app:alg_details}. The function \texttt{AdaptiveStepPDHG} in Line~\ref{alg_pdcs:line4} (Function~\ref{alg:adaptive_pdhg} in Appendix~\ref{sec:adaptive_stepsize}) chooses the step size adaptively, instead of using the conservative step sizes required by the theory. Lines~\ref{alg_pdcs:line5} and~\ref{alg:line:ref_Halpern} implement an adaptive variant of the reflected Halpern step in Line~\ref{alg_theory:halpern} of Algorithm~\ref{alg:theory}, in which the function \texttt{AdaptiveReflectionParameter} allows a more aggressive reflection parameter $\alpha^{t,k}$ as the current solution becomes more accurate. The function \texttt{GetRestartCandidate} in Line~\ref{alg:line:GetRestartCandidate} and the restart condition use the normalized duality gap (or a KKT error), following the adaptive restart scheme of PDLP \citep{applegate2021practical}, instead of the fixed-point residual of Algorithm~\ref{alg:theory}. Finally, the function \texttt{PrimalWeightUpdate} adjusts the primal weight $\omega$ at each restart from the progress of the primal and dual iterates, which rebalances the primal and dual step sizes.}

{In practice, PDCS preconditions the problem by diagonal rescaling. It applies Algorithm~\ref{alg:PDCS} to an equivalent rescaled instance of~\eqref{eq:prob_cp} and maps the solution back (Appendix~\ref{subsec:Projection-onto-convex}). The rescaled instance has the same form as~\eqref{eq:prob_cp}, except that each disciplined cone $\mcal K$ is replaced by the rescaled cone $\Dbf\mcal K$ for a positive diagonal matrix $\Dbf=\diag(d_1,\dots,d_n)$. Algorithm~\ref{alg:PDCS} therefore runs on it unchanged, but the projections in~\eqref{eq:PDHG} are onto rescaled cones.}
For zero and nonnegative cones this is trivial, and for second-order cones and exponential cones we use the following characterizations to compute the projections.

\begin{theorem}\textbf{\emph{Projection onto $\Dbf\mcal K_{\soc}^{n+1}$.}}\label{thm:soc_proj}
The projection of $(t,\xbf)$ onto the rescaled cone $\Dbf\mcal K_{\soc}^{n+1}$ is given as follows. For $\Dbf$, let $\hat{\Dbf}$ denote $\diag(\hat{d}_2,\ldots,\hat{d}_{n+1})$ with $\hat{d}_i = d_{i}/d_1,\forall i\geq 2$.
If $t \le 0$ and $\|\hat{\Dbf}\xbf\| \le -t$, then $\proj_{\Dbf \mcal K_{\soc}^{n+1}}\bcbra{(t,\xbf)} = (0,\zerobf)$.
If $\norm{\hat{\Dbf}^{-1}\xbf}\leq t$, then $\proj_{\Dbf \mcal K_{\soc}^{n+1}}\bcbra{(t,\xbf)}=(t,\xbf)$. 
If $t = 0$ and $\xbf \neq \zerobf$, then $\proj_{\Dbf \mcal K_{\soc}^{n+1}}\bcbra{(t,\xbf)} = \left(\|(\hat{\Dbf} + \hat{\Dbf}^{-1})^{-1}\xbf\|, (\Ibf_{n\times n} + \hat{\Dbf}^{-2})^{-1}\xbf \right)$.
Otherwise, if $\norm{\hat{\Dbf}^{-1}\xbf}> t$, then let $\lambda > 0$ be the solution of the following univariate equation of $\lambda$: 
\begin{equation}\label{eq:bisection-soc}
\sum_{i=1}^{n}\Brbra{\frac{\hat{d}_{i+1}^{-1}}{1+2\lambda\hat{d}_{i+1}^{-2}}\xbf_{i}}^{2}-\frac{t^{2}}{(1-2\lambda)^{2}}=0\, , \ \frac{t}{(1-2\lambda)}>0.
\end{equation}
Then the projection is given by $\proj_{\Dbf \mcal K_{\soc}^{n+1}}\bcbra{(t,\xbf)}=\brbra{(1-2\lambda)^{-1}t,(\Ibf+2\lambda\hat{\Dbf}^{-2})^{-1}\xbf}$.
\end{theorem}

{In the fourth case, the root of~\eqref{eq:bisection-soc} lies in $(0,\tfrac{1}{2})$ if $t>0$ and in $(\tfrac{1}{2},\infty)$ if $t<0$. Proposition~\ref{prop:soc_bisection} in Appendix~\ref{app:proj:soc} rewrites~\eqref{eq:bisection-soc} as an equation in $u\in(0,1)$ whose left-hand side is strictly monotone, so the root is unique and bisection finds it to accuracy $\epsilon$ in $O(n\log(1/\epsilon))$ operations.} The proof of Theorem~\ref{thm:soc_proj} is given in Appendix~\ref{app:proj:soc}.

The projection onto the rescaled exponential cone $\Dbf\mcal K_{\exp}$ can also be computed by case analysis. See an informal statement of the result below. 
A full statement and proof for the diagonally scaled exponential cone projection are provided in Appendix~\ref{app:proj:exp}.

\begin{theorem}\textbf{\emph{Projection onto $\Dbf\mcal K_{\exp}$ (Informal)}}.\label{thm:exp_proj_simplified}
Let $\Dbf=\diag(d_r,d_s,d_t)\succ 0$. For any $v_0=(r_0,s_0,t_0)$, the projection
$\proj_{\Dbf\mathcal K_{\exp}}(v_0)$ is given as follows. If $v_0$ is already in $\Dbf\mathcal K_{\exp}$, then $\proj_{\Dbf\mathcal K_{\exp}}(v_0)=v_0$. If $v_0$ is in the polar cone $-\Dbf^{-1}\mathcal K_{\exp}^*$, then $\proj_{\Dbf\mathcal K_{\exp}}(v_0)=\zerobf$. If $r_0\le 0$ and $s_0\le 0$, then $\proj_{\Dbf\mathcal K_{\exp}}(v_0)=(r_0,0,t_0^+)$. Otherwise, $\proj_{\Dbf\mathcal K_{\exp}}(v_0)$ is determined by a scalar $\rho$, which is {either explicit or} the root of a continuous univariate function
$h(\rho)$ on an explicitly computable interval. Once $\rho$ is obtained (via solving the root-finding problem), $\proj_{\Dbf\mathcal K_{\exp}}(v_0)$ is given by closed-form expressions in $\rho$.   
\end{theorem} 

{Both theorems reduce the projection to finding the root of a scalar function. PDCS computes this root by a warm-started Newton method, with bisection as a safeguard.}

{Most of the enhancements in PDCS, including the diagonal rescaling, are common in PDHG-based LP solvers \citep{applegate2021practical,lu2024restarted,lu2023cupdlp}, and all of them preserve the matrix-free structure of PDHG. With the step size, the reflection parameter and the primal weight fixed and restarts decided by the fixed-point residual, Algorithm~\ref{alg:PDCS} is Algorithm~\ref{alg:theory} applied to the saddle-point formulation of~\eqref{eq:prob_cp}, or of its rescaled instance. Theorem~\ref{thm:theory} then applies with exact projections, with $\Abf$ replaced by $\Gbf$ in $\Mbf$ and $\mcal F$ replaced by the KKT mapping of this formulation, because its proof uses only the properties of \texttt{OnePDHG} in Lemma~\ref{lem:theory_fixed_point} in Appendix~\ref{app:theory}, which carry over to~\eqref{eq:PDHG}. The adaptive components are heuristics, and the ablation study in Section~\ref{sec:cblib} shows how much they contribute in practice. The technical note \citep{alg_detail} documents the implementation and usage of cuPDCS.}

\begin{figure}[htbp]
\begin{centering}
\begin{subfigure}[t]{0.48\columnwidth}%
\includegraphics[width=0.95\textwidth]{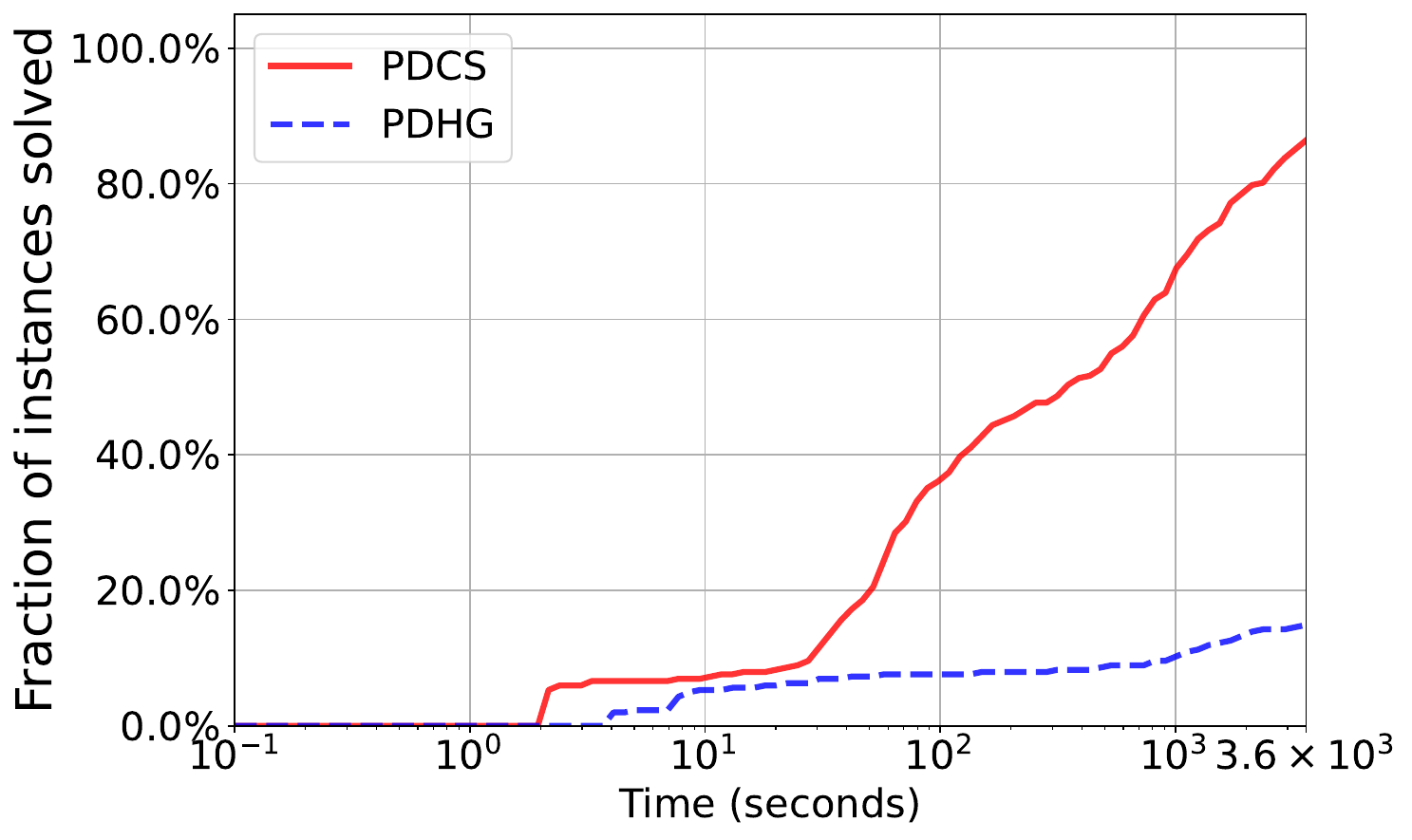}%
\caption{\label{fig:solved_time_soc}\SingleSpacedXI\small  Problems with second-order cone constraints}
\end{subfigure}
\hfill
\begin{subfigure}[t]{0.48\columnwidth}%
\includegraphics[width=0.95\textwidth]{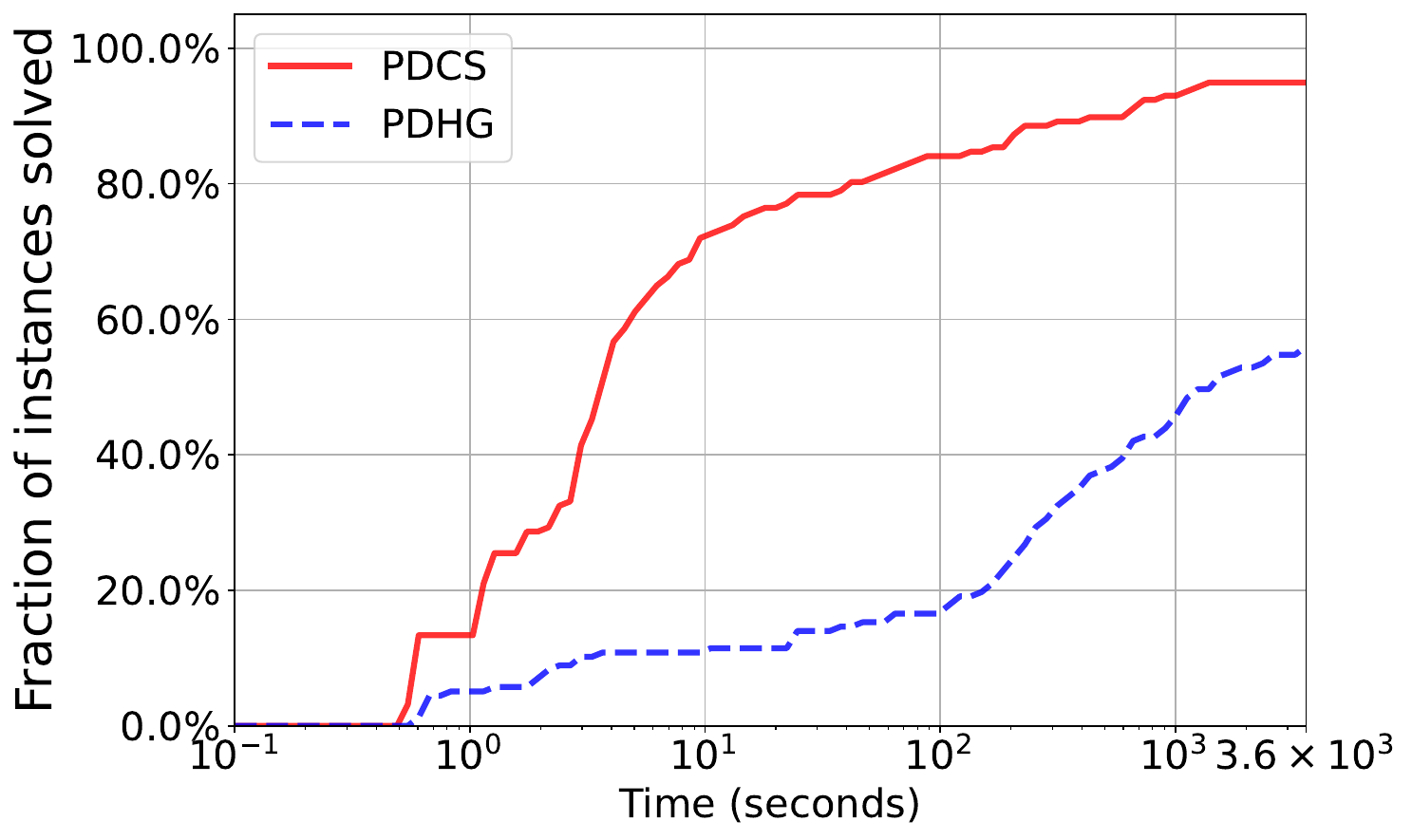}%
\caption{\label{fig:solved_time_exp}\SingleSpacedXI\small  Problems with exponential and second-order cone constraints}
\end{subfigure}
\end{centering}
\caption{\label{fig:solved_count_time}Performance of PDCS and PDHG in terms of the running time.}
\vspace{-10pt}
\end{figure}

We compare the standard PDHG method with PDCS. The performance of both methods is evaluated on two subsets of the \texttt{CBLIB} dataset. We say an instance is solved if the three relative errors in~\eqref{eq:PDCS_terminate_tolerance} are below $10^{-6}$.
Figure~\ref{fig:solved_count_time} plots the fraction of instances solved versus wall-clock time.
Figure~\ref{fig:solved_time_soc} shows the results on the dataset where all constraints are second-order cones, while Figure~\ref{fig:solved_time_exp} shows results on the dataset with exponential and second-order cone constraints. On both datasets, PDCS significantly outperforms PDHG, solving substantially more instances within one hour.

\section{\label{sec:cupdcs}cuPDCS: A PDCS Implementation with GPU Enhancements}
\begingroup

cuPDCS is available at \url{https://github.com/ZikaiXiong/PDCS} and can be called from JuMP~\citep{LubinDunningIJOC} and CVXPY~\citep{diamond2016cvxpy}. The interfaces convert a user model into a sparse conic program, which cuPDCS preprocesses, solves by PDCS, and postprocesses. The CPU and GPU versions share this structure. On the GPU, all problem data stay in GPU memory, and each projection is computed by a kernel selected according to the type, dimension, and rescaling of the cone. Supporting a new cone requires only its conversion and projection routines, and a new GPU mapping requires only a kernel and a rule for selecting it. The rules for the step size, restarts, and termination are implemented separately from the projection kernels.

\begin{figure}[t]
\centering
\includegraphics[width=0.6\textwidth]{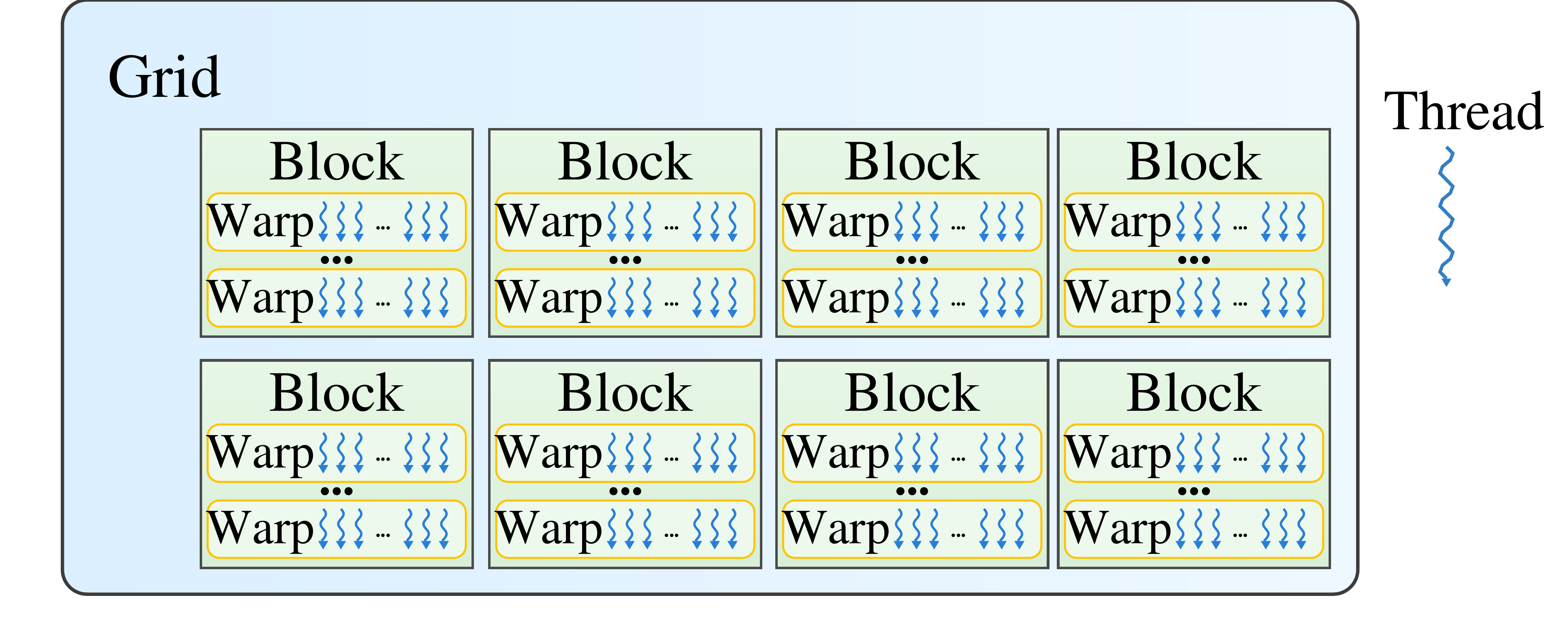}
\caption{{Hierarchical organization of GPU computation.}}
\label{fig:gpu_architecture}
\end{figure}
\begin{wrapfigure}[15]{r}{0.50\textwidth}
\centering
\makeatletter
\let\@makecaption\@makefigurecaption
\makeatother
\includegraphics[width=\linewidth]{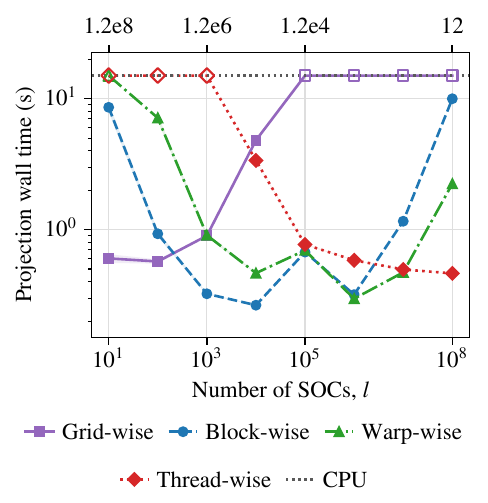}
\par\raggedright
\caption{{Projection time at fixed total SOC dimension
$ld=1.2\times10^9$.}}
\label{fig:soc_res}
\end{wrapfigure}
Sparse matrix-vector multiplication is well optimized on GPUs~\citep{bell2008efficient}, so we focus on the projections, whose parallelization depends on the number and the dimensions of the cones.

A GPU kernel is launched over a grid of thread blocks, and the threads in each block are scheduled in groups called warps (Figure~\ref{fig:gpu_architecture}). cuPDCS has four mappings. (a) \textbf{Grid-wise.} Several blocks cooperate on one cone. This gives the most parallelism within a cone, but repeated kernel launches become expensive when there are many cones. (b) \textbf{Block-wise.} One block handles one cone, and its threads cooperate through shared memory, which is effective for large reductions. (c) \textbf{Warp-wise.} One warp handles one cone. It still performs the reduction in parallel but uses fewer resources than a block, so it suits many cones of small or moderate dimension. (d) \textbf{Thread-wise.} One thread handles one cone. The work within a cone is serial, but many small cones can be processed at the same time. From Grid-wise to Thread-wise, the parallelism within each cone decreases and the number of cones processed at the same time increases.

The choice of mapping matters most for diagonally rescaled second-order cones. Theorem~\ref{thm:soc_proj} reduces each such projection to finding the root of a scalar function, and each evaluation of this function requires elementwise operations and a reduction over the coordinates of the cone. Since the four mappings parallelize this reduction to different degrees, the best mapping depends on both the number of cones and their dimensions.

We illustrate this dependence using randomly generated instances over
$\mathcal K=\mathcal K_{\mathrm{soc}}^{d}\times\cdots\times
\mathcal K_{\mathrm{soc}}^{d}$, a product of $l$ second-order cones of dimension $d$, fixing the total dimension at
$ld=1.2\times10^9$ while varying $l$. Figure~\ref{fig:soc_res}
reports the mean projection time over ten trials. The dotted CPU line marks the
15-second limit because all CPU runs time out. Grid-wise and Block-wise execution are faster when there are few cones of high dimension, and Warp-wise and Thread-wise execution are faster when there are many cones of low dimension.

Figure~\ref{fig:soc_res_dimension} fixes $l$ at 100 or 10{,}000 and varies $d$. As $d$ grows, the fastest mapping changes from Thread-wise execution for very small cones to Warp-wise execution at moderate dimensions. For $l=100$, Block-wise execution becomes the fastest for larger cones. For $l=10{,}000$, Warp-wise execution remains the fastest for the larger dimensions we tested, because more cones can be processed in parallel.

\begin{figure}[!htbp]
\centering
\includegraphics[width=\textwidth]{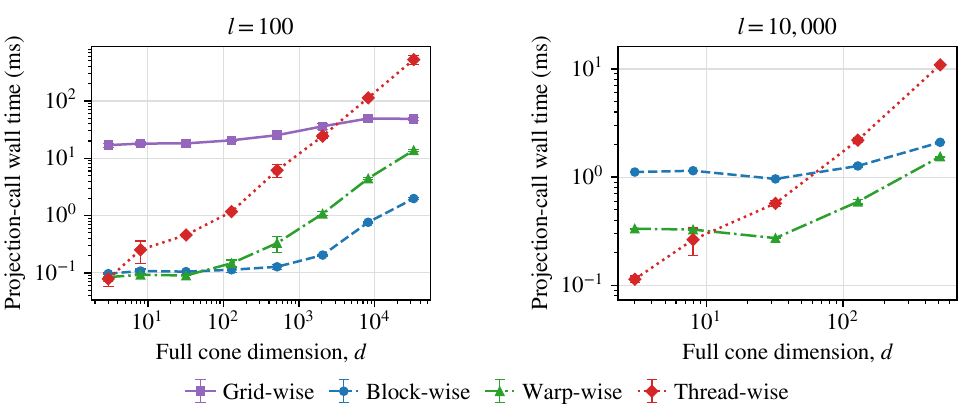}
\caption{{Projection time versus SOC dimension at fixed $l$. Grid-wise results
are omitted for $l=10{,}000$ because of the high kernel-launch overhead.}}
\label{fig:soc_res_dimension}
\end{figure}

\noindent\begin{minipage}[t]{\textwidth}
\setlength{\parindent}{1em}
\rightskip=0pt plus 2em
\parfillskip=0pt plus 1fil
\spaceskip=\fontdimen2{\intextsep}{0pt}
\begin{wrapfigure}[11]{r}{0.50\textwidth}
\centering
\makeatletter
\let\@makecaption\@makefigurecaption
\makeatother
\includegraphics[width=\linewidth]{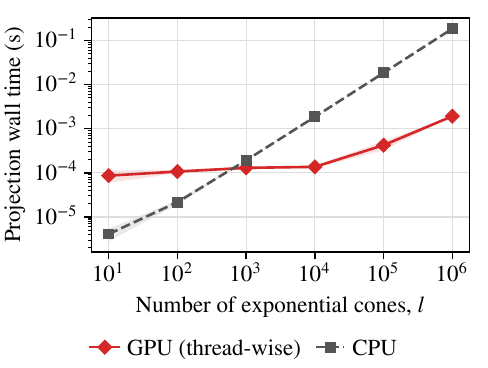}
\caption{{Projection time onto $l$ exponential cones on the CPU and the GPU.}}
\label{fig:exp_res}
\end{wrapfigure}
An exponential cone has fixed dimension three, so its root-finding projection maps naturally to one thread per cone. Figure~\ref{fig:exp_res} shows that the CPU is faster when there are few cones, because the overhead of launching GPU kernels dominates, and that the advantage of the GPU grows as more independent projections can be computed in parallel.
For zero and nonnegative cones, projection is elementwise and Thread-wise
execution suffices.
\par\WFclear
\end{minipage}
\par\medskip

\noindent\begin{minipage}{\textwidth}
\setlength{\parindent}{1em}
\rightskip=0pt plus 2em
\parfillskip=0pt plus 1fil
\spaceskip=\fontdimen2{\intextsep}{0pt}
\begin{wraptable}[6]{r}{0.50\textwidth}
\centering
\makeatletter
\let\@makecaption\@maketablecaption
\makeatother
\caption{{Implemented projection mappings.}}
\label{tab:projection_strategy}
{%
\normalcolor
\SingleSpacedXI
\setlength{\tabcolsep}{5pt}
\renewcommand{\arraystretch}{1.05}
\begin{tabular*}{\linewidth}{@{\extracolsep{\fill}}lcccc@{}}
\toprule
Cone type & Thread & Warp & Block & Grid \\
\midrule
Zero & \textcolor{green}{\ding{51}} & \textcolor{red}{\ding{55}} & \textcolor{red}{\ding{55}} & \textcolor{red}{\ding{55}} \\
Nonnegative & \textcolor{green}{\ding{51}} & \textcolor{red}{\ding{55}} & \textcolor{red}{\ding{55}} & \textcolor{red}{\ding{55}} \\
SOC & \textcolor{green}{\ding{51}} & \textcolor{green}{\ding{51}} & \textcolor{green}{\ding{51}} & \textcolor{green}{\ding{51}} \\
Exponential & \textcolor{green}{\ding{51}} & \textcolor{red}{\ding{55}} & \textcolor{red}{\ding{55}} & \textcolor{red}{\ding{55}} \\
\bottomrule
\end{tabular*}
}
\end{wraptable}
Table~\ref{tab:projection_strategy} lists the mappings implemented for each cone type. Only second-order cones use all four mappings, and the other cones use the Thread-wise mapping. The crossover points in Figures~\ref{fig:soc_res}--\ref{fig:exp_res} guide the choice of mapping in cuPDCS on the H100 GPU used in our experiments and may differ on other hardware. The rule that cuPDCS uses to select the mapping is given in the technical note \citep{alg_detail}.
\par\WFclear
\end{minipage}
\par

\endgroup
\section{\label{sec:numerical}Numerical experiments}

We evaluate our method's performance against several state-of-the-art conic solvers:

\begin{itemize}
    \item \textbf{SCS} \citep{o2016conic}: an ADMM-based solver with CPU-direct, CPU-indirect, and GPU-indirect variants, denoted SCS(direct), SCS(indirect), and SCS(GPU). Here, ``direct'' denotes the use of matrix factorization, while ``indirect'' refers to iterative methods for solving inner linear systems.
    \item \textbf{CuClarabel} \citep{CuClarabel}: an open-source GPU IPM using cuDSS.
    \item \textbf{MOSEK} \citep{MOSEK}, \textbf{COPT} \citep{ge2022cardinal}: commercial IPM-based solvers.
\end{itemize}
The methods that do not require matrix factorization are classified as \textit{matrix-free first-order methods}. They include SCS(indirect) and PDCS (including its GPU implementation cuPDCS). 
The methods that require a single matrix factorization are classified as \textit{non-matrix-free first-order methods}. They include SCS(direct).
The other solvers are essentially based on IPMs. They are CuClarabel, MOSEK and COPT. For certain problem instances, the presolve capabilities of the commercial solvers can exploit structural properties to reduce the problem complexity substantially, which may give these solvers a computational advantage. We denote versions of COPT and MOSEK with presolve enabled as COPT* and MOSEK*, respectively, and the versions without presolve as COPT and MOSEK.

We organize the experiments as follows.  Section~\ref{sec:cblib} evaluates performance on the CBLIB, a small-scale classical conic program dataset. 
 In Section~\ref{sec:fisher_market}, we consider the Fisher market equilibrium problem, an exponential cone program. Section~\ref{sec:lasso} focuses on solving large-scale Lasso problems. Finally, in Section~\ref{sec:mpo}, we consider multi-period portfolio optimization,  a conic program with multiple second-order cones. {The conic formulations and instance generators of the three problem classes are given in the technical note \citep{alg_detail}.} 

{All GPU experiments were run on an NVIDIA H100 80~GB GPU using Julia~1.10.4,
JuMP~1.31.2, MathOptInterface~1.53.0, CUDA Toolkit~12.6.2, and CUDA.jl~5.11.3.
Solver versions were SCS~3.2.11 (SCS.jl~2.6.4, CUDA~11.8), CuClarabel
(Clarabel.jl~0.11.0, CUDSS.jl~0.6.5), COPT~8.0.5, and MOSEK~10.0.46. CPU
benchmarks were run on a Mac mini M2 Pro with 32~GB of RAM.}

Because each method adopts different formulations and methodologies, their actual built-in termination criteria can vary even when we set the same tolerance. Moreover, some methods keep their termination criteria private. There is no uniform convergence standard across the methods we evaluated, but we attempt to document and summarize the known termination criteria for these methods in Appendix~\ref{app:solver_criteria}. 
For PDCS and cuPDCS, we assess primal and dual feasibility, as well as primal-dual optimality:
 \begin{equation}\label{eq:PDCS_terminate_tolerance}
    \small
    \begin{array}{ll}
        &\err_{\textrm{p}}(\xbf):={\norm{(\Gbf\xbf-\hbf)-\proj_{\mcal K_{\textrm{d}}^{*}}\{\Gbf\xbf-\hbf\}}_{\infty}}/\brbra{1+\max\left\{\norm{\hbf}_{\infty},\norm{\Gbf \xbf}_{\infty},\norm{\proj_{\mcal K_{\textrm{d}}^{*}}\{\Gbf\xbf-\hbf\}}_{\infty}\right\}},\\
        &\err_{\textrm{d}}(\bm{\lambda}_1,\bm{\lambda}_2):={\max\left\{\norm{\bm{\lambda}_{1}-\proj_{\Lambda}\{\bm{\lambda}_{1}\}}_{\infty},\norm{\bm{\lambda}_{2}-\proj_{\mcal K_{\textrm{p}}^*}\{\bm{\lambda}_{2}\}}_{\infty} \right\}}/\brbra{1+\max\left\{\norm{ \cbf}_{\infty},\norm{\Gbf^\top \ybf}_{\infty}\right\}},\\
        &\err_{\textrm{gap}}(\xbf,\bm{\lambda}_1,\bm{\lambda}_2):={\abs{\inner{\cbf}{\xbf}-\brbra{\ybf^{\top}\hbf+\lbf^{\top}\bm{\lambda}_{1}^{+}-\ubf^{\top}\bm{\lambda}_{1}^{-}}}}/\brbra{1+\max\bcbra{\abs{\inner{\cbf}{\xbf}},\abs{\ybf^{\top}\hbf+\lbf^{\top}\bm{\lambda}_{1}^{+}-\ubf^{\top}\bm{\lambda}_{1}^{-}}}}.
    \end{array}
\end{equation} 
{In $\err_{\textrm{gap}}$, terms involving infinite entries of $\lbf$ or $\ubf$ are omitted, which amounts to evaluating the dual objective at $\proj_{\Lambda}\{\bm{\lambda}_1\}$.} Our method terminates once all three errors are below a tolerance, {which is $10^{-6}$ in all our experiments}. According to the Karush--Kuhn--Tucker conditions, when these three criteria are all zero, the solution is optimal.  

In order to compare the runtime of different methods across multiple instances, we use the shifted geometric mean (SGM): $\text{SGM}(k):=\left[\prod_{i=1}^{N}(t_i + k)\right]^{1/N}-k$
where \(t_i\) is the time (in seconds) for the method to run on problem $i$, and $k$ is a shift parameter that is often set as 10 \citep{sgm}. SGM mitigates the potentially significant influence of tiny runtimes of only a few instances on the overall geometric mean. 
If a method fails to solve a given instance, we assign $t_i$ to be the maximum allowable time (e.g., the time limit).  

\subsection{Relaxation of Problems from the CBLIB Dataset\label{sec:cblib}}

CBLIB is a public dataset of mixed-integer conic programming problems collected from real-world applications \citep{friberg2016cblib}. We run our experiments on the instances with LP cones, second-order cones, and exponential cones. 
In line with standard practice for solving such problems, we first use COPT to apply presolve and subsequently relax all integer constraints, which produces two subsets: (1) 1,943 problems that do not involve exponential cones, and (2) 157 problems that include exponential cones.

\begingroup
\textbf{Ablation study.}\par
\begin{wraptable}[8]{r}{0.52\textwidth}
\vspace{-15pt}
\centering
\footnotesize
\setlength{\tabcolsep}{3pt}
\makeatletter
\let\@makecaption\@maketablecaption
\makeatother
\caption{{Ablation results on 62 representative CBLIB instances.}}
\label{tab:one-at-a-time-ablation-sgm10}
{\normalcolor
\begin{tabular}{lrr}
\toprule
Configuration & Solved & SGM(10) (s) \\
\midrule
PDCS (Full configuration)                    & {61/62} & {25.67} \\
No Halpern update                      & {61/62} & 31.34 \\
No adaptive reflection                & 60/62          & 28.36 \\
No weighted averaging & 59/62          & 33.82 \\
No adaptive step size $\eta$                  & 59/62          & 24.33 \\
No restart                             & 50/62          & 95.38 \\
No adaptive primal weight $\omega$                  & 49/62          & 93.61 \\
No diagonal rescaling                  & 45/62          & 94.08 \\
\bottomrule
\end{tabular}}
\end{wraptable}

We assess each algorithmic component by disabling it alone in the full solver. We select one instance from each CBLIB subcollection, which gives 62 instances, of which 40 contain only second-order cones and 22 contain exponential cones. Each run has a one-hour time limit. Table~\ref{tab:one-at-a-time-ablation-sgm10} reports the number of solved instances and the SGM(10) of the wall time.

The full configuration solves the most instances (61, as does the configuration without the Halpern update) and has the lower SGM(10) of the two. Removing the adaptive primal weight, restarts, or diagonal rescaling reduces the number of solved instances the most, to 49, 50, and 45, respectively. Disabling the adaptive step size slightly reduces SGM(10) but solves two fewer instances. We therefore enable all components by default.
\endgroup
\par

\textbf{Comparison with other solvers.} 
For the subset without exponential cones, we categorize problem instances as small-, medium-, and large-scale based on the number of nonzeros in the constraint matrix, using thresholds of 50{,}000 and 500{,}000. As a result, the dataset comprises 1{,}641 small-scale, 220 medium-scale and 82 large-scale problem instances. 
Almost all of these problems are relatively small, so the IPM-based solvers and the SCS(direct) are unlikely to reach computational bottlenecks and can easily perform the required matrix factorizations. We start with this dataset to give an overview of the performance of PDCS on small-scale problems.
All experiments in this benchmark are executed with a one-hour time limit and a termination tolerance of $10^{-6}$.

        \begin{table}[htbp]
    \vspace{-5pt}
    \SingleSpacedXI
    \footnotesize
      \centering
      \caption{Conic Program Problems without Exponential Cones\label{tab:socp_res_summary}}
      \resizebox{0.95\textwidth}{!}{
        \begin{tabular}{cclrrrrrrrr}
        \toprule
              &       &       & \multicolumn{2}{c}{Small(1641)} & \multicolumn{2}{c}{Medium(220)} & \multicolumn{2}{c}{Large(82)} & \multicolumn{2}{c}{Total(1943)} \\
    \cmidrule{4-11}          &       &       & \multicolumn{1}{l}{SGM(10)} & \multicolumn{1}{l}{solved} & \multicolumn{1}{l}{SGM(10)} & \multicolumn{1}{l}{solved} & \multicolumn{1}{l}{SGM(10)} & \multicolumn{1}{l}{solved} & \multicolumn{1}{l}{SGM(10)} & \multicolumn{1}{l}{solved} \\
        \midrule

        \multirow{4}{*}{\shortstack{First-order\\methods}}
              & \multirow{3}{*}{Matrix-free} & SCS(indirect) & 3.31  & 1638  & 231.68 & 188   & 1856.58 & 46    & 12.77 & 1872 \\
              & & PDCS  & 2.78  & 1640& 160.64 & 208   & 1602.14 & 53    & 11.02 & 1901 \\
              & & cuPDCS & {2.38}  & {1639}& {\textbf{14.23}} & {215}   & {\textbf{105.99}} & {65}    & {\textbf{4.68}}  & {1919}\\

        \cdashline{2-11}[1pt/2pt]
              & Non-matrix-free & SCS(direct) & \textbf{0.59}  & 1636  & 54.98 & 213& 463.21 & 58& {5.26}  & 1907 \\
 
        \hdashline[3pt/1.5pt]

        \multicolumn{2}{c}{\multirow{3}{*}{\shortstack{Interior-point\\methods}}}
              & CuClarabel & 0.15  & 1640& \textbf{1.32}  & 219& 46.17 & 61    & 1.05  & 1920 \\
        \multicolumn{2}{c}{}
              & COPT* & 0.08  & 1639  & 2.03  & 215   & 53.76 & 58    & 1.12  & 1912 \\
        \multicolumn{2}{c}{}
              & MOSEK* & \textbf{0.05}  & 1640& 2.34  & 214   & \textbf{6.21}  & 81& \textbf{0.49}  & 1935\\

        \bottomrule
        \end{tabular}}
        \vspace{-15pt}
    \end{table}%

Tables~\ref{tab:socp_res_summary} and \ref{tab:exp_res_summary} summarize the results for these two categories of problems. {Among the matrix-free first-order methods, cuPDCS is the fastest on all four subsets. Among all first-order methods, SCS(direct) is faster on the small-scale and exponential-cone subsets, whereas cuPDCS is the fastest on the medium- and large-scale subsets without exponential cones.} 

Both the CPU-based version of our method (denoted by PDCS) and its GPU-enhanced implementation (denoted by cuPDCS) exhibit strong numerical stability, solving nearly as many instances as the commercial solver COPT.  However, the CPU version occasionally struggles with more challenging instances, primarily due to the limited parallelism in projection operations and matrix-vector multiplications. In contrast, the GPU implementation, cuPDCS, enables more iterations within the allotted time and successfully solves some instances that are otherwise too difficult for the CPU version. 

\newcommand{\thickhline}{\noalign{\hrule height 0.7pt}}
\begin{table}[htbp]
    \vspace{-5pt}
    \SingleSpacedXI
    \footnotesize
  \centering
  \caption{Conic Program Problems with Exponential Cones~\label{tab:exp_res_summary}}
  \resizebox{0.8\textwidth}{!}{ 
    \begin{tabular}{lrrrr:rrr}
    \thickhline
          & \multicolumn{4}{c:}{First-order methods}      & \multicolumn{3}{c}{Interior-point methods} \\
          & \multicolumn{1}{l}{SCS(direct)} & \multicolumn{1}{l}{SCS(indirect)} & 
          \multicolumn{1}{l}{PDCS} & \multicolumn{1}{l:}{cuPDCS} & \multicolumn{1}{l}{CuClarabel} & \multicolumn{1}{l}{COPT*} & \multicolumn{1}{l}{MOSEK*} \\
    \thickhline
    SGM(10) & \textbf{6.07}  & 19.98 &  18.52 & {8.03} & 0.87  & \textbf{0.13}  & 5.19 \\
    solved & 152   & 152   & 149  & {156}   & 156   & {157}   & 146 \\
    \thickhline
    \end{tabular}}
    \vspace{-15pt}
\end{table}%
\setlength{\arrayrulewidth}{0.4pt}

The computational bottleneck of matrix-free methods lies in the sparse matrix-vector products, so we also compare the number of matrix-vector products required for cuPDCS and SCS(indirect) to solve a particular number of problems. Figure~\ref{fig:solved_count_SGM} reports the fraction of instances solved versus the number of sparse matrix-vector products. 
When we compare the number of matrix-vector products needed to solve 60\% of the problem instances, we see that cuPDCS requires only about one-fifth to one-tenth as many matrix-vector products as the indirect SCS. Since cuPDCS follows the PDCS algorithm, we can conclude that, even if not implemented on a GPU, PDCS still requires fewer matrix-vector products than the other matrix-free solver, SCS(indirect). 
{Fewer matrix-vector products do not reduce the running time proportionally. For each of the 2,100 CBLIB instances, we divide the iterative-phase time by the total cuPDCS runtime. The medians of these per-instance ratios are 31.62\%, 79.13\%, 86.89\%, and 60.90\% on the small, medium, large, and exponential-cone subsets, respectively. Thus, non-iterative overhead has a larger effect on the short-running small and exponential-cone instances, whereas the iterative phase dominates on the medium and large subsets.}

\begin{figure}[htbp]
\begin{centering}
\begin{subfigure}[t]{0.49\columnwidth}%
\includegraphics[width=0.95\textwidth]{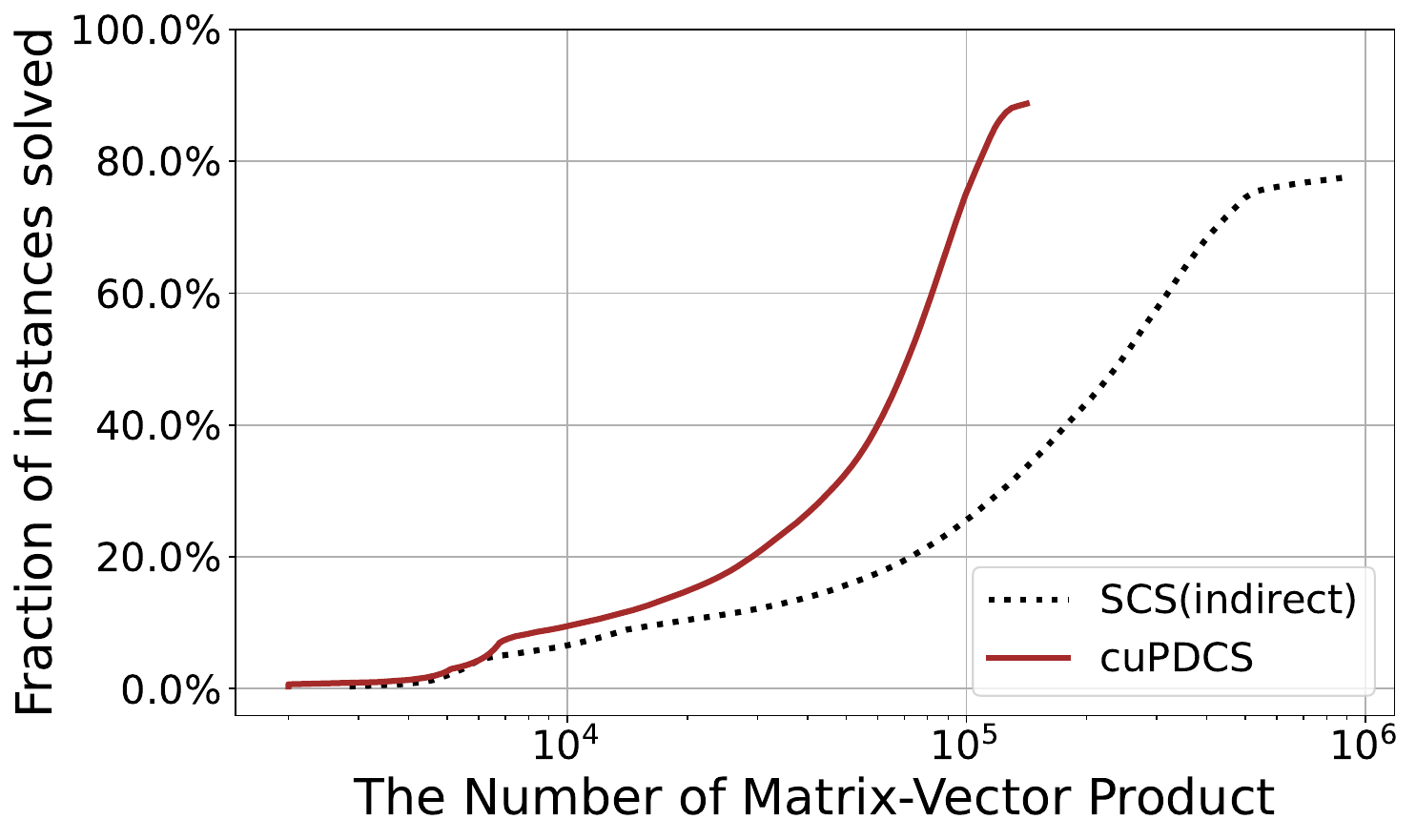}%
\caption{\SingleSpacedXI\small  Medium- and large-scale conic program problems without exponential cones}
\end{subfigure}\hfill
\begin{subfigure}[t]{0.49\columnwidth}%
\includegraphics[width=0.95\textwidth]{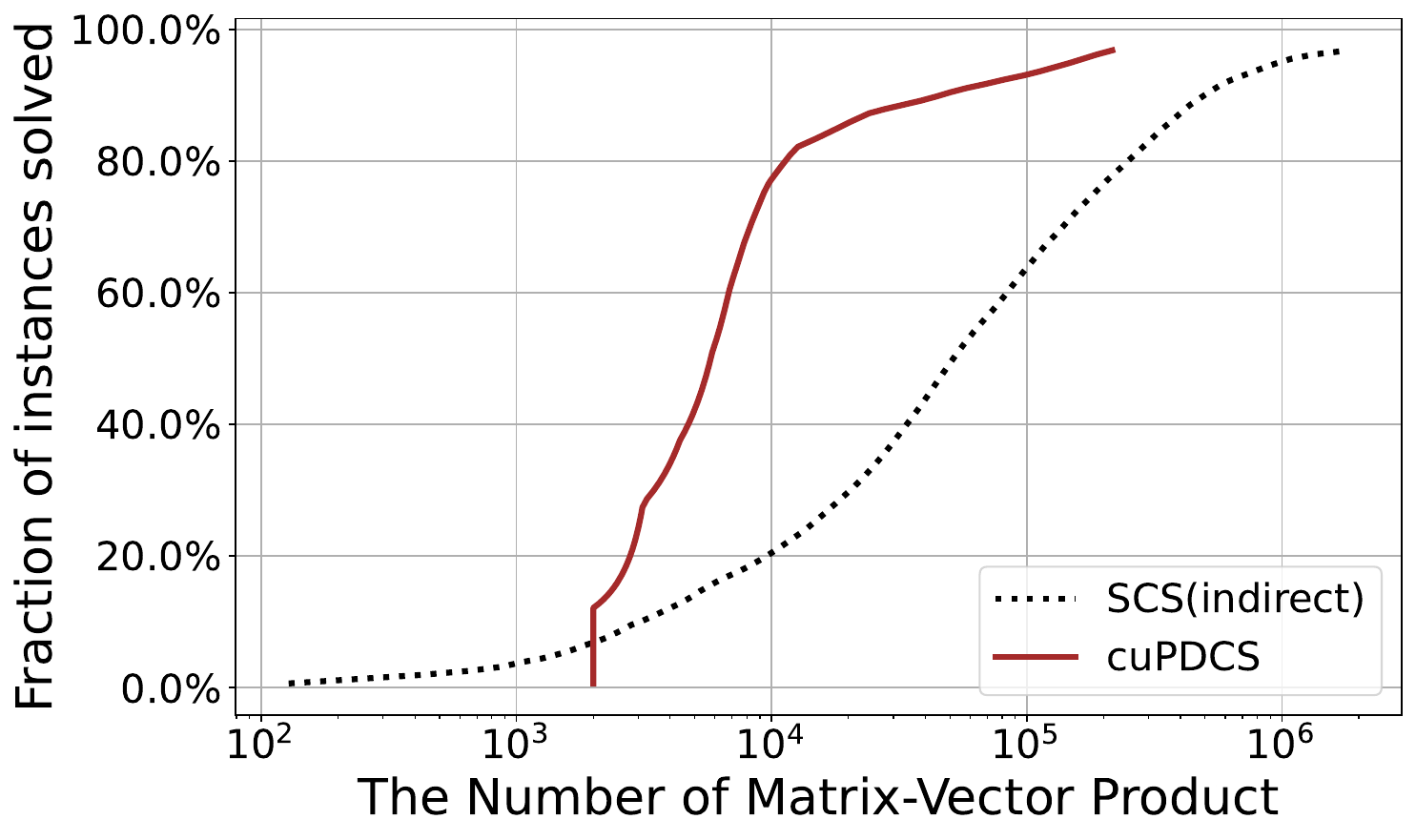}%
\caption{\SingleSpacedXI\small  Conic program problems with exponential cones}
\end{subfigure}
\end{centering}
\caption{\label{fig:solved_count_SGM}Performance of cuPDCS and SCS(indirect) in terms of the number of matrix-vector products.}
    \vspace{-10pt} 
\end{figure} 

Among the interior-point methods, MOSEK* generally demonstrates advantages over other methods for conic programs without exponential cones. However, for problems involving exponential cones, MOSEK* frequently fails to meet the necessary solution tolerance, resulting in a ``SLOW\_PROGRESS'' status. All such instances are regarded as failures, regardless of whether the time limit is reached. This makes the SGM of MOSEK* for problems with exponential cones appear worse than the other interior-point methods.

The problems in the CBLIB dataset are relatively small. For these problems, solving the matrix factorizations and the corresponding linear systems remains accessible. {Hence, SCS(direct) is the fastest first-order method on the small-scale and exponential-cone subsets}, while all IPMs significantly outperform first-order methods. However, as the dimension of the problem instance increases, SCS(direct) and all IPMs face greater challenges and become considerably slower.  In the following, we present the experiments on larger-scale problems.

\subsection{Fisher Market Equilibrium Problem~\label{sec:fisher_market}}
We next consider Fisher market equilibrium instances~\citep{eisenberg1959consensus,ye2008path}, an important optimization problem arising in economics.
{We formulate each instance as an exponential cone program 
 and evaluate all applicable solvers at a tolerance of $10^{-6}$ with a five-hour limit. 
Table~\ref{tab:fisher_res_repeated} reports the results, where $m$ and $n$ are the numbers of buyers and goods.}

\begin{table}[htbp]
  \centering
  \vspace{-5pt}
  \caption{{Results for Fisher market equilibrium problems at the target tolerance $10^{-6}$.} \label{tab:fisher_res_repeated}}
  \renewcommand{\arraystretch}{0.88}
  \resizebox{0.95\textwidth}{!}{
  \begin{tabular}{lrrrrrrr}
  \toprule
  & \multicolumn{5}{c}{Without Presolve} & \multicolumn{2}{c}{With Presolve} \\
  \cmidrule(lr){2-6}\cmidrule(lr){7-8}
  Size $(m,n)$ & cuPDCS & SCS(GPU) & CuClarabel & COPT & MOSEK & COPT* & MOSEK* \\
  \midrule
  $(100,5{,}000)$             & $43.0\pm6.3$       & $10135.5\pm1704.1$   & NI                   & $1.8$\textsuperscript{$\ddagger$} & \cellcolor[rgb]{ .816, .808, .808}$\boldsymbol{1.7\pm0.4}$ & NI & $1.9\pm0.5$ \\
  $(100{,}000,1{,}000)$       & $787.6\pm262.5$    & TL                    & $216.5\pm12.7$       & \cellcolor[rgb]{ .816, .808, .808}$136.4\pm8.3$ & $136.4\pm3.8$ & $\boldsymbol{93.7\pm2.0}$ & $96.1\pm2.1$ \\
  $(150{,}000,1{,}000)$       & $1025.4\pm251.0$   & TL\textsuperscript{$\S$}                    & $332.2\pm14.2$       & $360.0\pm24.3$ & \cellcolor[rgb]{ .816, .808, .808}$204.8\pm6.3$ & $272.8\pm6.4$ & $\boldsymbol{145.9\pm2.4}$ \\
  $(200{,}000,1{,}000)$       & $1312.1\pm488.8$   & TL\textsuperscript{$\S$}                    & $542.0\pm17.2$        & $613.8\pm29.6$ & \cellcolor[rgb]{ .816, .808, .808}$277.1\pm5.4$ & $423.8\pm9.7$ & $\boldsymbol{199.4\pm2.6}$ \\
  $(300{,}000,1{,}000)$       & $1791.3\pm754.6$   & TL\textsuperscript{$\S$}                    & \cellcolor[rgb]{ .816, .808, .808}$704.1\pm38.3$ & $1123.5\pm69.3$ & $919.1\pm82.1$ & $749.2\pm43.4$ & $\boldsymbol{688.7\pm28.1}$ \\
  $(400{,}000,1{,}000)$       & \cellcolor[rgb]{ .816, .808, .808}$2183.1\pm16.2$ & TL\textsuperscript{$\S$} & OOM & $4046.2\pm310.1$ & $2730.8\pm51.4$ & $2897.3\pm239.4$ & $\boldsymbol{1355.8\pm43.9}$ \\
  $(500{,}000,1{,}000)$       & \cellcolor[rgb]{ .816, .808, .808}$\boldsymbol{2888.3\pm2.3}$ & TL\textsuperscript{$\S$} & OOM\textsuperscript{$\S$} & $8180.9\pm655.0$ & $7332.3\pm94.3$ & $6531.9\pm590.2$ & $3397.9\pm71.1$ \\
  $(600{,}000,1{,}000)$       & \cellcolor[rgb]{ .816, .808, .808}$\boldsymbol{3523.8\pm248.3}$ & TL\textsuperscript{$\S$} & OOM\textsuperscript{$\S$} & $15471.0\pm768.9$ & $16936.5\pm539.4$ & $12289.8\pm770.9$ & $9599.4\pm387.1$ \\
  $(700{,}000,1{,}000)$       & \cellcolor[rgb]{ .816, .808, .808}$\boldsymbol{3940.3\pm36.7}$ & TL\textsuperscript{$\S$} & OOM\textsuperscript{$\S$} & TL & OOM & $15613.6\pm759.8$\textsuperscript{$\dagger$} & OOM \\
  $(800{,}000,1{,}000)$       & \cellcolor[rgb]{ .816, .808, .808}$\boldsymbol{4575.3\pm22.1}$ & TL\textsuperscript{$\S$} & OOM\textsuperscript{$\S$} & OOM & OOM & OOM & OOM \\
  \midrule
  SGM(10)                     & 1413.7              & 16977.7               & 3986.3                & 2756.5                      & \cellcolor[rgb]{ .816, .808, .808}1258.6 & 2555.4 & \textbf{906.1} \\
  Solved                      & 50/50               & 5/50                  & 20/50                 & 36/50                       & 40/50                  & 37/50                & 40/50 \\
  \bottomrule
  \end{tabular}}\par
  \vspace{5pt}
  \begin{minipage}{0.96\linewidth}
  \footnotesize\raggedright
  {\textit{Note}. Bold entries have the lowest mean wall time among all methods, while gray cells have the lowest mean wall time among methods without presolve. The same convention applies to SGM(10). Unmarked numerical entries summarize 5/5 solved instances. S, TL, OOM, and NI denote solved, time limit, out of memory, and numerical issue, respectively.  \textsuperscript{$\ddagger$}1 S and 4 NI. \textsuperscript{$\dagger$}2 S and 3 TL. \textsuperscript{$\S$}Some or all instances were not run after all five instances at a smaller size failed for the same reason.  Means and sample standard deviations use only solved instances at each size.}
  \end{minipage}
  \vspace{-15pt}
\end{table}%

{Table~\ref{tab:fisher_res_repeated} shows that the comparison depends on the problem size. cuPDCS is the only solver that solves all 50 instances, and it is much faster than SCS(GPU). CuClarabel and the commercial IPMs are faster on smaller instances, but they fail on more instances as the size grows, because of the time or memory limit. Without presolve, cuPDCS is the fastest solver from $m=400{,}000$, and it is the fastest of all methods from $m=500{,}000$.

Presolve generally improves COPT and MOSEK on moderate-sized instances, but it does not prevent failures on the largest instances. Thus, although MOSEK* has the lowest overall SGM(10), cuPDCS is the only solver that solves all instances, including those of the largest size, on which every other solver fails.}

\subsection{Lasso Problem} \label{sec:lasso}
The Lasso problem $\min_{\xbf\in \mbb R^n} \norm{\Abf \xbf - \bbf}^2+\lambda \norm{\xbf}_1$ is a classic problem in statistical learning, where $\Abf$ is the data matrix, $\bbf$ is the response vector, and $\lambda>0$ is the regularization parameter.

{Following the setting in~\citet{deng2024enhanced}, we generate five synthetic instances at each of five problem sizes and reformulate them as SOCPs by splitting $\xbf$ into its positive and negative parts. All runs use a tolerance of $10^{-6}$ and a five-hour limit. Table~\ref{tab:lasso_res_repeated} reports the results, where $m$ and $n$ are the numbers of rows and columns of $\Abf$.}

\begin{table}[htbp]
  \centering
  \vspace{-5pt}
  \caption{{Results for Lasso problems at the target tolerance $10^{-6}$.} \label{tab:lasso_res_repeated}}
  \renewcommand{\arraystretch}{0.88}
  \resizebox{0.95\textwidth}{!}{
  \begin{tabular}{lrrrrrrr}
  \toprule
  & \multicolumn{5}{c}{Without Presolve} & \multicolumn{2}{c}{With Presolve} \\
  \cmidrule(lr){2-6}\cmidrule(lr){7-8}
  Size $(m,n)$ & cuPDCS & SCS(GPU) & CuClarabel & COPT & MOSEK & COPT* & MOSEK* \\
  \midrule
  $(10{,}000,100{,}000)$       & $26.8\pm0.7$ & $27.3\pm0.8$ & $41.0\pm0.6$ & \cellcolor[rgb]{ .816, .808, .808}$7.2\pm0.3$ & $8.6\pm0.6$ & $\boldsymbol{5.7\pm0.1}$ & $6.7\pm0.3$ \\
  $(70{,}000,700{,}000)$       & \cellcolor[rgb]{ .816, .808, .808}$\boldsymbol{28.1\pm1.0}$ & $1225.7\pm49.0$ & $210.6\pm12.6$ & $3817.0\pm123.7$ & $2958.2\pm107.4$ & $3864.0\pm14.4$ & $3836.1\pm767.3$ \\
  $(400{,}000,7{,}000{,}000)$ & \cellcolor[rgb]{ .816, .808, .808}$\boldsymbol{157.8\pm2.7}$ & TL & OOM & OOM & OOM & OOM & OOM \\
  $(700{,}000,7{,}000{,}000)$ & \cellcolor[rgb]{ .816, .808, .808}$\boldsymbol{238.1\pm4.0}$ & TL\textsuperscript{$\S$} & OOM\textsuperscript{$\S$} & OOM\textsuperscript{$\S$} & OOM\textsuperscript{$\S$} & OOM\textsuperscript{$\S$} & OOM\textsuperscript{$\S$} \\
  $(750{,}000,7{,}500{,}000)$ & \cellcolor[rgb]{ .816, .808, .808}$\boldsymbol{303.6\pm35.7}$ & TL\textsuperscript{$\S$} & OOM\textsuperscript{$\S$} & OOM\textsuperscript{$\S$} & OOM\textsuperscript{$\S$} & OOM\textsuperscript{$\S$} & OOM\textsuperscript{$\S$} \\
  \midrule
  SGM(10)                       & \cellcolor[rgb]{ .816, .808, .808}\textbf{102.7} & 3051.7 & 2298.5 & 3276.9 & 3163.0 & 3227.0 & 3253.1 \\
  Solved                        & 25/25              & 10/25               & 10/25               & 10/25                & 10/25                & 10/25                & 10/25 \\
  \bottomrule
  \end{tabular}}\par
  \vspace{5pt}
  \begin{minipage}{0.96\linewidth}
  \footnotesize\raggedright
  {\textit{Note}. See Table~\ref{tab:fisher_res_repeated} for the reporting conventions, outcome abbreviations, and SGM(10) treatment.}
  \end{minipage}
  \vspace{-15pt}
\end{table}%

{Table~\ref{tab:lasso_res_repeated} shows that the commercial IPMs are the fastest on the smallest instances, on which cuPDCS and SCS(GPU) have similar running times. On the four larger sizes, cuPDCS is the fastest method and the only solver that solves all 25 instances, and it has the lowest SGM(10). On the three largest sizes, the other solvers either exceed the time or memory limit or are not run because they already failed at a smaller size, whereas cuPDCS solves every instance.
Presolve helps COPT and MOSEK only on the smallest size and does not let them solve more instances. CuClarabel is faster than the commercial IPMs at the second size but runs out of memory at the third size and is not run on the larger sizes.}

\subsection{\label{sec:mpo}Multi-period Portfolio Optimization Problems}

In this section, we study Multi-Period Portfolio Optimization (MPO) problems, 
which include a mix of multiple cone types and increasingly higher problem dimensions as the number of time periods grows.

We select all securities from five major indices (S\&P 500, NASDAQ-100, DOW 30, FTSE 100, and Nikkei 225), resulting in 844 assets in total. The historical data used to construct the problem were retrieved using the \texttt{cvxportfolio} package~\citep{boyd2017multi}, which sources prices from Yahoo Finance. The data from the year 2024 were used to estimate the returns, variances, and prices for the first period. For subsequent periods, the predicted data were generated by introducing a controlled level of random noise to the first period's estimates. {For each horizon $T$, we generate five instances using different realizations of this noise.} {The remaining model parameters are set as in the technical note \citep{alg_detail}.}

\begin{table}[htbp]
  \centering
  \vspace{-5pt}
  \caption{{Results for MPO problems at the target tolerance $10^{-6}$.} \label{tab:mpo_res_repeated}}
  \renewcommand{\arraystretch}{0.82}
  \resizebox{0.85\textwidth}{!}{
  \begin{tabular}{lrrrrrrr}
  \toprule
  & \multicolumn{5}{c}{Without Presolve} & \multicolumn{2}{c}{With Presolve} \\
  \cmidrule(lr){2-6}\cmidrule(lr){7-8}
  Horizon & cuPDCS & SCS(GPU) & CuClarabel & COPT & MOSEK & COPT* & MOSEK* \\
  \midrule
  $T=3$       & $19.0\pm13.5$ & $3.3\pm1.8$ & $13.9\pm13.8$ & $2.8\pm0.4$ & \cellcolor[rgb]{ .816, .808, .808}$1.6\pm0.5$ & $2.2\pm0.4$ & $\boldsymbol{1.4\pm0.5}$ \\
  $T=48$      & $42.0\pm14.4$ & $309.7\pm3.6$ & \cellcolor[rgb]{ .816, .808, .808}$23.7\pm21.2$ & $36.4\pm1.1$ & $27.2\pm0.4$ & NI & $\boldsymbol{22.0\pm1.0}$ \\
  $T=96$      & $45.5\pm8.4$ & $1286.3\pm20.7$ & \cellcolor[rgb]{ .816, .808, .808}$\boldsymbol{33.9\pm13.1}$ & $71.2\pm1.1$ & $55.6\pm1.7$ & NI & $47.2\pm1.1$ \\
  $T=360$     & $139.4\pm11.6$ & 4 TL, 1 OOM & \cellcolor[rgb]{ .816, .808, .808}$\boldsymbol{129.2\pm14.3}$ & $429.2\pm20.2$ & $228.2\pm10.4$ & NI & $181.2\pm5.2$ \\
  $T=1{,}440$ & \cellcolor[rgb]{ .816, .808, .808}$\boldsymbol{532.2\pm51.4}$ & TL & $638.5\pm10.9$ & TL & TL & TL & TL \\
  $T=2{,}160$ & \cellcolor[rgb]{ .816, .808, .808}$\boldsymbol{823.5\pm141.7}$ & TL & OOM & TL\textsuperscript{$\S$} & TL\textsuperscript{$\S$} & TL\textsuperscript{$\S$} & TL\textsuperscript{$\S$} \\
  $T=3{,}600$ & \cellcolor[rgb]{ .816, .808, .808}$\boldsymbol{1332.0\pm243.2}$ & TL & OOM\textsuperscript{$\S$} & TL\textsuperscript{$\S$} & TL\textsuperscript{$\S$} & TL\textsuperscript{$\S$} & TL\textsuperscript{$\S$} \\
  \midrule
  SGM(10)     & \cellcolor[rgb]{ .816, .808, .808}\textbf{172.0} & 2469.1 & 351.9 & 731.7 & 619.8 & 6339.6 & 574.4 \\
  Solved      & 35/35                & 15/35                & 25/35              & 20/35              & 20/35             & 5/35           & 20/35 \\
  \bottomrule
  \end{tabular}}\par
  \vspace{5pt}
  \begin{minipage}{0.96\linewidth}
  \footnotesize\raggedright
  {\textit{Note}. See Table~\ref{tab:fisher_res_repeated} for the reporting conventions, outcome abbreviations, and SGM(10) treatment.}
  \end{minipage}
  \vspace{-15pt}
\end{table}%

{As $T$ increases, the numbers of constraints, variables, and cones grow proportionally. Table~\ref{tab:mpo_res_repeated} shows that MOSEK* is the fastest for $T=3$ and $48$, and CuClarabel is the fastest for $T=96$ and $360$. cuPDCS is the fastest from $T=1{,}440$, it is the only solver that solves all 35 instances, and it has the lowest SGM(10).

Among the matrix-free methods, cuPDCS is much faster and more robust than SCS(GPU) at moderate and long horizons. Presolve speeds up MOSEK at shorter horizons, but MOSEK* reaches the time limit at $T=1{,}440$ and is not run on longer horizons. COPT* encounters numerical issues for $T=48$, $96$, and $360$. For $T\geq2{,}160$, the other solvers either fail or are not run because they already failed at a shorter horizon, while cuPDCS solves every instance.}

\begingroup
\subsection{Overall Performance Summary}

We summarize the three application classes by Dolan--Mor\'{e} performance profiles~\citep{DolanMore2002}. For each problem class, let $\mathcal P$ be the set of test instances and $\mathcal S$ the set of solvers. For each
successful run, let $t_{p,s}$ be the wall-clock time of solver $s$ on instance
$p$, and let $t_p^\star$ be the shortest such time over all solvers. We set
$r_{p,s}=t_{p,s}/t_p^\star$ if solver $s$ solves $p$, and $r_{p,s}=+\infty$
otherwise. For $\tau\geq1$, the profile
$\rho_s(\tau)=|\{p\in\mathcal P:r_{p,s}\leq\tau\}|/|\mathcal P|$ is the fraction
of the test set that solver $s$ solves within $\tau$ times the best
time. Hence, $\rho_s(1)$ is the fraction on which $s$ is fastest, while the final height is its solved fraction. Since the slowest solvers stretch the horizontal axes, the profiles mainly compare robustness, and Tables~\ref{tab:fisher_res_repeated}--\ref{tab:mpo_res_repeated} compare the running times at each size.

\begin{figure}[htbp]
    \centering
    \includegraphics[width=0.80\textwidth]{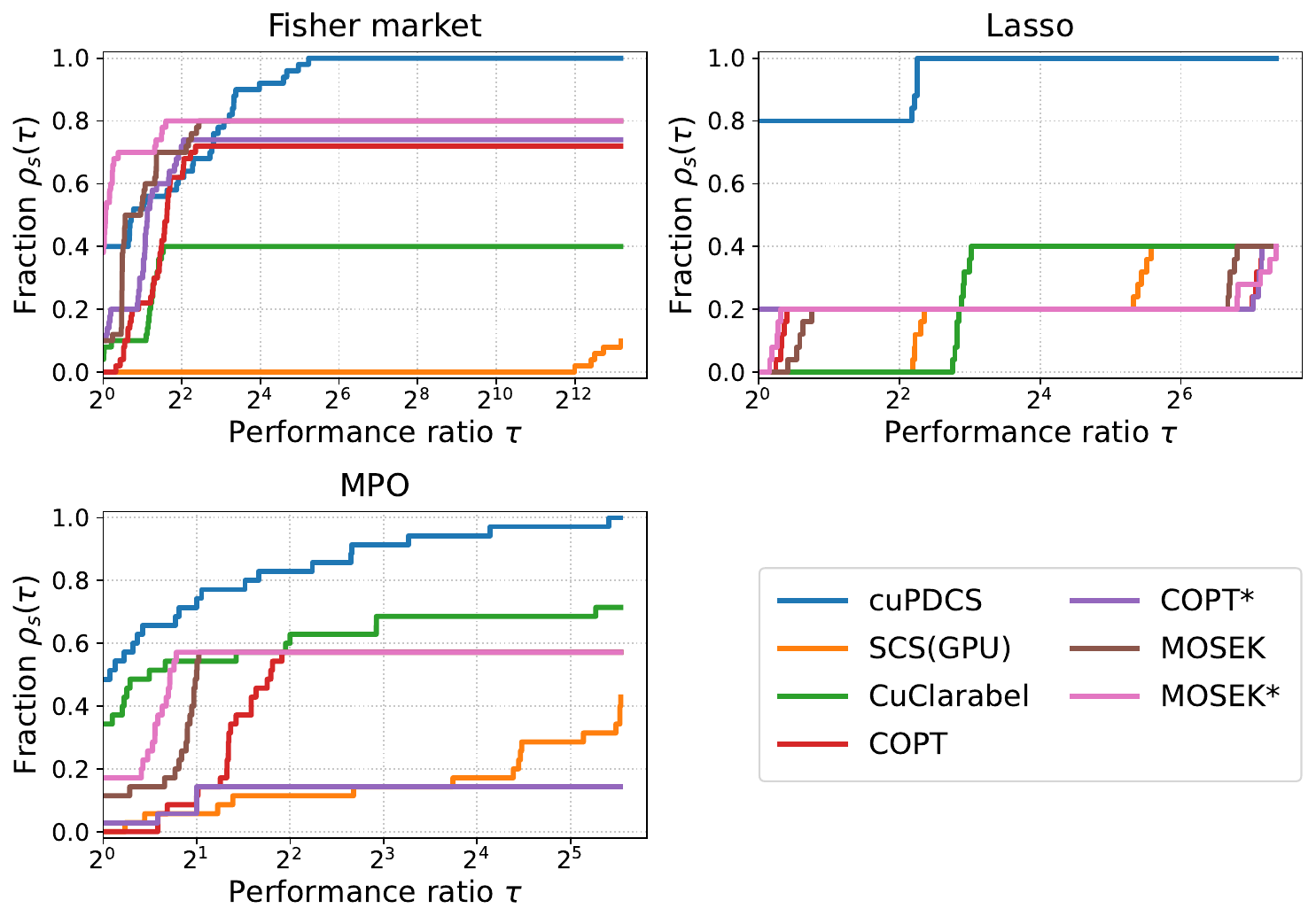}
    \caption{{Performance profiles for the Fisher market, Lasso, and MPO
    experiments. }}
    \label{fig:performance_profiles_combined}
\end{figure}

Figure~\ref{fig:performance_profiles_combined} shows that cuPDCS is the only solver whose curve reaches one in every panel, that is, the only solver that solves every instance of the three problem classes. Its advantage grows with the problem size, and it is the only solver that solves the largest instances of each class (Tables~\ref{tab:fisher_res_repeated}--\ref{tab:mpo_res_repeated}).
\section{Conclusions}\label{sec:conclusion}

We develop PDCS, a matrix-free first-order method for conic programs, and its GPU implementation cuPDCS. PDCS adds practical enhancements and specialized projections onto diagonally rescaled cones to reflected restarted Halpern PDHG with averaging, whose convergence on general conic programs we prove, and cuPDCS exploits the parallel structure of its matrix-vector multiplications and projections. Experiments on CBLIB and three application classes show that the advantage of cuPDCS grows with the problem size. For small- and medium-scale problems, particularly when high accuracy is required, interior-point solvers remain the preferred choice. cuPDCS complements these solvers on very large problems at moderate accuracy, such as the $10^{-6}$ tolerance used in our experiments, where its matrix-free GPU implementation offers better scalability.
The benefit of presolve depends on the problem structure. It substantially improves some MOSEK and COPT runs. A presolve procedure that runs on the GPU is a direction for future work for cuPDCS, which currently operates directly on the original formulation.
\endgroup

\ACKNOWLEDGMENT{The authors thank Haihao Lu for sharing their progress on developing a large-scale SOCP solver during the final preparation of this manuscript. The authors also thank the editors and the three anonymous referees for their very helpful suggestions for improving the quality of the manuscript.}

\def\bibsection{\section*{References}} 
\begingroup
    \renewcommand{\bibfont}{\footnotesize} 
    \bibliographystyle{informs2014}
    \bibliography{ref}

\begin{thebibliography}{68}
\providecommand{\natexlab}[1]{#1}
\providecommand{\url}[1]{\texttt{#1}}
\providecommand{\urlprefix}{URL }

\bibitem[{Applegate et~al.(2021)Applegate, D{\'\i}az, Hinder, Lu, Lubin,
  O'Donoghue, \protect\BIBand{} Schudy}]{applegate2021practical}
Applegate D, D{\'\i}az M, Hinder O, Lu H, Lubin M, O'Donoghue B, Schudy W
  (2021) Practical large-scale linear programming using primal-dual hybrid
  gradient. \emph{Advances in Neural Information Processing Systems},
  volume~34, 20243--20257.

\bibitem[{Applegate et~al.(2023)Applegate, Hinder, Lu, \protect\BIBand{}
  Lubin}]{applegate2023faster}
Applegate D, Hinder O, Lu H, Lubin M (2023) Faster first-order primal-dual
  methods for linear programming using restarts and sharpness.
  \emph{Mathematical Programming} 201(1):133--184.

\bibitem[{Bandi et~al.(2019)Bandi, Trichakis, \protect\BIBand{}
  Vayanos}]{bandi2019robust}
Bandi C, Trichakis N, Vayanos P (2019) Robust multiclass queuing theory for
  wait time estimation in resource allocation systems. \emph{Management
  Science} 65(1):152--187.

\bibitem[{Bell \protect\BIBand{} Garland(2008)}]{bell2008efficient}
Bell N, Garland M (2008) Efficient sparse matrix-vector multiplication on
  {CUDA}. Technical Report NVR-2008-004, NVIDIA Corporation.

\bibitem[{Biele \protect\BIBand{} Gade(2024)}]{xpressnwes}
Biele A, Gade D (2024) {FICO}{\textregistered} {Xpress} {Solver} 9.4.
  \urlprefix\url{https://community.fico.com/s/blog-post/a5QQi0000019II5MAM/fico4824},
  accessed: 2025-10-07.

\bibitem[{Bowman(1956)}]{bowman1956production}
Bowman EH (1956) Production scheduling by the transportation method of linear
  programming. \emph{Operations Research} 4(1):100--103.

\bibitem[{Boyd et~al.(2017)Boyd, Busseti, Diamond, Kahn, Nystrup,
  \protect\BIBand{} Speth}]{boyd2017multi}
Boyd S, Busseti E, Diamond S, Kahn R, Nystrup P, Speth J (2017) Multi--period
  trading via convex optimization. \emph{Foundations and Trends in
  Optimization} 3(1):1--76.

\bibitem[{Chambolle \protect\BIBand{} Pock(2011)}]{chambolle2011first}
Chambolle A, Pock T (2011) A first-order primal-dual algorithm for convex
  problems with applications to imaging. \emph{Journal of Mathematical Imaging
  and Vision} 40:120--145.

\bibitem[{Chambolle \protect\BIBand{} Pock(2016)}]{chambolle2016ergodic}
Chambolle A, Pock T (2016) On the ergodic convergence rates of a first-order
  primal--dual algorithm. \emph{Mathematical Programming} 159(1--2):253--287.

\bibitem[{Charnes \protect\BIBand{} Cooper(1954)}]{charnes1954stepping}
Charnes A, Cooper WW (1954) The stepping stone method of explaining linear
  programming calculations in transportation problems. \emph{Management
  Science} 1(1):49--69.

\bibitem[{Chen et~al.(2026{\natexlab{a}})Chen, Sun, Yuan, Zhang,
  \protect\BIBand{} Zhao}]{chen2024hpr}
Chen K, Sun D, Yuan Y, Zhang G, Zhao X (2026{\natexlab{a}}) {HPR-LP}: An
  implementation of an {HPR} method for solving linear programming.
  \emph{Mathematical Programming Computation} 18(1):183--210.

\bibitem[{Chen et~al.(2024)Chen, He, \protect\BIBand{}
  Zhou}]{chen2024exponential}
Chen L, He L, Zhou Y (2024) An exponential cone programming approach for
  managing electric vehicle charging. \emph{Operations Research}
  72(5):2215--2240.

\bibitem[{Chen et~al.(2026{\natexlab{b}})Chen, Tse, Nobel, Goulart,
  \protect\BIBand{} Boyd}]{CuClarabel}
Chen Y, Tse D, Nobel P, Goulart P, Boyd S (2026{\natexlab{b}}) {CuClarabel}:
  {GPU} acceleration for a conic optimization solver. \emph{ACM Transactions on
  Mathematical Software} 52(3):1--25.

\bibitem[{Cormen et~al.(2022)Cormen, Leiserson, Rivest, \protect\BIBand{}
  Stein}]{cormen2022introduction}
Cormen TH, Leiserson CE, Rivest RL, Stein C (2022) \emph{Introduction to
  algorithms} (MIT Press), 4th edition.

\bibitem[{Deng et~al.(2025)Deng, Feng, Gao, Ge, Jiang, Jiang, Liu, Liu, Xue,
  Ye, \protect\BIBand{} Zhang}]{deng2024enhanced}
Deng Q, Feng Q, Gao W, Ge D, Jiang B, Jiang Y, Liu J, Liu T, Xue C, Ye Y, Zhang
  C (2025) An enhanced alternating direction method of multipliers-based
  interior point method for linear and conic optimization. \emph{INFORMS
  Journal on Computing} 37(2):338--359.

\bibitem[{Diamond \protect\BIBand{} Boyd(2016)}]{diamond2016cvxpy}
Diamond S, Boyd S (2016) {CVXPY}: A {Python}-embedded modeling language for
  convex optimization. \emph{Journal of Machine Learning Research} 17(83):1--5.

\bibitem[{Ding et~al.(2026)Ding, Lu, \protect\BIBand{} Yang}]{ding2026local}
Ding L, Lu H, Yang J (2026) On the local linear convergence of operator
  splitting methods for conic programming. \emph{arXiv preprint
  arXiv:2608.16054} .

\bibitem[{Dolan \protect\BIBand{} Mor\'{e}(2002)}]{DolanMore2002}
Dolan ED, Mor\'{e} JJ (2002) Benchmarking optimization software with
  performance profiles. \emph{Mathematical Programming} 91(2):201--213.

\bibitem[{Eisenberg \protect\BIBand{} Gale(1959)}]{eisenberg1959consensus}
Eisenberg E, Gale D (1959) Consensus of subjective probabilities: The
  pari-mutuel method. \emph{The Annals of Mathematical Statistics}
  30(1):165--168.

\bibitem[{Esser et~al.(2010)Esser, Zhang, \protect\BIBand{}
  Chan}]{esser2010general}
Esser E, Zhang X, Chan TF (2010) A general framework for a class of first order
  primal-dual algorithms for convex optimization in imaging science. \emph{SIAM
  Journal on Imaging Sciences} 3(4):1015--1046.

\bibitem[{Fender(2024)}]{nvdianews}
Fender A (2024) Advances in optimization {AI}.
  \urlprefix\url{https://resources.nvidia.com/en-us-ai-optimization-content/gtc24-s62495},
  accessed: 2025-10-07.

\bibitem[{Friberg(2016)}]{friberg2016cblib}
Friberg HA (2016) {CBLIB} 2014: {A} benchmark library for conic mixed-integer
  and continuous optimization. \emph{Mathematical Programming Computation}
  8:191--214.

\bibitem[{Friberg(2023)}]{friberg2023projection}
Friberg HA (2023) Projection onto the exponential cone: {A} univariate
  root-finding problem. \emph{Optimization Methods and Software}
  38(3):457--473.

\bibitem[{Ge et~al.(2022)Ge, Huangfu, Wang, Wu, \protect\BIBand{}
  Ye}]{ge2022cardinal}
Ge D, Huangfu Q, Wang Z, Wu J, Ye Y (2022) {C}ardinal {O}ptimizer ({COPT}) user
  guide. \emph{arXiv preprint arXiv:2208.14314} .

\bibitem[{Glockner(2023)}]{gurobi_parallel_distributed_optimization}
Glockner G (2023) {P}arallel and distributed optimization with {Gurobi}.
  \urlprefix\url{https://www.gurobi.com/events/parallel-and-distributed-optimization-with-gurobi/},
  accessed: 2025-10-07.

\bibitem[{Greene(2017)}]{greene2003econometric}
Greene WH (2017) \emph{Econometric analysis} (Pearson), 7th edition.

\bibitem[{Halpern(1967)}]{halpern1967fixed}
Halpern B (1967) Fixed points of nonexpanding maps. \emph{Bulletin of the
  American Mathematical Society} 73(6):957--961.

\bibitem[{Han et~al.(2026{\natexlab{a}})Han, Li, Lin, Chen, Deng, Ge, Liu,
  \protect\BIBand{} Ye}]{han2024low}
Han Q, Li C, Lin Z, Chen C, Deng Q, Ge D, Liu H, Ye Y (2026{\natexlab{a}}) A
  low-rank {ADMM} splitting approach for semidefinite programming.
  \emph{INFORMS Journal on Computing} 38(5):1437--1453.

\bibitem[{Han et~al.(2026{\natexlab{b}})Han, Lin, Liu, Chen, Deng, Ge,
  \protect\BIBand{} Ye}]{han2024accelerating}
Han Q, Lin Z, Liu H, Chen C, Deng Q, Ge D, Ye Y (2026{\natexlab{b}})
  Large-scale semidefinite programming with graphics processing units.
  \emph{Proceedings of the National Academy of Sciences} 123(40):e2516128123.

\bibitem[{Hanssmann \protect\BIBand{} Hess(1960)}]{hanssmann1960linear}
Hanssmann F, Hess SW (1960) A linear programming approach to production and
  employment scheduling. \emph{Management Science} MT-1(1):46--51.

\bibitem[{Hinder(2024)}]{hinder2024worst}
Hinder O (2024) Worst-case analysis of restarted primal-dual hybrid gradient on
  totally unimodular linear programs. \emph{Operations Research Letters}
  57:107199.

\bibitem[{Hoffman(1952)}]{hoffman1952approximate}
Hoffman AJ (1952) On approximate solutions of systems of linear inequalities.
  \emph{Journal of Research of the National Bureau of Standards}
  49(4):263--265.

\bibitem[{Huang et~al.(2025)Huang, Zhang, Li, Ge, Liu, \protect\BIBand{}
  Ye}]{huang2024restarted}
Huang Y, Zhang W, Li H, Ge D, Liu H, Ye Y (2025) A restarted primal-dual hybrid
  conjugate gradient method for large-scale quadratic programming.
  \emph{INFORMS Journal on Computing}
  \urlprefix\url{http://dx.doi.org/10.1287/ijoc.2024.0983}.

\bibitem[{Huangfu \protect\BIBand{} Hall(2018)}]{huangfu2018parallelizing}
Huangfu Q, Hall JAJ (2018) Parallelizing the dual revised simplex method.
  \emph{Mathematical Programming Computation} 10(1):119--142.

\bibitem[{Kang et~al.(2025{\natexlab{a}})Kang, Han, Groudiev, \protect\BIBand{}
  Yang}]{kang2025factorization}
Kang S, Han H, Groudiev A, Yang H (2025{\natexlab{a}}) Factorization-free
  orthogonal projection onto the positive semidefinite cone with composite
  polynomial filtering. \emph{arXiv preprint arXiv:2507.09165} .

\bibitem[{Kang et~al.(2025{\natexlab{b}})Kang, Jiang, \protect\BIBand{}
  Yang}]{kang2025local}
Kang S, Jiang X, Yang H (2025{\natexlab{b}}) Local linear convergence of the
  alternating direction method of multipliers for semidefinite programming
  under strict complementarity. \emph{arXiv preprint arXiv:2503.20142} .

\bibitem[{Lieder(2021)}]{lieder2021convergence}
Lieder F (2021) On the convergence rate of the {H}alpern-iteration.
  \emph{Optimization Letters} 15(2):405--418.

\bibitem[{Lin et~al.(2021)Lin, Ma, Ye, \protect\BIBand{} Zhang}]{lin2021admm}
Lin T, Ma S, Ye Y, Zhang S (2021) {A}n {ADMM}-based interior-point method for
  large-scale linear programming. \emph{Optimization Methods and Software}
  36(2--3):389--424.

\bibitem[{Lin et~al.(2026)Lin, Xiong, Ge, \protect\BIBand{} Ye}]{alg_detail}
Lin Z, Xiong Z, Ge D, Ye Y (2026) A technical note on the implementation and
  use of {PDCS}. \emph{arXiv preprint arXiv:2603.15504} .

\bibitem[{Lu et~al.(2025)Lu, Peng, \protect\BIBand{} Yang}]{lu2025cupdlpx}
Lu H, Peng Z, Yang J (2025) cu{PDLPx}: A further enhanced {GPU}-based
  first-order solver for linear programming. \emph{arXiv preprint
  arXiv:2507.14051} .

\bibitem[{Lu \protect\BIBand{} Yang(2022)}]{lu2022infimal}
Lu H, Yang J (2022) On the infimal sub-differential size of primal-dual hybrid
  gradient method and beyond. \emph{arXiv preprint arXiv:2206.12061} .

\bibitem[{Lu \protect\BIBand{} Yang(2024)}]{lu2024restarted}
Lu H, Yang J (2024) Restarted {H}alpern {PDHG} for linear programming.
  \emph{arXiv preprint arXiv:2407.16144} .

\bibitem[{Lu \protect\BIBand{} Yang(2025{\natexlab{a}})}]{lu2023cupdlp}
Lu H, Yang J (2025{\natexlab{a}}) {cuPDLP.jl}: A {GPU} implementation of
  restarted primal-dual hybrid gradient for linear programming in {Julia}.
  \emph{Operations Research} 73(6):3440--3452.

\bibitem[{Lu \protect\BIBand{} Yang(2025{\natexlab{b}})}]{lu2024geometry}
Lu H, Yang J (2025{\natexlab{b}}) On the geometry and refined rate of
  primal--dual hybrid gradient for linear programming. \emph{Mathematical
  Programming} 212(1--2):349--387.

\bibitem[{Lu \protect\BIBand{} Yang(2026)}]{lu2023practical}
Lu H, Yang J (2026) A practical and optimal first-order method for large-scale
  convex quadratic programming. \emph{Mathematical Programming} 215:771--808.

\bibitem[{Lu et~al.(2023)Lu, Yang, Hu, Huangfu, Liu, Liu, Ye, Zhang,
  \protect\BIBand{} Ge}]{lu2023cupdlp-c}
Lu H, Yang J, Hu H, Huangfu Q, Liu J, Liu T, Ye Y, Zhang C, Ge D (2023)
  {cuPDLP-C}: A strengthened implementation of {cuPDLP} for linear programming
  by {C} language. \emph{arXiv preprint arXiv:2312.14832} .

\bibitem[{Lubin \protect\BIBand{} Dunning(2015)}]{LubinDunningIJOC}
Lubin M, Dunning I (2015) {C}omputing in operations research using {J}ulia.
  \emph{INFORMS Journal on Computing} 27(2):238--248.

\bibitem[{Mak et~al.(2015)Mak, Rong, \protect\BIBand{}
  Zhang}]{mak2015appointment}
Mak HY, Rong Y, Zhang J (2015) Appointment scheduling with limited
  distributional information. \emph{Management Science} 61(2):316--334.

\bibitem[{Mittelmann(2025)}]{sgm}
Mittelmann HD (2025) Benchmarks for optimization software.
  \urlprefix\url{https://plato.asu.edu/}, accessed: 2025-10-07.

\bibitem[{Moreau(1962)}]{moreau1962decomposition}
Moreau JJ (1962) D{\'e}composition orthogonale d'un espace hilbertien selon
  deux c{\^o}nes mutuellement polaires. \emph{Comptes Rendus Hebdomadaires des
  S{\'e}ances de l'Acad{\'e}mie des Sciences} 255:238--240.

\bibitem[{{MOSEK ApS}(2024)}]{MOSEK}
{MOSEK ApS} (2024) The {MOSEK} {Optimization} {Suite}, version 10.0.
  \urlprefix\url{https://docs.mosek.com/}, accessed: 2025-10-07.

\bibitem[{{NVIDIA Corporation}(2025)}]{NVIDIA_cuDSS}
{NVIDIA Corporation} (2025) {NVIDIA} {cuDSS}.
  \urlprefix\url{https://developer.nvidia.com/cudss}, accessed: 2025-10-07.

\bibitem[{O'Donoghue(2021)}]{o2021operator}
O'Donoghue B (2021) Operator splitting for a homogeneous embedding of the
  linear complementarity problem. \emph{SIAM Journal on Optimization}
  31(3):1999--2023.

\bibitem[{O'Donoghue et~al.(2016)O'Donoghue, Chu, Parikh, \protect\BIBand{}
  Boyd}]{o2016conic}
O'Donoghue B, Chu E, Parikh N, Boyd S (2016) Conic optimization via operator
  splitting and homogeneous self-dual embedding. \emph{Journal of Optimization
  Theory and Applications} 169:1042--1068.

\bibitem[{Parikh \protect\BIBand{} Boyd(2014)}]{parikh2014proximal}
Parikh N, Boyd S (2014) Proximal algorithms. \emph{Foundations and Trends in
  Optimization} 1(3):127--239.

\bibitem[{Pock et~al.(2009)Pock, Cremers, Bischof, \protect\BIBand{}
  Chambolle}]{pock2009algorithm}
Pock T, Cremers D, Bischof H, Chambolle A (2009) An algorithm for minimizing
  the {M}umford--{S}hah functional. \emph{Proceedings of the IEEE International
  Conference on Computer Vision}, 1133--1140 (IEEE).

\bibitem[{Rothberg(2024)}]{gurobinews}
Rothberg E (2024) {New} options for solving giant {LPs}.
  \urlprefix\url{https://www.gurobi.com/events/new-options-for-solving-giant-lps/},
  accessed: 2024-10-07.

\bibitem[{Rujeerapaiboon et~al.(2016)Rujeerapaiboon, Kuhn, \protect\BIBand{}
  Wiesemann}]{rujeerapaiboon2016robust}
Rujeerapaiboon N, Kuhn D, Wiesemann W (2016) Robust growth-optimal portfolios.
  \emph{Management Science} 62(7):2090--2109.

\bibitem[{Shmyrev(2009)}]{shmyrev2009algorithm}
Shmyrev VI (2009) An algorithm for finding equilibrium in the linear exchange
  model with fixed budgets. \emph{Journal of Applied and Industrial
  Mathematics} 3:505--518.

\bibitem[{Stellato et~al.(2020)Stellato, Banjac, Goulart, Bemporad,
  \protect\BIBand{} Boyd}]{stellato2020osqp}
Stellato B, Banjac G, Goulart P, Bemporad A, Boyd S (2020) {OSQP}: An operator
  splitting solver for quadratic programs. \emph{Mathematical Programming
  Computation} 12(4):637--672.

\bibitem[{Swirydowicz et~al.(2022)Swirydowicz, Darve, Jones, Maack, Regev,
  Saunders, Thomas, \protect\BIBand{} Peles}]{swirydowicz2022linear}
Swirydowicz K, Darve E, Jones W, Maack J, Regev S, Saunders MA, Thomas SJ,
  Peles S (2022) Linear solvers for power grid optimization problems: {A}
  review of {GPU}-accelerated linear solvers. \emph{Parallel Computing}
  111:102870.

\bibitem[{Wagner et~al.(2004)Wagner, Meller, \protect\BIBand{}
  Elber}]{wagner2004large}
Wagner M, Meller J, Elber R (2004) Large-scale linear programming techniques
  for the design of protein folding potentials. \emph{Mathematical Programming}
  101:301--318.

\bibitem[{Xiong(2026{\natexlab{a}})}]{xiong2024accessible}
Xiong Z (2026{\natexlab{a}}) Accessible complexity bounds for restarted {PDHG}
  on linear programs with a unique optimizer. \emph{Mathematics of Operations
  Research} \urlprefix\url{http://dx.doi.org/10.1287/moor.2024.0738}.

\bibitem[{Xiong(2026{\natexlab{b}})}]{xiong2025high}
Xiong Z (2026{\natexlab{b}}) High-probability polynomial-time complexity of
  restarted primal-dual hybrid gradient method for linear programming.
  \emph{Mathematics of Operations Research}
  \urlprefix\url{http://dx.doi.org/10.1287/moor.2024.0853}.

\bibitem[{Xiong \protect\BIBand{} Freund(2023)}]{xiong2023relation}
Xiong Z, Freund RM (2023) On the relation between {LP} sharpness and limiting
  error ratio and complexity implications for restarted {PDHG}. \emph{arXiv
  preprint arXiv:2312.13773} .

\bibitem[{Xiong \protect\BIBand{} Freund(2024)}]{xiong2024role}
Xiong Z, Freund RM (2024) {T}he role of level-set geometry on the performance
  of {PDHG} for conic linear optimization. \emph{arXiv preprint
  arXiv:2406.01942} .

\bibitem[{Xiong \protect\BIBand{} Freund(2026)}]{xiong2023computational}
Xiong Z, Freund RM (2026) Computational guarantees for restarted {PDHG} for
  {LP} based on ``limiting error ratios'' and {LP} sharpness.
  \emph{Mathematical Programming}
  \urlprefix\url{http://dx.doi.org/10.1007/s10107-026-02358-w}.

\bibitem[{Ye(2008)}]{ye2008path}
Ye Y (2008) A path to the {Arrow--Debreu} competitive market equilibrium.
  \emph{Mathematical Programming} 111(1):315--348.

\end{thebibliography}
\endgroup

\clearpage
\noindent{\small\bfseries Online Appendix}\par
\vspace{10pt}
\begin{center}
{\large\bfseries A Practical GPU-Enhanced Matrix-Free Primal-Dual Method for Large-Scale Conic Programs\par}
\end{center}
\vspace{6pt}

\begin{APPENDICES} 
\begingroup
\renewcommand{\theHsection}{theory.\arabic{section}}
\renewcommand{\theHsubsection}{\theHsection.\arabic{subsection}}
\section{\label{app:theory}Proofs of the convergence results}
This appendix proves the results stated in Section~\ref{sec:theoretical_foundation}. Throughout, $\mcal Z^*$ is the nonempty set of saddle points of~\eqref{eq:pd_sf}, $\mcal F$ is the KKT mapping of Section~\ref{sec:preliminaries}, defined for $\zbf=(\xbf,\ybf)$ with $\xbf\in\mcal K$, $\Mbf$ and $\|\cdot\|_{\Mbf}$ are as defined in Section~\ref{sec:theoretical_foundation}, $\operatorname{dist}_{\Mbf}(\zbf,\mcal Z^*):=\min_{\zbf^*\in\mcal Z^*}\|\zbf-\zbf^*\|_{\Mbf}$, and $T:=\texttt{OnePDHG}$, so that the fixed-point residual of $\zbf$ is $\|\zbf-T\zbf\|_{\Mbf}$. Since $\Mbf$ is positive definite, $\|\Mbf\vbf\|\le\sqrt{\|\Mbf\|}\|\vbf\|_{\Mbf}$ and $\|\vbf\|_{\Mbf}\le\sqrt{\|\Mbf\|}\|\vbf\|$ for every $\vbf$, and in particular $\operatorname{dist}_{\Mbf}(\zbf,\mcal Z^*)\le\sqrt{\|\Mbf\|}\operatorname{dist}(\zbf,\mcal Z^*)$. The symbol $\zbf^*$ denotes an arbitrary point of $\mcal Z^*$, and statements involving $\zbf^*$ hold for every such point. The proofs combine the Halpern estimate of \citet{lieder2021convergence} with the restart argument of \citet{lu2024restarted}, who treat $\alpha\in\{0,1\}$ for LP. We carry them out for $\alpha\in[0,1]$ and the averaged restart candidate.

\subsection{Properties of PDHG}
The following lemma expresses PDHG as a proximal point method in the metric $\Mbf$ \citep{chambolle2016ergodic}. It contains the properties of \texttt{OnePDHG} used in the proof of Theorem~\ref{thm:theory}.

\begin{lemma}\label{lem:theory_fixed_point}
For every $\zbf$, $\Mbf(\zbf-T\zbf)\in\mcal F(T\zbf)$, and $T$ is firmly nonexpansive in $\|\cdot\|_{\Mbf}$, that is, $\|T\zbf-T\zbf'\|_{\Mbf}^2+\|(\zbf-T\zbf)-(\zbf'-T\zbf')\|_{\Mbf}^2\le\|\zbf-\zbf'\|_{\Mbf}^2$ for all $\zbf$ and $\zbf'$. Consequently, the following hold.
\begin{itemize}
\item[(i)] $2T-\Ibf$ is nonexpansive in $\|\cdot\|_{\Mbf}$.
\item[(ii)] The fixed points of $T$ are exactly the points of $\mcal Z^*$, and $\|T\zbf-\zbf^*\|_{\Mbf}^2+\|\zbf-T\zbf\|_{\Mbf}^2\le\|\zbf-\zbf^*\|_{\Mbf}^2$, so that $\|\zbf-T\zbf\|_{\Mbf}\le\operatorname{dist}_{\Mbf}(\zbf,\mcal Z^*)$.
\item[(iii)] $\operatorname{dist}(\zerobf,\mcal F(T\zbf))\le\sqrt{\|\Mbf\|}\|\zbf-T\zbf\|_{\Mbf}$.
\end{itemize}
\end{lemma}
\begin{proof}
Write $T\zbf=(\hat\xbf,\hat\ybf)$. The projection in~\eqref{eq:PDHG_sf} gives $\xbf-\tau(\cbf-\Abf^{\top}\ybf)-\hat\xbf\in N_{\mcal K}(\hat\xbf)$, which rearranges to $\tau^{-1}(\xbf-\hat\xbf)+\Abf^{\top}(\ybf-\hat\ybf)\in\cbf-\Abf^{\top}\hat\ybf+N_{\mcal K}(\hat\xbf)$. The dual update rearranges to $\Abf(\xbf-\hat\xbf)+\sigma^{-1}(\ybf-\hat\ybf)=\Abf\hat\xbf-\bbf$. These two relations are the two rows of $\Mbf(\zbf-T\zbf)\in\mcal F(T\zbf)$.

The mapping $\mcal F$ is monotone. Indeed, if $\gbf\in\mcal F(\zbf)$ and $\gbf'\in\mcal F(\zbf')$, the first components of $\gbf$ and $\gbf'$ contain normal vectors $\ubf\in N_{\mcal K}(\xbf)$ and $\ubf'\in N_{\mcal K}(\xbf')$, the terms involving $\Abf$ cancel, and $\inner{\gbf-\gbf'}{\zbf-\zbf'}=\inner{\ubf-\ubf'}{\xbf-\xbf'}$. Writing $\ubf=-\sbf$ and $\ubf'=-\sbf'$ with $\sbf,\sbf'\in\mcal K^*$, $\inner{\sbf}{\xbf}=0$ and $\inner{\sbf'}{\xbf'}=0$, this equals $\inner{\sbf}{\xbf'}+\inner{\sbf'}{\xbf}\ge0$, since $\xbf,\xbf'\in\mcal K$. Let $\abf:=T\zbf-T\zbf'$ and $\dbf:=(\zbf-T\zbf)-(\zbf'-T\zbf')$. Applying monotonicity to $\Mbf(\zbf-T\zbf)\in\mcal F(T\zbf)$ and $\Mbf(\zbf'-T\zbf')\in\mcal F(T\zbf')$ gives $\inner{\abf}{\Mbf\dbf}\ge0$. Since $\zbf-\zbf'=\abf+\dbf$ and $(2T\zbf-\zbf)-(2T\zbf'-\zbf')=\abf-\dbf$,
\[
\|\zbf-\zbf'\|_{\Mbf}^2=\|\abf\|_{\Mbf}^2+\|\dbf\|_{\Mbf}^2+2\inner{\abf}{\Mbf\dbf}=\|\abf-\dbf\|_{\Mbf}^2+4\inner{\abf}{\Mbf\dbf},
\]
which gives the firm nonexpansiveness of $T$ and (i).

For (ii), note that $\zerobf\in\mcal F(\zbf^*)$. Applying monotonicity to this inclusion and to $\Mbf(\zbf^*-T\zbf^*)\in\mcal F(T\zbf^*)$ gives $\inner{\Mbf(\zbf^*-T\zbf^*)}{T\zbf^*-\zbf^*}\ge0$, that is, $\|\zbf^*-T\zbf^*\|_{\Mbf}^2\le0$, so $T\zbf^*=\zbf^*$. Conversely, if $T\zbf=\zbf$, the inclusion gives $\zerobf\in\mcal F(\zbf)$, so $\zbf\in\mcal Z^*$. Taking $\zbf'=\zbf^*$ in the firm nonexpansiveness gives the inequality in (ii). Finally, the inclusion gives $\operatorname{dist}(\zerobf,\mcal F(T\zbf))\le\|\Mbf(\zbf-T\zbf)\|\le\sqrt{\|\Mbf\|}\|\zbf-T\zbf\|_{\Mbf}$, which is (iii).\qedtext
\end{proof}

\subsection{Estimates within one outer iteration}
Fix an outer iteration $t$ of Algorithm~\ref{alg:theory}, and let $S:=(1+\alpha)T-\alpha\Ibf$. By Lemma~\ref{lem:theory_fixed_point}(i), $2T-\Ibf$ is nonexpansive. Hence $S=\frac{1-\alpha}{2}\Ibf+\frac{1+\alpha}{2}(2T-\Ibf)$, a convex combination of $\Ibf$ and $2T-\Ibf$, is nonexpansive. Its fixed points are those of $T$, and $\zbf-S\zbf=(1+\alpha)(\zbf-T\zbf)$. Since $\hat\zbf^{t,k+1}=T\zbf^{t,k}$, Line~\ref{alg_theory:halpern} of Algorithm~\ref{alg:theory} is the Halpern iteration $\zbf^{t,k+1}=\frac{k+1}{k+2}S\zbf^{t,k}+\frac{1}{k+2}\zbf^{t,0}$ for $S$ with anchor $\zbf^{t,0}$.

\begin{lemma}\label{lem:theory_epoch}
For every $k\ge0$ for which $\zbf^{t,k}$ is computed:
\begin{itemize}
\item[(i)] $\|\zbf^{t,k}-\zbf^*\|_{\Mbf}\le\|\zbf^{t,0}-\zbf^*\|_{\Mbf}$, and the same bound holds for $\bar\zbf^{t,k}$ and $\zbf_c^{t,k}$.
\item[(ii)] $\|\zbf^{t,k}-T\zbf^{t,k}\|_{\Mbf}\le\min\Big\{\|\zbf^{t,0}-T\zbf^{t,0}\|_{\Mbf},\ \dfrac{2\operatorname{dist}_{\Mbf}(\zbf^{t,0},\mcal Z^*)}{(1+\alpha)(k+1)}\Big\}$, and the same bound holds for $\zbf_c^{t,k}$.
\end{itemize}
\end{lemma}
\begin{proof}
In this proof, write $\zbf^k:=\zbf^{t,k}$. Part (i) follows by induction from $\|\zbf^{k+1}-\zbf^*\|_{\Mbf}\le\frac{k+1}{k+2}\|S\zbf^{k}-S\zbf^*\|_{\Mbf}+\frac{1}{k+2}\|\zbf^{0}-\zbf^*\|_{\Mbf}$. The average $\bar\zbf^{t,k}$ is a convex combination of $\zbf^1,\dots,\zbf^k$, and $\zbf_c^{t,k}$ is either $\zbf^k$ or $\bar\zbf^{t,k}$.

For (ii), the second bound follows from the estimate of \citet{lieder2021convergence} for the Halpern iteration of a nonexpansive map, which here reads $\|\zbf^k-S\zbf^k\|_{\Mbf}\le\tfrac{2\|\zbf^0-\zbf^*\|_{\Mbf}}{k+1}$, since $\mcal Z^*$ is the set of fixed points of $S$. Choosing $\zbf^*$ as the point of $\mcal Z^*$ closest to $\zbf^0$, so that $\|\zbf^0-\zbf^*\|_{\Mbf}=\operatorname{dist}_{\Mbf}(\zbf^0,\mcal Z^*)$, and using $\zbf^k-S\zbf^k=(1+\alpha)(\zbf^k-T\zbf^k)$ gives the second bound. For the first bound, let $e:=\|\zbf^0-S\zbf^0\|_{\Mbf}$. We first show that $\|\zbf^j-\zbf^{j-1}\|_{\Mbf}\le\tfrac{e}{2}$ for all $j\ge1$. For $j=1$ this holds with equality, since $\zbf^1-\zbf^0=\frac12(S\zbf^0-\zbf^0)$. If it holds up to $j$, then $\|\zbf^{j-1}-\zbf^0\|_{\Mbf}\le\tfrac{(j-1)e}{2}$, hence $\|S\zbf^{j-1}-\zbf^0\|_{\Mbf}\le\|S\zbf^{j-1}-S\zbf^0\|_{\Mbf}+e\le\tfrac{(j+1)e}{2}$, and the identity
\[
\zbf^{j+1}-\zbf^j=\frac{j+1}{j+2}\big(S\zbf^j-S\zbf^{j-1}\big)+\frac{S\zbf^{j-1}-\zbf^0}{(j+1)(j+2)}
\]
gives $\|\zbf^{j+1}-\zbf^j\|_{\Mbf}\le\frac{(j+1)e}{2(j+2)}+\frac{e}{2(j+2)}=\frac e2$. Now let $k\ge1$. Since $\zbf^k-S\zbf^{k-1}=\tfrac{1}{k+1}(\zbf^0-S\zbf^{k-1})$,
\[
\|\zbf^k-S\zbf^k\|_{\Mbf}\le\|\zbf^k-S\zbf^{k-1}\|_{\Mbf}+\|S\zbf^{k-1}-S\zbf^k\|_{\Mbf}\le\frac{\|\zbf^0-S\zbf^{k-1}\|_{\Mbf}}{k+1}+\|\zbf^{k-1}-\zbf^k\|_{\Mbf}\le e.
\]
Dividing by $1+\alpha$ gives $\|\zbf^k-T\zbf^k\|_{\Mbf}\le \|\zbf^0-T\zbf^0\|_{\Mbf}$. Finally, Line~\ref{alg_theory:candidate} ensures $\|\zbf_c^{t,k}-T\zbf_c^{t,k}\|_{\Mbf}\le \|\zbf^k-T\zbf^k\|_{\Mbf}$.\qedtext
\end{proof}

\subsection{\label{app:theory_proof}Proof of Theorem~\ref{thm:theory}}
\begin{proof}[Proof of Theorem~\ref{thm:theory}]
Let $K_t$ be the number of iterations in outer iteration $t$, $N_t:=K_0+\cdots+K_{t-1}$, $s_t:=\|\zbf^{t,0}-T\zbf^{t,0}\|_{\Mbf}$ and $D_0:=\operatorname{dist}_{\Mbf}(\zbf^{0,0},\mcal Z^*)$. If $s_t=0$ for some $t$, then $\zbf^{t,0}\in\mcal Z^*$ by Lemma~\ref{lem:theory_fixed_point}(ii), the inner loop performs no iterations, and the algorithm generates no further iterates. The arguments below then apply to the finitely many earlier outer iterations. We therefore assume $s_t>0$ for all $t$.

We first bound the length of the inner loops. Applying Lemma~\ref{lem:theory_epoch}(i) to $\zbf^{t+1,0}=\zbf_c^{t,K_t}$, with $\zbf^*$ the point of $\mcal Z^*$ closest to $\zbf^{t,0}$, shows that $\operatorname{dist}_{\Mbf}(\zbf^{t,0},\mcal Z^*)$ is nonincreasing in $t$, so it is at most $D_0$. By Lemma~\ref{lem:theory_epoch}(ii), the restart condition $\|\zbf_c^{t,k}-T\zbf_c^{t,k}\|_{\Mbf}\le \beta s_t$ holds once $k\ge1$ and $k+1\ge x_t:=2\operatorname{dist}_{\Mbf}(\zbf^{t,0},\mcal Z^*)/\big((1+\alpha)\beta s_t\big)$. Since $s_t\le\operatorname{dist}_{\Mbf}(\zbf^{t,0},\mcal Z^*)$ by Lemma~\ref{lem:theory_fixed_point}(ii) and $(1+\alpha)\beta<2$, we have $x_t>1$, so the smallest such $k$ equals $\lceil x_t\rceil-1<x_t$. Therefore
\begin{equation}\label{eq:theory_Kt}
K_t<\frac{2\operatorname{dist}_{\Mbf}(\zbf^{t,0},\mcal Z^*)}{(1+\alpha)\beta\,s_t}\le\frac{2D_0}{(1+\alpha)\beta\,s_t}.
\end{equation}

Next, we show that the iterates converge. Each restart gives $s_{t+1}=\|\zbf_c^{t,K_t}-T\zbf_c^{t,K_t}\|_{\Mbf}\le \beta s_t$, so $s_t\to0$. The anchors $\zbf^{t,0}$ are bounded by Lemma~\ref{lem:theory_epoch}(i), so they have a cluster point $\zbf^*$, and $\|\zbf^*-T\zbf^*\|_{\Mbf}=0$ by the continuity of $T$, so $\zbf^*\in\mcal Z^*$ by Lemma~\ref{lem:theory_fixed_point}(ii). By Lemma~\ref{lem:theory_epoch}(i), $\|\zbf^{t,0}-\zbf^*\|_{\Mbf}$ is nonincreasing in $t$, and it tends to zero along a subsequence, so $\zbf^{t,0}\to\zbf^*$. Lemma~\ref{lem:theory_epoch}(i) then gives $\|\zbf^{t,k}-\zbf^*\|_{\Mbf}\le\|\zbf^{t,0}-\zbf^*\|_{\Mbf}$ for all $k$, so all iterates converge to $\zbf^*$.

We now prove the rate in part (a). Since $s_t\le \beta^{t-j}s_j$ for $j<t$, \eqref{eq:theory_Kt} gives
\[
N_ts_t=\sum_{j=0}^{t-1}K_js_t\le\frac{2D_0}{(1+\alpha)\beta}\sum_{j=0}^{t-1}\beta^{t-j}\le\frac{2D_0}{(1+\alpha)(1-\beta)}.
\]
The iterate $\zbf^{t,k}$ is computed after $N=N_t+k$ iterations, and Lemma~\ref{lem:theory_epoch}(ii) gives $\|\zbf^{t,k}-T\zbf^{t,k}\|_{\Mbf}\le s_t$ and $(k+1)\,\|\zbf^{t,k}-T\zbf^{t,k}\|_{\Mbf}\le\tfrac{2D_0}{1+\alpha}$. Since $N+1=N_t+(k+1)$, adding $N_t$ times the first bound to the second gives
\[
(N+1)\,\|\zbf^{t,k}-T\zbf^{t,k}\|_{\Mbf}\le N_ts_t+\frac{2D_0}{1+\alpha}\le\frac{2D_0}{1+\alpha}\Big(\frac{1}{1-\beta}+1\Big)=\frac{2(2-\beta)\,D_0}{(1+\alpha)(1-\beta)}.
\]
Since $\operatorname{dist}(\zerobf,\mcal F(T\zbf^{t,k}))\le\sqrt{\|\Mbf\|}\|\zbf^{t,k}-T\zbf^{t,k}\|_{\Mbf}$ by Lemma~\ref{lem:theory_fixed_point}(iii) and $D_0\le\sqrt{\|\Mbf\|}\operatorname{dist}(\zbf^{0,0},\mcal Z^*)$, this gives~\eqref{eq:theory_sublinear}.

Finally, we prove the iteration count in part (b). Let $t_0$ and $N_{t_0}$ be as in part (b). Such $t_0$ exists because the iterates converge to $\zbf^*$, $T$ is continuous and $T\zbf^*=\zbf^*$. Let $t\ge t_0$ and $\wbf:=T\zbf^{t,0}$, which lies in $\mcal U$. By~\eqref{eq:theory_error_bound} and Lemma~\ref{lem:theory_fixed_point}(iii), $\gamma\operatorname{dist}(\wbf,\mcal Z^*)\le\operatorname{dist}(\zerobf,\mcal F(\wbf))\le\sqrt{\|\Mbf\|}s_t$. Hence
\[
\operatorname{dist}_{\Mbf}(\zbf^{t,0},\mcal Z^*)\le\|\zbf^{t,0}-\wbf\|_{\Mbf}+\operatorname{dist}_{\Mbf}(\wbf,\mcal Z^*)\le s_t+\sqrt{\|\Mbf\|}\operatorname{dist}(\wbf,\mcal Z^*)\le\Big(1+\frac{\|\Mbf\|}{\gamma}\Big)s_t,
\]
and the first inequality in~\eqref{eq:theory_Kt} gives $K_t<L:=c\log(1/\beta)\,(1+\tfrac{\|\Mbf\|}{\gamma})$, by the definition of $c$. Now let $t\ge t_0$, and let $\zbf^{t,k}$ be computed after $N$ iterations. Since $k\le K_t<L$, we have $N-N_{t_0}=K_{t_0}+\cdots+K_{t-1}+k<(t-t_0+1)L$, that is, $t-t_0>\tfrac{N-N_{t_0}}{L}-1$. By Lemma~\ref{lem:theory_fixed_point}(iii), Lemma~\ref{lem:theory_epoch}(ii), $s_{j+1}\le\beta s_j$ for all $j$, and $s_{t_0}\le D_0\le\sqrt{\|\Mbf\|}\operatorname{dist}(\zbf^{0,0},\mcal Z^*)$,
\[
\operatorname{dist}(\zerobf,\mcal F(T\zbf^{t,k}))\le\sqrt{\|\Mbf\|}s_t\le\sqrt{\|\Mbf\|}\beta^{\,t-t_0}s_{t_0}\le\frac{\|\Mbf\|\operatorname{dist}(\zbf^{0,0},\mcal Z^*)}{\beta}\exp\Big(-\frac{\log(1/\beta)}{L}\,(N-N_{t_0})\Big),
\]
and the right-hand side is at most $\epsilon$ when $N$ satisfies~\eqref{eq:theory_linear}, since $L/\log(1/\beta)=c\,(1+\tfrac{\|\Mbf\|}{\gamma})$. Finally, an iterate $\zbf^{t,k}$ with $t<t_0$ has $N\le N_{t_0}$, so if it satisfies~\eqref{eq:theory_linear}, then $\|\Mbf\|\operatorname{dist}(\zbf^{0,0},\mcal Z^*)\le\beta\epsilon$, and the same chain of inequalities with $s_t\le s_0\le D_0$ gives $\operatorname{dist}(\zerobf,\mcal F(T\zbf^{t,k}))\le\|\Mbf\|\operatorname{dist}(\zbf^{0,0},\mcal Z^*)<\epsilon$.\qedtext
\end{proof}

\endgroup

\renewcommand{\theHsection}{app.\Alph{section}}
\renewcommand{\theHsubsection}{\theHsection.\arabic{subsection}}

\begingroup
\section{\label{app:alg_details}Algorithm details}
This appendix describes the algorithmic enhancements used in PDCS. 
\subsection{\label{sec:adaptive_stepsize}Adaptive step size} 

PDHG converges when $\tau\sigma\|\Gbf\|^2<1$ \citep{chambolle2011first}. Because this bound can be overly conservative in practice, PDCS adapts the step sizes 
\begin{equation}\label{eq:adaptive_step_size}
    \tau = \frac{\eta}{\omega},\quad \sigma = \eta \omega,\quad \eta \leq \frac{\norm{\zbf^{t+1}-\zbf^t}_{\omega}^2}{2\abs{(\ybf^{t+1}-\ybf^t)^\top \Gbf (\xbf^{t+1}-\xbf^t)}},
\end{equation}
where $\norm{\zbf}_{\omega} := \sqrt{\omega\norm{\xbf}^2 + \frac{\norm{\ybf}^2}{\omega}}$, and $\omega$ is the primal weight. 
Function~\ref{alg:adaptive_pdhg} updates $\eta$. It returns the accepted step size $\eta$, which weights the average in Line~\ref{alg:line:average} of Algorithm~\ref{alg:PDCS}, and the step size $\eta'$ to be tried next. Here, $\bar{k}$ counts the PDHG steps, including rejected ones, starting from $\bar{k}=1$, and controls how aggressively $\eta$ is adjusted.

{\renewcommand{\algorithmname}{Function}\renewcommand{\thealgorithm}{1}\def\theHalgorithm{function.1}
\begin{algorithm}[h]
\caption{{AdaptiveStepPDHG(($\xbf$,$\ybf$), $\omega$, $\eta$)}\label{alg:adaptive_pdhg}}
    \small\SingleSpacedXI
    \begin{algorithmic}[1]
    \While{True}
    \State{$\hat{\xbf} = \proj_{[\lbf,\ubf] \times \mcal K_{\textrm{p}}}\bcbra{\xbf - \frac{\eta}{\omega}(\cbf-\Gbf^\top \ybf)}$ ,  $\hat{\ybf}=\proj_{\mcal K_{\textrm{d}}}\bcbra{\ybf + \eta \omega(\hbf-\Gbf(2\hat{\xbf}-\xbf))}$}
    \State{$\bar{\eta}\leftarrow \frac{1}{2}{\norm{(\hat{\xbf}-\xbf,\hat{\ybf}-\ybf)}_{\omega}^2}/{\abs{(\hat{\ybf}-\ybf)^\top \Gbf (\hat{\xbf}-\xbf)}}$\ ,\ $\eta^{\prime}\leftarrow \min\bcbra{(1-(\bar{k}+1)^{-0.3})\bar{\eta}, (1+(\bar{k}+1)^{-0.6})\eta}$}
    \State{\textbf{if} $\eta<\bar{\eta}$ \textbf{then} $\bar{k}\leftarrow\bar{k}+1$ \textbf{and return} $(\hat{\xbf},\hat{\ybf}),\eta,\eta^{\prime}$}
    \State{$\eta\leftarrow \eta^{\prime},\bar{k}\leftarrow\bar{k}+1$}
    \EndWhile
    \end{algorithmic}
\end{algorithm}\addtocounter{algorithm}{-1}
}

\subsection{\label{sec:adaptive_reflection}Adaptive reflection parameter}

 Lines~\ref{alg_pdcs:line5} and~\ref{alg:line:ref_Halpern} of Algorithm~\ref{alg:PDCS} perform the reflected Halpern step of Algorithm~\ref{alg:theory} with a reflection parameter $\alpha^{t,k}$ that varies over the iterations.
A larger $\alpha$ may accelerate convergence but can also cause instability. PDCS therefore uses a small $\alpha$ when the relative error $\text{maxErr}$ defined below is large and increases $\alpha$ as this error decreases. It sets $\alpha^{t,k}$ from $\text{maxErr}$ at $\hat{\zbf}^{t,k+1}$ by
\begin{equation}\label{eq:adaptive_reflection}
    \texttt{AdaptiveReflectionParameter}(\zbf)=\proj_{[0,1]}\left(-0.1\cdot\log_{10}\left(\text{maxErr}(\zbf)\right)+0.2\right) \ ,
\end{equation}
where $\text{maxErr}(\cdot)$ measures the overall optimality error and is defined by
\begin{equation}\label{eq:tolerance} 
    \text{maxErr}(\zbf):=\max\{\err_{\textrm{rel,p}}(\xbf),\err_{\textrm{rel,d}}(\bm{\lambda}_1,\bm{\lambda}_2),\err_{\textrm{rel,gap}}(\xbf,\bm{\lambda}_1,\bm{\lambda}_2)\} \ .
\end{equation}
Here $\err_{\textrm{rel,p}}(\xbf),\err_{\textrm{rel,d}}(\bm{\lambda}_1,\bm{\lambda}_2),\err_{\textrm{rel,gap}}(\xbf,\bm{\lambda}_1,\bm{\lambda}_2)$ are relative primal infeasibility, relative dual infeasibility and relative primal-dual gap: 
{\small
\begin{equation}\label{eq:tolerance-3}
\begin{aligned}
\err_{\textrm{abs,p}}(\xbf):=&\left\|{(\Gbf\xbf-\hbf)-\proj_{\mcal K_{\textrm{d}}^{*}}\{\Gbf\xbf-\hbf\}}\right\|,\quad \err_{\textrm{abs,gap}}(\xbf,\bm{\lambda}_{1},\bm{\lambda}_{2}):=\abs{\inner{\cbf}{\xbf}-\brbra{\ybf^{\top}\hbf+\lbf^{\top}\bm{\lambda}_{1}^{+}-\ubf^{\top}\bm{\lambda}_{1}^{-}}},\\
&
 \err_{\textrm{abs,d}}(\bm{\lambda}_{1},\bm{\lambda}_{2}):=\sqrt{\norm{\bm{\lambda}_{1}-\proj_{\Lambda}\{\bm{\lambda}_{1}\}}^2+\norm{\bm{\lambda}_{2}-\proj_{\mcal K_{\textrm{p}}^{*}}\{\bm{\lambda}_{2}\}}^2}.
\end{aligned}
\end{equation}
}
\begin{equation}\label{eq:tolerance-2}
    \begin{aligned}
    \err_{\textrm{rel,p}}(\xbf) : =\tfrac{\err_{\textrm{abs,p}}(\xbf)}{1+\norm{\hbf}_{1}} \ , \ \ 
    \err_{\textrm{rel,d}}(\bm{\lambda}_{1},\bm{\lambda}_{2}) :=\tfrac{\err_{\textrm{abs,d}}(\bm{\lambda}_{1},\bm{\lambda}_{2})}{1+\norm{\cbf}_{1}} \  ,  \err_{\textrm{rel,gap}}(\xbf,\bm{\lambda}_{1},\bm{\lambda}_{2}) : =\tfrac{\err_{\textrm{abs,gap}}(\xbf,\bm{\lambda}_{1},\bm{\lambda}_{2})}{1+\abs{\inner{\cbf}{\xbf}}+\abs{\ybf^{\top}\hbf+\lbf^{\top}\bm{\lambda}_{1}^{+}-\ubf^{\top}\bm{\lambda}_{1}^{-}}} \ .
    \end{aligned}
    \end{equation}
As in~\eqref{eq:PDCS_terminate_tolerance}, terms involving infinite entries of $\lbf$ or $\ubf$ are omitted from~\eqref{eq:tolerance-3} and~\eqref{eq:tolerance-2}.

\subsection{\label{sec:adaptive_restart}Adaptive restart}

To select the restart point, PDCS measures optimality by the normalized duality gap, which is defined as follows.
\begin{definition}[normalized duality gap~\citep{applegate2023faster}]\label{defn:normalized_dual_gap}
   For any $\zbf = (\xbf, \ybf)$, $r > 0$, and step sizes $\tau > 0$, $\sigma > 0$ satisfying $\tau\sigma\|\Gbf\|^2<1$, define $\Mbf:=\bigg(\begin{smallmatrix}
\frac{1}{\tau}\Ibf & \Gbf^{\top}\\
\Gbf & \frac{1}{\sigma}\Ibf
\end{smallmatrix}\bigg)$ and $$\mcal{B} (r;\zbf):=\bcbra{\hat{\zbf}:=(\hat{\xbf},\hat{\ybf}):\hat{\xbf}\in \mcal X, \hat{\ybf}\in \mcal Y,\norm{\zbf-\hat{\zbf}}_{\Mbf}\leq r}\ .$$ The normalized duality gap is $
\rho(r;\zbf):=\frac{1}{r}\sup_{\hat{\zbf}\in \mcal B(r;\zbf)}[\mcal L(\xbf,\hat{\ybf})-\mcal L(\hat{\xbf},\ybf)].$
\end{definition}
Computing $\rho(r;\zbf)$ exactly involves a nontrivial projection onto $\mcal X \times \mcal Y$ under the $\Mbf$-norm. 
Hence, we approximate the gap by replacing $\Mbf$ with $\Nbf:=\operatorname{diag}(\Ibf/\tau,\Ibf/\sigma)$. The resulting gap is
\begin{equation}\label{pro:compute_rho}
\rho^{\Nbf}(r;\zbf):=\frac{1}{r}\max_{\substack{\hat{\zbf}\in\mcal X\times\mcal Y\\\norm{\zbf-\hat{\zbf}}_{\Nbf}^{2}\leq r^{2}}}\bbf^{\top}(\hat{\zbf}-\zbf),
\qquad
\bbf:=\binom{\Gbf^{\top}\ybf-\cbf}{\hbf-\Gbf\xbf}.
\end{equation}
PDCS evaluates $\rho^{\Nbf}$ with the current step sizes, given by~\eqref{eq:adaptive_step_size} with $\eta=\eta^{t,k}$ and $\omega=\omega^t$.

PDCS takes as restart candidate whichever of $\zbf^{t,k+1}$ and $\bar{\zbf}^{t,k+1}$ has the smaller approximate normalized duality gap:
    \begin{equation*}
    \resizebox{0.90\linewidth}{!}{$\displaystyle
    \texttt{GetRestartCandidate}(\zbf^{t,k+1},\bar{\zbf}^{t,k+1})=
    \begin{cases}
    \zbf^{t,k+1} & \rho^{\Nbf}\brbra{\norm{\zbf^{t,k+1}-\zbf^{t,0}}_{\Nbf};\zbf^{t,k+1}}\leq\rho^{\Nbf}\brbra{\norm{\bar{\zbf}^{t,k+1}-\zbf^{t,0}}_{\Nbf};\bar{\zbf}^{t,k+1}},\\
    \bar{\zbf}^{t,k+1} & \text{otherwise}.
    \end{cases}$}
    \end{equation*}
Let $\zbf_{c}^{t,k+1}$ denote the resulting restart candidate. A restart to $\zbf_{c}^{t,k+1}$ is triggered whenever its approximate normalized duality gap satisfies any of the following conditions. Here we use $\beta_{\text{sufficient}} = 0.4$, $\beta_{\text{necessary}} = 0.8$, and $\beta_{\text{artificial}} = 0.223$.  
\begin{enumerate}
    \item (Sufficient decay) $\rho^{\Nbf}\Brbra{\norm{\zbf_{c}^{t,k+1}-\zbf^{t,0}}_{\Nbf};\zbf_{c}^{t,k+1}}\leq\beta_\text{sufficient} \cdot \rho^{\Nbf}\brbra{\norm{\zbf^{t,0}-\zbf^{t-1,0}}_{\Nbf};\zbf^{t,0}}$
    \item (Necessary decay and no local progress) $\rho^{\Nbf}\left(\norm{\zbf_{c}^{t,k+1}-\zbf^{t,0}}_{\Nbf};\zbf_{c}^{t,k+1}\right)>\rho^{\Nbf}\Brbra{\norm{\zbf_{c}^{t,k}-\zbf^{t,0}}_{\Nbf};\zbf_{c}^{t,k}}$
    and $\rho^{\Nbf}\Brbra{\norm{\zbf_{c}^{t,k+1} -\zbf^{t,0}}_{\Nbf};\zbf_{c}^{t,k+1}}\leq\beta_{\text{necessary}} \cdot \rho^{\Nbf}\brbra{\norm{\zbf^{t,0}-\zbf^{t-1,0}}_{\Nbf};\zbf^{t,0}}$
    \item (Long inner loop) $k \geq\beta_{\text{artificial}} \cdot \bar{k}$.
\end{enumerate}

For problems with second-order or exponential cones, floating-point errors may make the value of $\rho^{\Nbf}$ returned by the bisection method negative when the true gap is small. In this case, PDCS replaces $\rho^{\Nbf}(r;\zbf)$ in \texttt{GetRestartCandidate} and in the restart condition by the weighted KKT error
$$\text{KKT}_{\omega}(\zbf)=\sqrt{\omega^{2}\err_{\textrm{abs,p}}(\xbf)^{2}+\frac{1}{\omega^{2}}\err_{\text{\textrm{abs,d}}}(\bm{\lambda}_{1},\bm{\lambda}_{2})^{2}+\err_{\text{\textrm{abs,gap}}}(\xbf,\bm{\lambda}_{1},\bm{\lambda}_{2})^{2}}\ ,
$$
where $\err_{\textrm{abs,p}}, \err_{\textrm{abs,d}}, \err_{\textrm{abs,gap}}$ are as defined in~\eqref{eq:tolerance-3}, and the dual variables $\bm{\lambda}_1, \bm{\lambda}_2$ are recovered from $\ybf$ using the linear relation in~\eqref{eq:dual_conic}.

\subsection{Computation of the approximate normalized duality gap\label{app:normalized_duality_gap}}

We compute the approximate normalized duality gap~\eqref{pro:compute_rho} using the approach of \citet{applegate2023faster,xiong2024role}. For any $t$ satisfying $\|\zbf-\zbf(t)\|_{\Nbf}=r$, the optimizer of~\eqref{pro:compute_rho} is
\begin{equation}\label{pro:compute_zt}
\zbf(t) := \argmax_{\tilde{\zbf} \in \mcal X \times \mcal Y} \quad t \cdot \bbf^{\top}(\tilde{\zbf} - \zbf) - \frac{1}{2} \cdot \norm{\tilde{\zbf} - \zbf}_{\Nbf}^2.
\end{equation}
When a root is bracketed, Algorithm~\ref{alg:binary_search_normalized_gap} finds it by bisection. If no root is bracketed, the argument of \citet[Lemma~A.2]{xiong2024role}, adjusted for the factor $\tfrac12$ in~\eqref{pro:compute_zt}, gives, for exact $\zbf(t)$ and $d_t:=\|\zbf(t)-\zbf\|_{\Nbf}\leq r$,
\[
0\leq\rho^{\Nbf}(r;\zbf)-\frac{\bbf^{\top}(\zbf(t)-\zbf)}{r}
\leq\frac{r^2-d_t^2}{2rt}\leq\frac{r}{2t}.
\]
The doubling search therefore returns the current gap estimate once this bound is at most $\eps$, even if the distance never exceeds $r$. 
\begin{algorithm}[htbp]
\caption{{Bisection search for normalized duality gap $\rho^\Nbf(r;\zbf)$}\label{alg:binary_search_normalized_gap}}
\small\SingleSpacedXI
\begin{algorithmic}[1]
    \Require{Initial bound $t_0>0$, iterate $\zbf\in\mcal X\times\mcal Y$, radius $r>0$, tolerance $\eps>0$}
    \State{$t_{\text{left}}\gets0$, $t_{\text{right}}\gets t_0$}
    \For{$k=0,1,\ldots$} \Comment{Find a bracketing interval}
        \State{$\tilde{\zbf}\gets\zbf(t_k)$, $d_k\gets\norm{\zbf-\tilde{\zbf}}_\Nbf$}
        \If{$d_k\leq r$ \textbf{and} $(r^2-d_k^2)/(2rt_k)\leq\eps$}
            \State{\Return $\bbf^\top(\tilde{\zbf}-\zbf)/r$} \Comment{Gap error bound satisfied}
        \EndIf
        \State{\textbf{if} $d_k>r$ \textbf{then set} $t_{\text{right}}\gets t_k$, set $t_{\text{left}}\gets t_k/2$ \textbf{if} $k>0$, \textbf{and break}}
        \State{$t_{k+1}\gets2t_k$}
    \EndFor
    \State{$\tilde{\zbf}\gets\zbf(t_{\text{left}})$}
    \While{$t_{\text{right}}-t_{\text{left}}>\eps$} \Comment{Refine by bisection}
        \State{$t_{\text{mid}}\gets(t_{\text{left}}+t_{\text{right}})/2$, $\tilde{\zbf}\gets\zbf(t_{\text{mid}})$}
        \State{\textbf{if} $\norm{\zbf-\tilde{\zbf}}_\Nbf<r$ \textbf{then} $t_{\text{left}}\gets t_{\text{mid}}$; \textbf{otherwise} $t_{\text{right}}\gets t_{\text{mid}}$}
    \EndWhile
    \State{\Return $\bbf^\top(\tilde{\zbf}-\zbf)/r$}
\end{algorithmic}
\end{algorithm}

\subsection{\label{sec:weight_adjustment}Primal weight adjustment}

The primal weight $\omega$ balances the primal and dual step sizes $\tau=\tfrac{\eta}{\omega}$ and $\sigma=\eta\omega$ in~\eqref{eq:adaptive_step_size}. Following PDLP~\citep{applegate2021practical}, PDCS adjusts this weight at each restart based on the relative changes in the primal and dual iterates. A well-chosen primal weight can accelerate restarted PDHG for LPs~\citep{xiong2024accessible}. We initialize the weight by
$\texttt{InitializePrimalWeight}(\cbf,\hbf):=\begin{cases}
\frac{\|\cbf\|_2}{\|\hbf\|_2}, & \|\cbf\|_2,\|\hbf\|_2>10^{-10}\\
1 & \text{otherwise}
\end{cases}$. 
At the beginning of restart $t$, we update it by
\begin{equation*}
\resizebox{0.9\linewidth}{!}{$\displaystyle
\texttt{PrimalWeightUpdate}(\zbf^{t,0},\zbf^{t-1,0},\omega^{t-1}):=
\begin{cases}
\exp\Brbra{\theta\log\brbra{\frac{\Delta_{y}^{t}}{\Delta_{x}^{t}}}+\brbra{1-\theta}\log\brbra{\omega^{t-1}}} & \Delta_{x}^{t},\Delta_{y}^{t}>10^{-10}\\
\omega^{t-1} & \text{otherwise}\ ,
\end{cases}$}
\end{equation*} 
where $\Delta_{x}^{t}=\norm{\xbf^{t,0}-\xbf^{t-1,0}}$ and $\Delta_{y}^{t}=\norm{\ybf^{t,0}-\ybf^{t-1,0}}$ measure the primal and dual changes. The parameter $\theta=0.5$ smooths the update. Because very large or very small weights can destabilize the solver, we reset $\omega$ to its initial value whenever $\omega\notin[10^{-5},10^5]$.

\subsection{\label{subsec:Projection-onto-convex}Diagonal rescaling}

As for other first-order methods, the practical performance of PDCS depends strongly on the conditioning of the problem instance. One common strategy is diagonal rescaling, which may improve the condition number of the problem while preserving the sparsity pattern. 
We rescale the constraint matrix $\Gbf \in \mbb{R}^{m \times n}$ to $\hat{\Gbf} = \Dbf_1 \Gbf \Dbf_2$, where $\Dbf_1$ and $\Dbf_2$ are positive diagonal matrices. The corresponding rescaled optimization problem is:
\begin{equation}\label{eq:rescaled_prob}
\min\ \   \inner{\hat{\cbf}}{\tilde{\xbf}} \ , \quad  
\st\ \   \hat{\mathbf{G}}\tilde{\xbf}-\hat{\mathbf{h}}\in\hat{\mathcal{K}}_{\textrm{d}}^{*} \ , \ \  \hat{\lbf}\leq(\tilde{\xbf})_{[n_{1}]}\leq\hat{\ubf} \ , \ \ (\tilde{\xbf})_{[n]\backslash[n_{1}]}\in\hat{\mathcal {K}}_{\textrm{p}} 
\end{equation}
where $\hat{\cbf}:=\Dbf_2\cbf,\hat{\Gbf}:=\Dbf_1\Gbf\Dbf_2,\hat{\hbf}:=\Dbf_1\hbf,\hat{\lbf}:=\Dbf_2^{-1}\lbf$, $\hat{\ubf}:=\Dbf_2^{-1}\ubf$, $\hat{\mcal K}_{\textrm{d}}^*:=\Dbf_1 \mcal K_{\textrm{d}}^*$ and $\hat{\mcal K}_{\textrm{p}}:=\Dbf_2^{-1}\mcal K_{\textrm{p}}$. 
Algorithm~\ref{alg:PDCS_after_precondition} summarizes PDCS with diagonal rescaling.
\begin{algorithm}[h]
\caption{{Primal-Dual Conic Programming Solver (PDCS) with preconditioning}\label{alg:PDCS_after_precondition}}
    \small\SingleSpacedXI
    \begin{algorithmic}[1]
    \State{Compute rescaled instance data $\hat{\cbf},\hat{\Gbf},\hat{\hbf},\hat{\lbf}$ and $\hat{\ubf}$.}
    \State{Use Algorithm~\ref{alg:PDCS} to solve the rescaled problem~\eqref{eq:rescaled_prob} and obtain solution $\tilde{\xbf}$ and $\tilde{\ybf}$.}
    \State{Output solution for \eqref{eq:prob_cp}: $\xbf=\Dbf_2\tilde{\xbf}$ and $\ybf=\Dbf_1 \tilde{\ybf}$.}
    \end{algorithmic}
\end{algorithm}

\endgroup
\section{Projection onto rescaled second-order cone and exponential cone\label{app:proj}}

This appendix proves Theorems \ref{thm:soc_proj} and \ref{thm:exp_proj_simplified}, which characterize the projections onto the rescaled second-order cone $\Dbf \mcal K_{\soc}^{n+1}$ and the rescaled exponential cone $\Dbf \mcal K_{\exp}$.
\subsection{Projection onto rescaled second-order cone\label{app:proj:soc}}

\begin{proof}[Proof of Theorem~\ref{thm:soc_proj}]

    Let $\Dbf$ be the diagonal rescaling matrix. The projection onto $\Dbf \mcal K_{\soc}^{n+1}$ is the solution of the following problem:   
    \begin{equation}\label{eq:proj_diagonal_rescaling}
        \begin{aligned}\proj_{\Dbf \mcal K_{\soc}^{n+1}}\bcbra{(t,\xbf)}=\argmin_{(s,\ybf)\in \mbb R^{n+1}}\frac{1}{2}\norm{(t,\xbf)-(s,\ybf)}^{2},\ \ \st\norm{\hat{\Dbf}^{-1}\ybf}^2\leq s^2, \ s\ge 0\ ,
        \end{aligned}
        \end{equation}
    where $(t,\xbf)\in \mbb R^{n+1}$,  $\hat{\Dbf}=\diag(\hat{\dbf}_2,\cdots,\hat{\dbf}_{n+1})$ with $\hat{\dbf}_i=\dbf_i / \dbf_1$ and $\dbf_i$ is the $i$-th entry in the diagonal of $\Dbf$.

    We first consider the simplest case, in which $(s,\ybf)$ is the degenerate solution $(0,\zerobf)$. The projection onto the convex cone $\Dbf \mcal K_{\soc}^{n+1}$ is $(0,\zerobf)$ if and only if $-(t,\xbf)$ lies in $(\Dbf \mcal K_{\soc}^{n+1})^*$. Because $(\Dbf \mcal K_{\soc}^{n+1})^* = \Dbf^{-1}\mcal K_{\soc}^{n+1}$, it implies that when $\|\hat{\Dbf}\xbf\| \le -t$, the projection $(s,\ybf)$ is $(0,\zerobf)$. This proves the first case.

    Now we introduce the multipliers $\lambda\ge0$ and $\mu\ge0$ for the two constraints of  \eqref{eq:proj_diagonal_rescaling}, then the Lagrangian is $\mathcal L(\ybf,s;\lambda,\mu)    =\tfrac12\|\ybf - \xbf\|^2 + \tfrac12(s - t)^2
    +\lambda\bigl(\|\,\hat \Dbf^{-1}\ybf\|^2 - s^2\bigr)
    +\mu\,( -s )$. Thus, the KKT conditions of \eqref{eq:proj_diagonal_rescaling} are as follows:
    \begin{equation}\label{eq:kkt_socp_1_2_3}
        \|\hat{\Dbf}^{-1}\ybf\|^2 - s^2 \le 0, 
        \ybf - \xbf + 2\lambda \hat{\Dbf}^{-2} \ybf = 0,  (s - t) - 2\lambda s - \mu = 0, 
        \lambda \left( \|\hat{\Dbf}^{-1}\ybf\|^2 - s^2 \right) = 0, \mu s = 0,s \ge 0,  \lambda \ge 0, \mu \ge 0,
    \end{equation}
    {Problem~\eqref{eq:proj_diagonal_rescaling} is convex when its constraints are written as $\|\hat{\Dbf}^{-1}\ybf\|\le s$. At any $(s,\ybf)$ with $s>0$ satisfying~\eqref{eq:kkt_socp_1_2_3}, either $\lambda=0$ or $\|\hat{\Dbf}^{-1}\ybf\|=s$, and in the latter case the gradient of $\|\hat{\Dbf}^{-1}\ybf\|^2-s^2$ is $2s$ times that of $\|\hat{\Dbf}^{-1}\ybf\|-s$. Hence $(s,\ybf)$ satisfies the KKT conditions of this convex problem with multiplier $2\lambda s$ and is optimal for~\eqref{eq:proj_diagonal_rescaling}.}
    
    If $(t,\xbf)$ is already in the rescaled cone, i.e., $\|\hat{\Dbf}^{-1}\xbf\| \le t$, then $(s,\ybf) = (t,\xbf)$ is an optimal solution for  \eqref{eq:proj_diagonal_rescaling} because it is feasible and has the smallest possible objective value $0$. This proves the second case.
    
    If $\|\hat{\Dbf}^{-1}\xbf\| > t$, since $(t,\xbf)$ is not in the rescaled cone, the projected solution $(s,\ybf)$ must lie in the boundary of the cone. In other words, $\|\hat{\Dbf}^{-1}\ybf\| = s$. If $t = 0$, then the optimality conditions \eqref{eq:kkt_socp_1_2_3} become 
    \begin{equation}\label{eq:kkt_sys_fixed_1} 
        (\ybf-\xbf)+2\lambda\hat{\Dbf}^{-2}\ybf=0\ , \ 
        s-2\lambda s=\mu\ , \ 
        s\mu = 0 \ , \
        \norm{\hat{\Dbf}^{-1}\ybf}^{2}=s^{2}, \ 
        s \ge 0\  , \ 
        \lambda \ge 0\ , \  
        \mu \ge 0 \ .
        \end{equation}
    \begin{equation}\label{eq:kkt_sys_fixed_1_0} 
        \ybf = (\Ibf_{n\times n} + \hat{\Dbf}^{-2})^{-1}\xbf\ ,\  s = \left\|(\hat{\Dbf} + \hat{\Dbf}^{-1})^{-1}\xbf\right\| \ ,\  \lambda = \frac{1}{2}\ ,\  \mu = 0 
        \end{equation}
    satisfy the above system. Therefore, when $t = 0$ and $\left\|\hat{\Dbf}^{-1}\xbf\right\| > 0$, \eqref{eq:kkt_sys_fixed_1_0}  gives the solution $s$ and $\ybf$ for \eqref{eq:proj_diagonal_rescaling}. This proves the third case.
    
    If $t > 0$, let $\mu = 0$ and then the optimality conditions \eqref{eq:kkt_socp_1_2_3} become   
    \begin{equation}\label{eq:kkt_sys_fixed_2} 
    (\ybf-\xbf)+2\lambda\hat{\Dbf}^{-2}\ybf=0\ , \ 
    (1-2\lambda) s=t\ , \ 
    \norm{\hat{\Dbf}^{-1}\ybf}^{2}=s^{2}\ , \ 
    s \ge 0 , \ 
    \lambda \ge 0 \ .
    \end{equation}
    If there exist $\ybf, \ s$ and $\lambda \in (0, \frac{1}{2})$ satisfying \eqref{eq:kkt_sys_fixed_2}, then $\ybf$ and $s$ can be computed by the first and second equations:
    \begin{equation}\label{eq:kkt_sys_fixed_2_0} 
    \ybf=(\Ibf+2\lambda\hat{\Dbf}^{-2})^{-1}\xbf \ ,\quad
    s=(1-2\lambda)^{-1}t\geq0\\
    \end{equation}
    Substituting them into the third equation of \eqref{eq:kkt_sys_fixed_2} yields:
    \begin{equation}\label{eq:kkt_sys_fixed_3} 
        f(\lambda) := \|(\hat{\Dbf}+2\lambda\hat{\Dbf}^{-1})^{-1}\xbf\|^2  -(1-2\lambda)^{-2}t^{2} = 0\ .
        \end{equation} 
    Note that there must exist a root $\lambda \in (0,\frac{1}{2})$ of the above equation $f(\lambda) = 0$ because $f$ is continuous on $(0,\frac{1}{2})$, while $\lim_{\lambda\to 0+} f(\lambda) > 0$ (since $\|\hat{\Dbf}^{-1}\xbf\|^2 > t^2 $) and  $\lim_{\lambda\to \frac{1}{2}} f(\lambda) < 0$ (since $t >0$). Therefore, let $\lambda$ be a root of \eqref{eq:kkt_sys_fixed_3} and then \eqref{eq:kkt_sys_fixed_2_0}  gives the solution $\ybf$ and $s$.
    
    Similarly, if $t <0$ and $\|\hat{\Dbf}\xbf\| > -t$ (the case $\|\hat{\Dbf}\xbf\| \le -t$ has been discussed in the beginning), then there exists $\lambda > \frac{1}{2}$ that is a root of $f(\lambda) = 0$. The existence of such a root is due to $\|\hat{\Dbf}\xbf\| > -t$.
    Then \eqref{eq:kkt_sys_fixed_2_0}  gives the solution $\ybf$ and $s$ based on the root $\lambda$. Now the fourth case is proven.\qedtext
    \end{proof}

\begingroup
\begin{proposition}\label{prop:soc_bisection}
Let $\xbf\in \mbb R^n$ and $c_i:=\hat d_{i+1}^{-2}$. In the fourth case of Theorem~\ref{thm:soc_proj}, equation~\eqref{eq:bisection-soc} is equivalent to
\begin{equation}\label{eq:soc_monotone}
\begin{aligned}
F_{+}(u)&:=\sum_{i=1}^{n}c_i x_i^2\left(\frac{1-u}{1+c_i u}\right)^2-t^2=0, &&u=2\lambda &&(t>0),\\
F_{-}(u)&:=\sum_{i=1}^{n}c_i x_i^2\left(\frac{u}{c_i+1-u}\right)^2-t^2=0, &&u=1-(2\lambda)^{-1} &&(t<0).
\end{aligned}
\end{equation}
On $(0,1)$, $F_{+}$ is strictly decreasing and $F_{-}$ is strictly increasing. Each has a unique root, computable by bisection in $O(n\log(1/\epsilon_u))$ operations to absolute accuracy $\epsilon_u$.
\end{proposition}

\begin{proof}
The sign condition in~\eqref{eq:bisection-soc} places both transformations in $(0,1)$. Monotonicity follows from
\[
\brbra{\frac{1-u}{1+c_i u}}^{\prime}=-\frac{1+c_i}{(1+c_i u)^2}<0,
\qquad
\brbra{\frac{u}{c_i+1-u}}^{\prime}=\frac{1+c_i}{(c_i+1-u)^2}>0.
\]
Both ratios are positive and $\xbf\neq\zerobf$ in the fourth case, so the monotonicity is strict. Moreover,
\[
F_{+}(0)=\norm{\hat{\Dbf}^{-1}\xbf}^2-t^2>0>F_{+}(1)=-t^2,
\qquad
F_{-}(0)=-t^2<0<F_{-}(1)=\norm{\hat{\Dbf}\xbf}^2-t^2.
\]
The strict inequalities hold because $(t,\xbf)\notin\Dbf\mcal K_{\soc}^{n+1}$ and $-(t,\xbf)\notin(\Dbf\mcal K_{\soc}^{n+1})^*$. Thus each of $F_{+}$ and $F_{-}$ has a unique root. Each evaluation costs $O(n)$, and bisection on $[0,1]$ takes $O(\log(1/\epsilon_u))$ iterations.{\qedtext}
\end{proof}
\endgroup

\subsection{Projection onto rescaled exponential cone\label{app:proj:exp}} 
To compute the projection of $\vbf_0$ onto the diagonally rescaled exponential cone $\Dbf\mcal K_{\exp}$ and its dual cone $\Dbf^{-1}\mcal K_{\exp}^*$, we first recall the definitions of $\mcal K_{\exp}$ and $\mcal K_{\exp}^*$:
\begin{equation}\label{eq:exp_cone}
\mcal K_{\exp}=\bcbra{(r,s,t)\in\mbb R^{3}\mid s > 0,\ t \ge s\cdot\exp\left(\frac{r}{s}\right)}\bigcup \bcbra{(r,s,t)\in\mbb R^{3}\mid s=0,\ t\geq0,\ r\leq0}  
\end{equation}
\begin{equation}\label{eq:exp_cone_dual}
\mcal K_{\exp}^{*}= \bcbra{(r,s,t)\in\mbb R^{3}\mid r <0, \ e\cdot t \ge -r \cdot \exp\left(\frac{s}{r}\right)} \bigcup \bcbra{(r,s,t)\in\mbb R^{3}\mid r=0,\ s\ge 0,\ t\geq0} 
\end{equation}
and then we present the formal statement and proof of Theorem~\ref{thm:exp_proj_simplified}. 
To compute the projection onto the rescaled cone, we need to use the following Moreau's decomposition theorem \citep{moreau1962decomposition}:
\begin{lemma}\label{lem:moreau_decomp}
Let $\mcal K$ be a closed convex cone in $\mbb R^n$ and $\mcal K^*$ be its dual cone. For any $\vbf_p$ and $\vbf_d \in \mbb R^n$, the following statements are equivalent. (a) $\vbf_p \in\mathcal{K}$,\, $\vbf_d \in-\mathcal{K}^*$ and $\langle \vbf_p,\vbf_d \rangle=0$. (b) $\vbf_p=\proj_{\mathcal{K}}(\vbf_p+\vbf_d)$ and $\vbf_d=\proj_{-\mathcal{K}^*}(\vbf_p+\vbf_d)$. 
\end{lemma}
Therefore, the projection problem essentially reduces to computing $\vbf_p$ and $\vbf_{d}$ such that
\begin{equation}\label{eq:diagonal_moreau}
\vbf_{0}=\vbf_{\textrm{p}}+\vbf_{\textrm{d}}\ ,\ \vbf_{\textrm{p}}\in \Dbf\mcal K_{\exp}\ ,\ \vbf_{\textrm{d}}\in -\Dbf^{-1}\mcal K_{\exp}^{*}\ ,\ \langle \vbf_{\textrm{p}},\vbf_{\textrm{d}} \rangle=0\ .
\end{equation}
Since the projection of $\vbf_0$ onto the rescaled cone is unique, the corresponding $\vbf_p$ and $\vbf_d$ satisfying \eqref{eq:diagonal_moreau} are also unique. Now, we are ready to present the formal statement of Theorem~\ref{thm:exp_proj_simplified}:

\begin{theorem}\label{thm:Moreau_sys_exp}
For any $\vbf_{0}=(r_{0},s_{0},t_{0})\in\mbb R^{3}$, the corresponding $\vbf_p$ and $\vbf_d$ that satisfy \eqref{eq:diagonal_moreau} can be computed as follows:
\begin{enumerate}
\item \label{enu:item1}If $\vbf_{0}\in\Dbf\mcal K_{\exp}$, then $\vbfp=\vbf_{0}$
and $\vbfd=\zerobf_{3\times1}$.
\item \label{enu:item2}If $\vbf_{0}\in-\Dbf^{-1}\mcal K_{\exp}^{*}$, then
$\vbfp=\zerobf_{3\times1}$ and $\vbfd=\vbf_{0}$.
\item \label{enu:item3}If $r_{0}\leq0$ and $s_{0}\leq0$, then $\vbfp=\brbra{r_{0},0,t_{0}^{+}}$
and $\vbfd=\brbra{0,s_{0},-t_{0}^{-}}$. 
\item \label{enu:item4}Otherwise,  
solve the following root-finding problem:
\begin{equation}
h\brbra{\rho}:=d_{t}\exp(\rho)s_{\textrm{p}}(\rho)-d_{t}^{-1}\exp(-\rho)r_{\textrm{d}}\brbra{\rho}-t_{0}=0\ ,\label{eq:h_rho}
\end{equation}
\begin{equation}
s_{\textrm{p}}\brbra{\rho}=\frac{s_{0}\dr^{-2}\ds-\dr^{-1}r_{0}(1-\rho)}{\ds^{2}\dr^{-2}+\rho(\rho-1)}\ \text{and}\ 
r_{\textrm{d}}(\rho)=
\frac{\ds^{2}\dr^{-1}\left(r_{0}-\ds^{-1}\dr\rho s_{0}\right)}
{\ds^{2}\dr^{-2}+\rho(\rho-1)}.
\label{eq:sp_rd}
\end{equation} 
for $\rho$ in interval $(l_{\rho},u_{\rho})$, where  
\begin{subequations}\label{eq:lrho_urho}
\begin{numcases}{a_{3}=\frac{r_{0}\ds}{s_{0}\dr}\ ,\ a_{4}=1-\frac{s_{0}\ds}{r_{0}\dr}\ \text{ and }\ (l_\rho,u_\rho)=}
  (\min\{a_{3},a_{4}\},\max\{a_{3},a_{4}\}), & $r_{0}>0,\ s_{0}>0$ \label{eq:lrho_urho:minmax} \\
  (-\infty,a_{3}), & $r_{0}\leq 0,\ s_{0}>0$ \label{eq:lrho_urho:a3} \\
  (a_{4},+\infty), & $r_{0}>0,\ s_{0}\leq 0$ . \label{eq:lrho_urho:a4}
\end{numcases}
\end{subequations}
{Furthermore, $h(\rho)$ is monotone on the interval $(l_{\rho}, u_{\rho})$.} 
 Then the solution $\brbra{\vbfp,\vbfd}$ to system~\eqref{eq:diagonal_moreau}
is given by
\begin{equation}\label{eq:vpvd_defn}
\vbf_{\textrm{p}}=\brbra{\dr\rho,\ds,\dt\exp(\rho)}s_{\textrm{p}}\brbra{\rho}\ \text{and}\ \vbf_{\textrm{d}}=\brbra{d_{r}^{-1},(1-\rho)d_{s}^{-1},-d_{t}^{-1}\exp(-\rho)}r_{\textrm{d}}\brbra{\rho}\ .
\end{equation}
{If $r_{0}>0$, $s_{0}>0$ and $a_{3}=a_{4}$, then $(l_{\rho},u_{\rho})$ is empty, and~\eqref{eq:vpvd_defn} holds with $\rho=a_{3}$ and with $s_{\textrm{p}}(\rho)$, $r_{\textrm{d}}(\rho)$ replaced by
\[
s_{\textrm{p}}=\frac{t_{0}+\dr d_{t}^{-1}r_{0}e^{-a_{3}}}{d_{t}e^{a_{3}}+\dr^{2}d_{t}^{-1}a_{3}e^{-a_{3}}}\quad\text{and}\quad r_{\textrm{d}}=\dr\brbra{r_{0}-\dr a_{3}s_{\textrm{p}}}.
\]}
\end{enumerate}
\end{theorem}
 
Theorem~\ref{thm:Moreau_sys_exp} provides the  solution to~\eqref{eq:diagonal_moreau}. 
The first three cases of Theorem~\ref{thm:Moreau_sys_exp} are simple cases. In practice, we first check the first three cases, followed by solving the root-finding problem of the nonlinear equation \eqref{eq:h_rho}.
Our strategy is built upon the projection algorithm for the unscaled exponential cone \citep{friberg2023projection}, but the scaling factors can make the scalar equation singular, so the interval that contains the root has to be determined anew (Lemmas~\ref{lem:a1234cond} and~\ref{lem:l_rho_u_rho_sp_rd}). 

\begin{remark}
Projecting onto the rescaled dual exponential cone ($\proj_{\Dbf\mcal K_{\exp}^{*}}(\vbf_{0})$)
is equivalent to the following problem:
\begin{equation}
\vbf_{\textrm{d}}=\argmin_{\vbf}\   \norm{\vbf-\vbf_{0}}^{2}\ \ 
\st\   \vbf\in \Dbf\mcal K_{\exp}^{*}\ . 
\label{eq:proj_exp}
\end{equation}
This task can be completed by first projecting onto the primal cone, specifically, $\vbf_{\textrm{p}}=\proj_{\Dbf^{-1}\mcal K_{\exp}}(-\vbf_0)$, and the solution can then be recovered via  $\vbf_{\textrm{d}}=\vbf_0 + \vbf_{\textrm{p}}$.
\end{remark}

Before proving Theorem~\ref{thm:Moreau_sys_exp}, we first present
two preparatory lemmas.
\begin{lemma}
\label{lem:a1234cond}Let $\nu\in (0,\frac{1}{2}],\zeta>0$ and
denote 
\begin{equation}
a_{1}=\brbra{1-\sqrt{1-4\nu^{2}}}/{2},a_{2}=\brbra{1+\sqrt{1-4\nu^{2}}}/{2},a_{3}=\zeta\nu,a_{4}=1-{\nu}/{\zeta},I=\brbra{\min\bcbra{a_{3},a_{4}},\max\bcbra{a_{3},a_{4}}}.
\end{equation}
If $a_{3}>a_{4}$,
then   $I\subseteq (-\infty , a_1) \cup (a_2,\infty )$. 
If $a_{4}>a_{3}$, then $I\subseteq\brbra{a_{1},a_{2}}$. 
\end{lemma}

\begin{proof}
Recall $a_{1}+a_{2}=1$ and $a_{1}a_{2}=\nu^{2}$. Consider $a_{4}-a_{3}\;=\;1-\frac{\nu}{\zeta}-\zeta\nu\;=\;-\frac{1}{\zeta}\bigl(\nu\zeta^{2}-\zeta+\nu\bigr).$
The quadratic function $q(\zeta):=\nu\zeta^{2}-\zeta+\nu$ has two positive roots $
\zeta_{1}:=\frac{a_{1}}{\nu}$ and $\zeta_{2}:=\frac{a_{2}}{\nu}$,
since $a_{1}+a_{2}=1$ and $a_{1}a_{2}=\nu^{2}$. {Because $\zeta>0$,}  $a_{3}>a_{4}$ is equivalent to $\zeta\in(0,\zeta_{1})\cup(\zeta_{2},\infty)$ and $a_{4}>a_{3}$ is equivalent to $\zeta\in(\zeta_{1},\zeta_{2})$.

\textit{Case 1: $a_{3}>a_{4}$.}
If $\zeta\le \zeta_{1}$, then $a_{3}=\zeta\nu\le a_{1}$ and
$
a_{4}=1-\frac{\nu}{\zeta}\;<\;1-\frac{\nu}{\zeta_{1}}
\;=\;1-\frac{\nu}{a_{1}/\nu}
\;=\;1-\frac{\nu^{2}}{a_{1}}
\;=\;1-(1-a_{1})\;=\;a_{1},
$
so $I=\brbra{a_{4},a_{3}}\subset(-\infty,a_{1})$.  
If $\zeta\ge \zeta_{2}$, then $a_{3}=\zeta\nu\ge a_{2}$ and
$
a_{4}=1-\frac{\nu}{\zeta}\;\ge\;1-\frac{\nu}{\zeta_{2}}
\;=\;1-\frac{\nu}{a_{2}/\nu}
\;=\;1-\frac{\nu^{2}}{a_{2}}
\;=\;1-a_{1}\;=\;a_{2},
$
so $I=\brbra{a_{4},a_{3}}\subset[a_{2},\infty)$.  
In either subcase $I\nsubseteq\brbra{a_{1},a_{2}}$, and in particular $a_{1},a_{2}\notin I$.

\textit{Case 2: $a_{4}>a_{3}$.}
Here $\zeta\in(\zeta_{1},\zeta_{2})$, so
$
a_{1}\;<\;\zeta\nu=a_{3}$ and $a_{4}=1-\frac{\nu}{\zeta}\;<\;1-\frac{\nu}{\zeta_{2}}
=1-{a_{1}}=a_{2}.
$
Thus $a_{1}<a_{3}<a_{4}<a_{2}$, hence $I=\brbra{a_{3},a_{4}}\subset (a_{1},a_{2})$
(and therefore $a_{1},a_{2}\notin I$). When $\nu=\tfrac12$, we have $a_{1}=a_{2}=\tfrac12$, so the case $a_{4}>a_{3}$ does not occur.\qedtext
\end{proof}

\begin{lemma}
\label{lem:l_rho_u_rho_sp_rd}
Suppose $r_{0}>0$ or $s_{0}>0$ and $(l_{\rho},u_{\rho})\neq\varnothing$. 
Then for every $\rho\in(l_{\rho},u_{\rho})$ (defined in~\eqref{eq:lrho_urho}), 
the quantities $s_{\textrm{p}}(\rho)$ and $r_{\textrm{d}}(\rho)$ 
(defined in~\eqref{eq:sp_rd}) satisfy $s_{\textrm{p}}(\rho)>0$ and $r_{\textrm{d}}(\rho)>0$.
\end{lemma}

\begin{proof}[Proof]
For simplicity, set $\tau:=d_{s}/d_{r}>0$. Then $
a_{3}=\frac{r_{0}}{s_{0}}\tau$ and $a_{4}=1-\frac{s_{0}}{r_{0}}\tau$,
and introduce
\begin{equation}\label{eq:tilde_defs}
\tilde d(\rho)=(\rho-\tfrac12)^2+(\tau^2-\tfrac14),\quad
\tilde s(\rho)=s_{0}\tau-r_{0}+r_{0}\rho,\quad
\tilde r(\rho)=\tau r_{0}-\rho s_{0}.
\end{equation}
A direct simplification of~\eqref{eq:sp_rd} gives $s_{\textrm{p}}(\rho)=d_{r}^{-1}\,\frac{\tilde s(\rho)}{\tilde d(\rho)}$ and $r_{\textrm{d}}(\rho)=d_{s}\,\frac{\tilde r(\rho)}{\tilde d(\rho)}$.
Since $d_{r}^{-1}$ and $d_{s}$ are positive, it suffices to show that 
for $\rho\in(l_{\rho},u_{\rho})$, the fractions $\tilde s(\rho)/\tilde d(\rho)$ 
and $\tilde r(\rho)/\tilde d(\rho)$ are positive.
 
\noindent\textit{Case 1. $r_{0}=0$, $s_{0}>0$.}  
Here $\tilde s(\rho)=s_{0}\tau>0$ and $\tilde r(\rho)=-\rho s_{0}$. 
For $\rho<a_{3}=0$ we have $\tilde r(\rho)>0$, and since $\rho<0$,
\(\tilde d(\rho)=(\rho-\tfrac12)^2+(\tau^2-\tfrac14)>\tau^2-1/4+1/4=\tau^2>0\).  
Thus all signs match and positivity of the fractions $\tilde s(\rho)/\tilde d(\rho)$ 
and $\tilde r(\rho)/\tilde d(\rho)$ holds.
 
\noindent\textit{Case 2. $r_{0}>0$, $s_{0}=0$.}  
Here $\tilde r(\rho)=\tau r_{0}>0$ and $\tilde s(\rho)=r_{0}(\rho-1)$. 
For $\rho>a_{4}=1$, we have $\tilde s(\rho)>0$. Also $\tilde d(\rho)>0$ for $\rho>1$ since $(\rho-\tfrac12)^2\ge 1/4$.  
Thus positivity holds.
 
\noindent\textit{Case 3. $r_{0}<0$, $s_{0}>0$.}  
For $\rho<a_{3}=(r_{0}/s_{0})\tau<0$, we have $
\tilde r(\rho)=\tau r_{0}-\rho s_{0}>\tau r_{0}-a_{3}s_{0}=0,$
and also $\tilde s(\rho)=s_{0}\tau-r_{0}+r_{0}\rho> s_{0}\tau-r_{0}>0$ since $r_{0}<0<s_{0}$.  
Furthermore $\rho<0$ implies $\tilde d(\rho)>0$. Thus positivity holds.

\noindent\textit{Case 4. $r_{0}>0$, $s_{0}<0$.}  
For $\rho>a_{4}=1-(s_{0}/r_{0})\tau>1$, we compute $
\tilde s(\rho)=s_{0}\tau-r_{0}+r_{0}\rho > s_{0}\tau-r_{0}+r_{0}(1-(s_{0}/r_{0})\tau)=0,$
and
$
\tilde r(\rho)=\tau r_{0}-\rho s_{0}>\tau r_{0}-s_{0}>0,
$
since $s_{0}<0<r_0$. Also $\rho>1$ guarantees $\tilde d(\rho)>0$. Thus positivity holds.

\noindent\textit{Case 5. $r_{0}>0$, $s_{0}>0$.}  This is the most complicated case, and we distinguish subcases according to the value of $\tau$. 

  If $\tau > 1/2$, then $\tilde{d}(\rho) >0$.   Moreover,
$a_{3}-a_{4}=\tau\!\left(\frac{r_{0}}{s_{0}}+\frac{s_{0}}{r_{0}}\right)-1 \;\ge\;2\tau-1\;>\;0,$
so $(l_{\rho},u_{\rho})=(a_{4},a_{3})$. Hence $\tilde s(\rho)>0$, $\tilde r(\rho)>0$, and $\tilde d(\rho)>0$ on $(l_{\rho},u_{\rho})$, which gives $s_{p}(\rho)>0$ and $r_{d}(\rho)>0$.
 If $\tau = 1/2$, then  $\tilde d(\rho)\ge0$, with equality only at $\rho=\tfrac12$.  
By Lemma~\ref{lem:a1234cond} (with $\nu=\tfrac12$, $\zeta=r_{0}/s_{0}$), if $a_{3}>a_{4}$ then $(a_{4},a_{3})\subset (-\infty,\tfrac12)\cup(\tfrac12,\infty)$, so $\tfrac12\notin(l_{\rho},u_{\rho})$ and $\tilde d(\rho)>0$ on $(l_{\rho},u_{\rho})$.
If $a_{4}\ge a_{3}$, then $(l_{\rho},u_{\rho})$ would be empty, contradicting the assumptions.  
Thus we must be in the former situation, where again $\tilde s(\rho),\tilde r(\rho)>0$ and positivity holds.
  Finally, if $\tau \in (0, \frac{1}{2})$, let $a_{1}=\frac{1-\sqrt{1-4\tau^{2}}}{2}$ and $a_{2}=\frac{1+\sqrt{1-4\tau^{2}}}{2}$,
so $\tilde d(\rho)<0$ exactly for $\rho\in(a_{1},a_{2})$ and $\tilde d(\rho)>0$ otherwise. By Lemma~\ref{lem:a1234cond}, if $a_{3}>a_{4}$, then $(l_{\rho},u_{\rho})=(a_{4},a_{3})\subset (-\infty,a_{1})\cup(a_{2},\infty)$, where $\tilde d(\rho)>0$. On this interval $\tilde s(\rho),\tilde r(\rho)>0$, so the fractions $\tilde s(\rho)/\tilde d(\rho)$ 
and $\tilde r(\rho)/\tilde d(\rho)$ are positive. Otherwise, if $a_{4}>a_{3}$, then $(l_{\rho},u_{\rho})=(a_{3},a_{4})\subset (a_{1},a_{2})$, where $\tilde d(\rho)<0$. On this interval $\tilde s(\rho),\tilde r(\rho)<0$, so the fractions $\tilde s(\rho)/\tilde d(\rho)$ 
and $\tilde r(\rho)/\tilde d(\rho)$  are again positive.

\smallskip
In all subcases we obtain $\tilde s(\rho)/\tilde d(\rho)>0$ and $\tilde r(\rho)/\tilde d(\rho)>0$ for $\rho\in(l_{\rho},u_{\rho})$.  
Since $d_{r}^{-1}$ and $d_{s}$ are positive, this proves $s_{p}(\rho)>0$ and $r_{d}(\rho)>0$.\qedtext
\end{proof}

Now, we are ready to prove Theorem~\ref{thm:Moreau_sys_exp}.

\begin{proof}[Proof of Theorem~\ref{thm:Moreau_sys_exp}]
By Lemma~\ref{lem:moreau_decomp}, the desired pair $(\vbfp,\vbfd)$ is the unique solution of \eqref{eq:diagonal_moreau}.
We use the convention $x^{+}=\max\{x,0\}$ and $x^{-}=-\min\{x,0\}$, so $x=x^{+}-x^{-}$ and $x^{+}x^{-}=0$.

\medskip\noindent{Cases \ref{enu:item1}--\ref{enu:item2}.}
These follow directly from Lemma~\ref{lem:moreau_decomp} and uniqueness of projection.

\medskip\noindent{Case \ref{enu:item3} ($r_0\le 0$ and $s_0\le 0$).}
Set $\vbfp=(r_0,\,0,\,t_0^{+})$ and $\vbfd=(0,\,s_0,\,-t_0^{-}).$
Then $\vbfp+\vbfd=\vbf_0$ because $t_0=t_0^{+}-t_0^{-}$. Moreover $
\Dbf^{-1}\vbfp=(r_0/d_r,\,0,\,t_0^{+}/d_t)\in\mcal K_{\exp}$
(the $s=0$ branch with $r\le 0$ and $t\ge 0$), and $
-\Dbf\,\vbfd=(0,\,-d_s s_0,\,d_t t_0^{-})\in\mcal K_{\exp}^{*}$
(the $r=0$ branch with $s\ge 0$ and $t\ge 0$, since $s_0\le 0$ and $-t_0^{-}\le 0$). Finally, $
\langle \vbfp,\vbfd\rangle=r_0\cdot 0+0\cdot s_0-t_0^{+}\,t_0^{-}=0,$
because $t_0^{+}t_0^{-}=0$. Hence \eqref{eq:diagonal_moreau} holds.

\medskip\noindent{Case \ref{enu:item4} (general case).}
Here $\vbf_0\notin \Dbf\Kcal_{\exp}$ and $\vbf_0\notin -\,\Dbf^{-1}\Kcal_{\exp}^{*}$, and at least one of $r_0>0$ or $s_0>0$ holds (the complement of Case~\ref{enu:item3}). We show that the formulas \eqref{eq:sp_rd} and \eqref{eq:vpvd_defn} with $\rho\in(l_\rho,u_\rho)$ yield the unique Moreau pair.

\emph{Step 1 (membership and orthogonality).}
For any $\rho\in\mbb R$ and positive scalars $s_{\mathrm p},r_{\mathrm d}$ define $
\vbfp=\bigl(d_r\rho,\,d_s,\,d_t e^{\rho}\bigr)\,s_{\mathrm p}$ and $
\vbfd=\bigl(d_r^{-1},\,(1-\rho)d_s^{-1},\,-d_t^{-1} e^{-\rho}\bigr)\,r_{\mathrm d}.$
Then
\(
\Dbf^{-1}\vbfp=s_{\mathrm p}\,(\rho,1,e^{\rho})\in\Kcal_{\exp}
\)
and
\(
-\Dbf\,\vbfd=r_{\mathrm d}\,(-1,\rho-1,e^{-\rho})\in\Kcal_{\exp}^{*}
\)
whenever $s_{\mathrm p}>0$ and $r_{\mathrm d}>0$, each with equality in the defining inequalities. Moreover, $
\langle \vbfp,\vbfd\rangle=s_{\mathrm p}r_{\mathrm d}\,\bigl(\rho+(1-\rho)-1\bigr)=0,$
so orthogonality is automatic, independently of $\rho$, $s_{\mathrm p}$, and $r_{\mathrm d}$. By Lemma~\ref{lem:l_rho_u_rho_sp_rd}, for every $\rho\in(l_\rho,u_\rho)$ we have $s_{\mathrm p}(\rho)>0$ and $r_{\mathrm d}(\rho)>0$.  Therefore \(\vbfp\in\Dbf\Kcal_{\exp}\) and \(\vbfd\in-\,\Dbf^{-1}\Kcal_{\exp}^{*}\) for all $\rho\in(l_\rho,u_\rho)$, and $\langle \vbfp,\vbfd\rangle=0$.

\emph{Step 2 (matching the $r$- and $s$-coordinates).} A  direct substitution into the $r$- and $s$-components of $\vbfp+\vbfd=\vbf_0$ gives the equations:
$
\underbrace{d_r\rho\,s_{\mathrm p}(\rho)}_{\text{$r$ from }\vbfp}
\;+\;
\underbrace{d_r^{-1}\,r_{\mathrm d}(\rho)}_{\text{$r$ from }\vbfd}
\;=\;r_0,
$ and $
\underbrace{d_s\,s_{\mathrm p}(\rho)}_{\text{$s$ from }\vbfp}
\;+\;
\underbrace{d_s^{-1}(1-\rho)\,r_{\mathrm d}(\rho)}_{\text{$s$ from }\vbfd}
\;=\;s_0.
$
Furthermore, with $s_{\mathrm p}$
 and $r_{\mathrm d}(\rho) $ as in \eqref{eq:sp_rd}, it can be checked that $\vbfp$ and $\vbfd$ match $\vbf_0$ in the $r$- and $s$-coordinates.

\emph{Step 3 (matching the $t$-coordinate).}
The remaining $t$-coordinate equation $\bigl(\vbfp+\vbfd\bigr)_t=t_0$ is precisely equivalent to the scalar equation \eqref{eq:h_rho}.
The only possible singularities of $h(\rho)$ occur when the denominator is zero, i.e.,
\[
d_s^{2}d_r^{-2}+\rho(\rho-1)=\bigl(\rho-\tfrac12\bigr)^2+\bigl((d_s/d_r)^2-\tfrac14\bigr)=0,
\]
i.e., at $\rho=a_1,a_2\in(0,1)$ (defined in the above proof of Lemma~\ref{lem:l_rho_u_rho_sp_rd}) when $d_s/d_r\le\tfrac12$. By the proof of Lemma~\ref{lem:l_rho_u_rho_sp_rd}, the open interval $(l_\rho,u_\rho)$ used in \eqref{eq:lrho_urho}  never contains $a_1$ or $a_2$. Hence $h$ is continuous on $(l_\rho,u_\rho)$.
We then show that $h(\rho)$ attains opposite signs at the endpoints (or limits) of $(l_\rho,u_\rho)$ in each branch of \eqref{eq:lrho_urho}.

\smallskip
\emph{(i) $r_0>0$, $s_0>0$.} Here $(l_\rho,u_\rho)=(\min\{a_3,a_4\},\max\{a_3,a_4\})$. {We first assume $a_3\neq a_4$.} A direct simplification of \eqref{eq:sp_rd} shows $
r_{\mathrm d}(a_3)=0$ and $s_{\mathrm p}(a_4)=0,$
and consequently
\begin{equation}\label{eq:h_at_a3}
h(a_{3}) =\frac{s_{0}\dr^{-1}\ds-r_{0}(1-a_{3})}{\ds^{2}\dr^{-1}+\dr a_{3}(a_{3}-1)}d_{t}\exp(a_{3})-t_{0}=\frac{s_{0}\dr^{-1}\ds-r_{0}(1-\frac{r_{0}\ds}{s_{0}\dr})}{\ds^{2}\dr^{-1}+\dr\frac{r_{0}\ds}{s_{0}\dr}(\frac{r_{0}\ds}{s_{0}\dr}-1)}d_{t}\exp(a_{3})-t_{0}=\frac{s_{0}}{\ds}d_{t}e^{a_3}-t_{0}
\end{equation}
and 
\begin{equation}\label{eq:h_at_a4}
\begin{aligned}
  & h(a_{4})=-\frac{r_{0}-\ds^{-1}\dr\brbra{1-\frac{s_{0}\ds}{r_{0}\dr}}s_{0}}{\dr^{-1}+\ds^{-2}\dr\brbra{1-\frac{s_{0}\ds}{r_{0}\dr}}(-\frac{s_{0}\ds}{r_{0}\dr})}d_{t}^{-1}\exp\brbra{\frac{s_{0}\ds}{r_{0}\dr}-1}-t_{0}\\
 & =\frac{\ds^{-1}\dr\brbra{1-\frac{s_{0}\ds}{r_{0}\dr}}s_{0}-r_{0}}{\dr^{-1}+\ds^{-1}\brbra{1-\frac{s_{0}\ds}{r_{0}\dr}}\brbra{-\frac{s_{0}}{r_{0}}}}d_{t}^{-1}\exp\brbra{\frac{s_{0}\ds}{r_{0}\dr}-1}-t_{0} =-\dr r_{0}d_{t}^{-1}\exp\brbra{\frac{s_{0}\ds}{r_{0}\dr}-1}-t_{0} \ .
\end{aligned} 
\end{equation}
Since $\vbf_0\notin \Dbf\Kcal_{\exp}$ with $s_0>0$, the membership test for $(r_0/d_r,s_0/d_s,t_0/d_t)$ in $\Kcal_{\exp}$ fails, i.e.,
\(
t_0< (s_0/d_s)\,d_t\,e^{a_3},
\)
so $h(a_3)>0$. Likewise, since $\vbf_0\notin -\,\Dbf^{-1}\Kcal_{\exp}^{*}$ with $r_0>0$, the test for $(-d_r r_0,-d_s s_0,-d_t t_0)\in \Kcal_{\exp}^{*}$ fails on the $r<0$ branch, i.e.,
\(
t_0> -\,d_r r_0\,d_t^{-1}\exp\!\bigl(\tfrac{d_s s_0}{d_r r_0}-1\bigr),
\)
so $h(a_4)<0$. Hence $h(\rho)$ changes sign across the nonempty open interval $(l_\rho,u_\rho)$.

{If $a_3=a_4=:a$, then $a(1-a)=a_3(1-a_4)=\ds^2\dr^{-2}$. With $s_{\mathrm p}$ and $r_{\mathrm d}$ as in item~\ref{enu:item4} of Theorem~\ref{thm:Moreau_sys_exp}, the $r$- and $t$-coordinates of $\vbfp+\vbfd=\vbf_0$ hold by construction, and the $s$-coordinate equals $(1-a)\dr\ds^{-1}r_0=s_0$. The two membership tests above give $t_0>-\dr r_0d_t^{-1}e^{-a}$ and $t_0<(s_0/\ds)\,d_t\,e^{a}$, which are equivalent to $s_{\mathrm p}>0$ and $r_{\mathrm d}>0$, respectively. Step~1 and Lemma~\ref{lem:moreau_decomp} then show that $(\vbfp,\vbfd)$ solves~\eqref{eq:diagonal_moreau}.}

\smallskip
\emph{(ii) $r_0\le 0$, $s_0>0$.} Here $(l_\rho,u_\rho)=(-\infty,a_3)$ (with $a_3\le 0$). As $\rho\to -\infty$, the term $-d_t^{-1}e^{-\rho}r_{\mathrm d}(\rho)$ dominates with negative sign and $h(\rho)\to -\infty$. At $\rho=a_3$, the same calculation as in \eqref{eq:h_at_a3} yields $h(a_3)=(s_0/d_s)\,d_t\,e^{a_3}-t_0>0$ because $\vbf_0\notin \Dbf\Kcal_{\exp}$ with $s_0>0$. Thus a sign change occurs on $(-\infty,a_3)$.

\smallskip
\emph{(iii) $r_0>0$, $s_0\le 0$.} Here $(l_\rho,u_\rho)=(a_4,\infty)$ (with $a_4\ge 1$). As $\rho\to +\infty$, the term $d_t e^{\rho}s_{\mathrm p}(\rho)$ dominates positively and $h(\rho)\to +\infty$. At $\rho=a_4$, the same dual-side evaluation as in \eqref{eq:h_at_a4} yields $h(a_4)<0$. Hence a sign change occurs on $(a_4,\infty)$.

{We claim that $h$ is strictly monotone on the interval $(l_{\rho}, u_{\rho})$. Let $\delta(\rho):=\ds^2\dr^{-2}+\rho(\rho-1)$. Differentiating \eqref{eq:h_rho} and \eqref{eq:sp_rd} gives
\begin{equation}\label{eq:h_derivative_sign}
h'(\rho)=\frac{e^{-\rho}}{\dr^2\dt\delta(\rho)}\left[s_{\mathrm p}(\rho)A_{\mathrm p}(\rho)+r_{\mathrm d}(\rho)A_{\mathrm d}(\rho)\right],
\end{equation}
where $A_{\mathrm p}(\rho):=\dr^2\ds^2+\dt^2e^{2\rho}[\dr^2(\rho-1)^2+\ds^2]>0$ and $A_{\mathrm d}(\rho):=\dr^2\rho^2+\ds^2+\dt^2e^{2\rho}>0$. Lemma~\ref{lem:l_rho_u_rho_sp_rd} gives $s_{\mathrm p},r_{\mathrm d}>0$, so $\operatorname{sign}h'(\rho)=\operatorname{sign}\delta(\rho)$. If $r_0,s_0>0$, then $\delta>0$ on $(a_4,a_3)$ when $a_3>a_4$, and $\delta<0$ on $(a_3,a_4)$ when $a_4>a_3$. Hence $h$ is strictly increasing in the first case and strictly decreasing in the second. If $r_0\leq0<s_0$, then $\rho<a_3\leq0$, so $\delta>0$ and $h$ is strictly increasing. If $s_0\leq0<r_0$, then $\rho>a_4\geq1$, so again $\delta>0$ and $h$ is strictly increasing.}

{The endpoint signs and strict monotonicity imply that each branch has a unique $\rho^\star\in(l_\rho,u_\rho)$ satisfying $h(\rho^\star)=0$.}
 
\emph{Step 4.}
With this unique $\rho^\star\in(l_\rho,u_\rho)$ and $s_{\mathrm p}=s_{\mathrm p}(\rho^\star)$, $r_{\mathrm d}=r_{\mathrm d}(\rho^\star)$ from \eqref{eq:sp_rd}, the definitions \eqref{eq:vpvd_defn}  give $\vbfp\in\Dbf\Kcal_{\exp}$, $\vbfd\in-\,\Dbf^{-1}\Kcal_{\exp}^{*}$, $\vbfp+\vbfd=\vbf_0$, and $\langle \vbfp,\vbfd\rangle=0$. By Lemma~\ref{lem:moreau_decomp}, $\vbfp=\proj_{\Dbf\Kcal_{\exp}}(\vbf_0)$ and $\vbfd=\proj_{-\Dbf^{-1}\Kcal_{\exp}^{*}}(\vbf_0)$, completing the proof.\qedtext
\end{proof}

\section{Optimality Termination Criteria}\label{app:solver_criteria}
This section describes the termination criteria of the solvers in our experiments. Generally, these criteria comprise three types of errors: primal infeasibility ($\err_\textrm{p}$), dual infeasibility ($\err_\textrm{d}$), and the duality gap ($\err_{\textrm{gap}}$).

SCS solves the conic program problems in the following formulation: 
\begin{equation}\label{eq:scs_opt}
         \min_{\xbf}\ \  \cbf^\top \xbf \ \ \st \Abf \xbf + \sbf = \bbf,\ \ \sbf\in \mathcal{K},\,\ \ 
         \max_{\ybf}\ \  -\bbf^\top \ybf \ \ \st \Abf^\top \ybf + \cbf = 0, \ \ \ybf\in \mathcal{K}^*.
\end{equation}
If the problem is not in this formulation, we first reformulate it into this formulation and then use SCS to solve it.
All iterates of  SCS satisfy the conic constraints $s\in \mathcal{K}$ and $y\in \mathcal{K}^*$ as the solver does projections onto the cones in each iteration. 
On the formulation \eqref{eq:scs_opt}, SCS determines its termination criteria based on the following residual thresholds:
\begin{footnotesize}
\begin{equation*}
\begin{aligned}
    \err_{\textrm{p}} = \norm{\Abf \xbf+\sbf-\bbf}_{\infty}  \leq \vep_{\textrm{abs}} + \vep_{\textrm{rel}}\cdot & \max\left\{\norm{\Abf \xbf}_{\infty},\norm{\sbf}_{\infty},\norm{\bbf}_{\infty}\right\}, \ \err_{\textrm{d}} = \norm{\Abf^\top \ybf + \cbf}_{\infty} \leq \vep_{\textrm{abs}} + \vep_{\textrm{rel}}\cdot \max\left\{\norm{\Abf^\top \ybf}_{\infty},\norm{\cbf}_{\infty}\right\}, \ \\
    & \err_{\textrm{gap}} = |\cbf^\top \xbf + \bbf^\top \ybf| \leq \vep_{\textrm{abs}} + \vep_{\textrm{rel}}\cdot \max\left\{|\cbf^\top \xbf|,|\bbf^\top \ybf|\right\}.
\end{aligned}
\end{equation*}\end{footnotesize}
In our experiments, we say a solution satisfies the $\eps$ tolerance if  the above conditions all hold with $\eps_{\textrm{rel}} = \eps_{\textrm{abs}} = \eps$.
 
MOSEK uses an interior-point method to solve the conic program problem. It addresses the model below:
\begin{equation*}
\Abf \xbf - \bbf\tau  =0, \,\, \Abf^{\top}\ybf+\sbf-\cbf\tau =0,\ \ 
-\cbf^{\top}\xbf+\bbf^{\top}\ybf-\kappa  =0, \,\, \xbf\in\mathcal{K},\ \ \sbf\in\mathcal{K}^*, \ \ \tau,\kappa\geq0 \ .
\end{equation*} 
A solution is regarded as satisfying the $\eps_{\mathrm{rel}}$ tolerance in MOSEK if the following conditions hold:
\begin{equation*}
    \begin{aligned}
        \err_{\textrm{p}} &= \frac{\norm{\Abf\frac{\xbf}{\tau}-\bbf}_{\infty}}{1+\norm{\bbf}_{\infty}} \leq \vep_{\textrm{rel}}\  ,\ \ 
        \err_{\textrm{d}} = \frac{\norm{\Abf^\top \frac{\ybf}{\tau} + \frac{\sbf}{\tau} - \cbf}_{\infty}}{1 + \norm{\cbf}_{\infty}} \leq \vep_{\textrm{rel}}, \\
        \err_{\textrm{gap}} & = \max\left\{\frac{\xbf^\top \sbf}{\tau^2}, \left|\frac{\cbf^\top \xbf}{\tau} - \frac{\bbf^\top \ybf}{\tau}\right|\right\}\ \Big/ \ \max\left\{1,\frac{\min\left\{|\cbf^\top \xbf|,|\bbf^\top \ybf|\right\}}{\tau}\right\} \leq \vep_{\textrm{rel}}\ .
    \end{aligned}
\end{equation*}

The problem that CuClarabel directly addresses is similar to that of SCS, i.e.,~\eqref{eq:scs_opt}. However, CuClarabel introduces two scalar variables $\tau,\kappa \geq 0$ of the homogeneous embedding and considers the following optimization problem:
\begin{equation*}
        \min \ \ \sbf^\top \ybf + \tau \kappa\ \ 
        \st  \ \  \cbf^\top \xbf + \bbf^\top \ybf = -\kappa\ ,\ \Abf^\top \ybf + \cbf\tau = 0 , \Abf \xbf + \sbf - \bbf \tau = 0,\ \ (\sbf,\ybf,\tau,\kappa)\in \mcal K\times \mcal K^*\times \mbb R_+ \times \mbb R_+ .
\end{equation*}
For the above formulation, a solution is regarded as satisfying the $\eps_{\mathrm{rel}}$ tolerance if the following conditions hold:
\begin{equation*}
    \begin{aligned}
        \err_{\textrm{p}} = \frac{\norm{\Abf\frac{\xbf}{\tau} + \frac{\sbf}{\tau} - \bbf}_{\infty}}{\max\{1,\norm{\bbf}_{\infty}+\norm{\xbf/\tau}_{\infty}+\norm{\sbf/\tau}_{\infty}\}}&\leq \eps_{\mathrm{rel}}\ , \ \err_{\textrm{d}} = \frac{\norm{\Abf^\top \frac{\ybf}{\tau} + \cbf}_{\infty}}{\max\{1,\norm{\cbf}_{\infty}+\norm{\xbf/\tau}_{\infty}+\norm{\ybf/\tau}_{\infty}\}}\leq \eps_{\mathrm{rel}}\\
        \err_{\textrm{gap}} = &\frac{\abs{\cbf^\top \xbf + \bbf^\top \ybf}}{\max\{1,\min\{\abs{\cbf^\top \xbf},\abs{\bbf^\top \ybf}\}\}}\leq \eps_{\mathrm{rel}}\ .
    \end{aligned}
\end{equation*}
Since CuClarabel also uses an interior-point method, all iterates of CuClarabel are within the cones.

COPT uses an interior-point method to solve conic programs, but the definition of tolerance in its optimality criteria is not publicly available.

\begingroup
\section{Default Parameter Settings}\label{app:default_parameters}

Table~\ref{tab:cupdcs_default_parameters} summarizes the default configuration of cuPDCS, which is used when the corresponding options are not specified.
\begin{table}[H]
\caption{{Default cuPDCS configuration.}}
\label{tab:cupdcs_default_parameters}
\centering
\small
\renewcommand{\arraystretch}{1.08}
\normalcolor
\begin{tabularx}{0.96\textwidth}{@{}p{0.25\textwidth}X@{}}
\toprule
Component & Default setting \\
\midrule
Preprocessing
& Ten iterations of Ruiz rescaling in the $\ell_\infty$-norm, followed by one Pock--Chambolle rescaling with exponent~1. \\
Cone rescaling
& Scalar or diagonal rescaling, chosen separately for each cone. \\
Adaptive step size
& Enabled by default, with the step size updated by Function~\ref{alg:adaptive_pdhg}. \\
Primal weight
& Initialized from the data and updated at each restart with smoothing parameter $\theta=0.5$ (Appendix~\ref{sec:weight_adjustment}). \\
Averaging and reflection
& Averaging weighted by the step sizes, the Halpern iteration, and reflection are enabled.\\
Restart
& Restarts based on the normalized duality gap, checked every 2,000 iterations, with both the current and the averaged iterates as candidates. \\
Termination
& Absolute and relative tolerances are both $10^{-6}$, checked every 2,000
  iterations on the original problem scale. \\
GPU projection mapping
& Grid-wise, Block-wise, Warp-wise, or Thread-wise, selected according to the types and dimensions of the cones, with batching for many small cones. \\
\bottomrule
\end{tabularx}
\end{table}
\endgroup

\end{APPENDICES}

\end{document}